\documentclass{amsart}
\usepackage{amssymb}
\usepackage{tikz}
\usepackage{placeins}
\usepackage[hyphens]{url}
\usetikzlibrary{arrows.meta,positioning}
\usepackage[colorlinks=true,linkcolor=blue,citecolor=blue,urlcolor=blue]{hyperref}
\hypersetup{
  pdftitle={The locally quaternion group is a countable B-group},
  pdfauthor={Alex J. Sutherland},
  pdfsubject={MSC2020: Primary 20B07; Secondary 20B15, 05C25, 20E22},
  pdfkeywords={B-group, primitive permutation group, Pr\"ufer group, generalized quaternion group, Cayley graph, orbital graph}}
\newcommand{\Qg}{Q}                       \newcommand{\Dic}{\operatorname{Dic}}
\newcommand{\Cay}{\operatorname{Cay}}
\newcommand{\Aut}{\operatorname{Aut}}
\newcommand{\Sym}{\operatorname{Sym}}
\newcommand{\Alt}{\operatorname{Alt}}

\newcommand{\Fix}{\operatorname{Fix}}
\newcommand{\GL}{\operatorname{GL}}
\newcommand{\Ann}{\operatorname{Ann}}
\newcommand{\Ii}{\mathcal{I}}             \newcommand{\Ea}{\mathcal{E}_A}           \newcommand{\F}{\mathbb{F}}
\newcommand{\Z}{\mathbb{Z}}
\newcommand{\one}{\mathbf{1}}             \newcommand{\pst}[1]{P_{(#1)}}            \newcommand{\sst}[1]{P_{\{#1\}}}          \newcommand{\symd}{\mathbin{\triangle}}
\newcommand{\Opt}{O_2}

\newtheorem{theorem}{Theorem}[section]
\newtheorem{lemma}[theorem]{Lemma}
\newtheorem{proposition}[theorem]{Proposition}
\newtheorem{corollary}[theorem]{Corollary}
\theoremstyle{definition}
\newtheorem{definition}[theorem]{Definition}
\newtheorem{remark}[theorem]{Remark}
\numberwithin{equation}{section}
\begin{document}

\title{The locally quaternion group is a countable $B$-group}

\author{Alex J. Sutherland}
\address{Oregon State University, Corvallis, Oregon, USA}
\email{suthalex@oregonstate.edu}

\subjclass[2020]{Primary 20B07; Secondary 20B15, 05C25, 20E22}
\keywords{$B$-group, primitive permutation group, Pr\"ufer group, generalized quaternion group, Cayley graph, orbital graph}

\begin{abstract}
A group $G$ is a $B$-group if every primitive permutation group containing a
regular subgroup isomorphic to $G$ is doubly transitive. We prove that the
locally quaternion group $\Dic(C_{2^\infty},c)$, the extension of the Pr\"ufer
$2$-group by an element squaring to its involution and inverting it, is a
$B$-group. This gives an example of a countably infinite $B$-group.
This answers Problem~13.30 of the Kourovka Notebook. The proof analyses the automorphism
group of an orbital Cayley graph through the central involution. In the main
case two half-flip automorphisms generate, with their conjugates, an elementary
abelian normal subgroup whose orbits form a nontrivial block system; in the
remaining case, where the pointwise stabilizers of the two cosets of the Pr\"ufer
group have order two, the conjugates of the central translation generate an
elementary abelian normal subgroup, which would be regular if the automorphism group were
primitive; an affine lemma for the Pr\"ufer group excludes this. The proof does not depend on computer verification: its two
finite case analyses are written out in full, and programs that re-check them
are provided as ancillary files.
\end{abstract}

\maketitle

\section{Introduction}

A group $G$ is a \emph{$B$-group} if every primitive permutation group that
contains a regular subgroup isomorphic to $G$ is doubly transitive
\cite[\S0]{CameronJohnson1987}. The notion is due to Wielandt, who defined it for finite groups
\cite[Definition~25.1]{WielandtFPG}. We prove that the locally quaternion group $\Qg$, the extension of the Pr\"ufer
$2$-group by an element squaring to its involution and inverting it, is a
$B$-group (Theorem~\ref{thm:main}); in particular, countably infinite
$B$-groups exist.

Let $A=C_{2^\infty}$ be the Pr\"ufer $2$-group and $c$ its unique involution.
The \emph{locally quaternion group} is
\[
\Qg=\Dic(A,c)=\langle A,b \mid b^2=c,\ b^{-1}ab=a^{-1}\ (a\in A)\rangle .
\]

\begin{theorem}\label{thm:main}
Every primitive permutation group containing a regular subgroup isomorphic to
$\Qg$ is doubly transitive. Thus $\Qg$ is a countably infinite $B$-group.
\end{theorem}

Being a $B$-group does not survive enlarging $A$ by a countably infinite abelian factor: for every
countably infinite abelian group $\mathcal{K}$, the group $\Dic(A\times \mathcal{K},c)\cong \mathcal{K}\rtimes\Qg$, defined by the
presentation above with $A\times \mathcal{K}$ in place of $A$, is not a $B$-group
(Corollary~\ref{cor:dic-product}).
The main result also shows that the square-root-set obstruction of Cameron and Johnson
\cite[Theorem~3.1]{CameronJohnson1987} does not separate $B$-groups from
non-$B$-groups, even among groups of the same shape: $\Qg$ and
$\Dic(C_{2^\infty}\times C_{3^\infty},c)$ both meet the covering condition satisfied by every countably infinite $B$-group, each being covered by two translates
of the square-root set of its central involution $c$, as Cameron observed for the infinite
dicyclic group \cite[\S2.3]{Cameron2007Aspects}, \cite{Cameron2012blog}, \cite[\S6]{Cameron2026},
yet $\Qg$ is a $B$-group (Theorem~\ref{thm:main}) and the second group is not
(Corollary~\ref{cor:dic-product}); Cameron and Johnson had already noted non-$B$-groups that
meet this covering condition \cite[\S4]{CameronJohnson1987}. 

To our knowledge, the locally quaternion group is the first countably infinite $B$-group,
answering a question asked by K.~W.~Johnson in Oxford in 1981--82 \cite[\S6]{Cameron2026},
\cite[p.~231]{CameronJohnson1987} and listed by Cameron as Problem~13.30 of the Kourovka
Notebook \cite{Kourovka}. Every countably infinite $B$-group is covered by finitely many translates of finitely many
non-principal square-root sets $\sqrt{g}=\{x:x^2=g\}$, $g\ne1$: Cameron and Johnson showed that
a countably infinite group with no such covering acts regularly on the random graph
\cite[Theorem~3.1]{CameronJohnson1987}, and since the automorphism group of the random graph is
primitive but not doubly transitive \cite[\S6]{Cameron2026}, such a group is not a $B$-group
\cite[\S3]{CameronJohnson1987}. For $\Qg=A\cup Ab$, where $A$ is the Pr\"ufer $2$-group and $b$ the adjoined element, every
element of the coset $Ab$ squares to $c$, since $(ab)^2=a(ba)b=aa^{-1}b^2=c$
(Section~\ref{s:setup}); so $Ab\subseteq\sqrt{c}$ and $\Qg=\sqrt{c}\cup\sqrt{c}\,b^{-1}$ is covered by two translates of $\sqrt{c}$. The
group $\Dic(C_{2^\infty}\times C_{3^\infty},c)$ is treated in the same way in Section~\ref{s:boundary}, where $\Dic(D,c)$ is defined for an abelian group $D$ with involution $c$.
For $\Qg$, the proof combines rigidity of the half-stabilizers produced by the half-flips, pointwise limits in the permutation topology that produce supported automorphisms, and the exclusion of finite marked permutation models.
Section~\ref{ss:history} records the finite case, the earlier accounts of the question and the
earlier uses of the half-flips.

Theorem~\ref{thm:main} can be stated equivalently in terms of $2$-closures. For a
permutation group $G$ on a set $\Omega$, its \emph{$2$-closure} $G^{(2)}$ is the largest
subgroup of $\Sym(\Omega)$ with the same orbits on $\Omega\times\Omega$, and $G$ is
\emph{$2$-closed} if $G^{(2)}=G$ (see \cite[\S5, Definition~5.3, Theorem~5.4(ii) and Definition~5.21]{Wielandt1969OSU},
where Wielandt defines $k$-closures and $k$-closed groups for actions on finite or infinite
sets). We write $R(\Qg)$ for the right regular representation
of $\Qg$ (Section~\ref{ss:conv}).

\begin{corollary}\label{cor:twoclosure}
If $R(\Qg)\le G\le\Sym(\Qg)$ with $G$ primitive and $2$-closed, then $G=\Sym(\Qg)$.
\end{corollary}

\begin{proof}
By Theorem~\ref{thm:main}, $G$ is doubly transitive, so its orbits on $\Qg\times\Qg$ are
the diagonal and its complement, which every permutation preserves; hence
$G=G^{(2)}=\Sym(\Qg)$.
\end{proof}

Conversely, Corollary~\ref{cor:twoclosure} implies Theorem~\ref{thm:main}: if every
primitive $2$-closed overgroup of $R(\Qg)$ is $\Sym(\Qg)$, then for any primitive
overgroup $G$ of $R(\Qg)$, the group $G^{(2)}\supseteq G$ is primitive and $2$-closed, so
$G^{(2)}=\Sym(\Qg)$ and $G$ has a single orbit on ordered pairs of distinct points.

\subsection{Idea of the proof}

Suppose a primitive group contains a regular copy of $\Qg$ but is not doubly
transitive. Its orbital graph through the pair consisting of the identity and
the central involution $c$ (the graph whose edges are the images of that pair
under the group) is a Cayley graph of $\Qg$ whose full automorphism
group $P$ is again primitive. Its connection set has the form $X\cup Tb$ with
$X,T$ subsets of $A$, and the inverse-closure of the connection set forces
$cT=T$.

\begin{itemize}
\item \emph{Half-flips.} Because $cT=T$, the permutation $t$ that fixes the
  coset $A$ pointwise and multiplies the coset $Ab$ by $c$ is a graph
  automorphism, and so is its partner $v$, which fixes $Ab$ pointwise and
  multiplies $A$ by $c$. Their product is the central translation $\rho$ by
  $c$. We call $A$ and $Ab$ the two \emph{sides}.
\item \emph{Two invariants of the graph.} The period subgroup
  $H=\Pi(T)=\{h\in A:hT=T\}$ of $T$ is a finite cyclic subgroup of $A$ containing $c$. The set $U=X\symd cX$ is
  infinite: otherwise the graph has a nonidentity automorphism of finite
  support, and Wielandt's theorem on finitary permutations contradicts
  primitivity.
  The subgroup $H$ selects the branch: the pointwise stabilizer of a side is cyclic of order $|H|$, so it
  contains a square root of the half-flip fixing that side exactly when $|H|\ge4$ (Corollary~\ref{cor:H}).
  The set $U$ consists of the $u\in A$ for which every vertex $d\in A$ is adjacent to exactly one of $ud$
  and $cud$; that it is infinite gives the rigid sides below, and the synchronization lemma
  (Lemma~\ref{lem:sync}) uses it again.
\item \emph{Rigid sides.} Because $U$ is infinite, the pointwise stabilizer of
  the side $A$ in $P$ consists exactly of the translations of $Ab$ by elements of
  $H$. It is cyclic and its unique involution is $t$. The same holds for every
  conjugate of $A$ under $P$.
\item \emph{The case $|H|\ge4$: the fourth-power bound.} A supported square
  root of $t$ (an automorphism that fixes the side $A$ pointwise and squares to
  $t$) exists, and so does one of every conjugate of $t$, including $v$
  (Corollary~\ref{cor:H}(3)). With it one shows that a product of a conjugate of $t$ and
  an involution of $P$ has finite $2$-power order whenever it has finite order;
  that no point of the fixed side has an infinite orbit under such a product (a
  limit argument in the permutation topology); and, by a second limit argument,
  that all such orbits through the fixed side have $2$-power length. The
  two-mask lemma on two perfect matchings (Lemma~\ref{lem:two-mask}) then leads to the
  fourth-power bound $(xy)^4=1$ for all conjugates $x,y$ of $\rho$
  (Theorem~\ref{thm:fourth-power}).
\item \emph{The case $|H|\ge4$: normal form and the block.} A finite analysis of
  the components of two such matchings, a supported-root enclosure argument and
  the exclusion of one remaining eight-point configuration give a normal form
  for every
  pair; a synchronization lemma (Lemma~\ref{lem:sync}), which uses again that $U$ is
  infinite, then shows that all conjugates of $t$ commute.
  The group $\langle t^P\rangle$ is then elementary abelian and normal, and its
  orbit $\langle t^P\rangle(1)$ contains $1$ and $c$ and lies in $A$. It is a nontrivial block,
  contradicting primitivity.
\item \emph{The case $|H|=2$.} Here the pointwise stabilizers of the sides have
  order two, and distinct pairs $\{z,cz\}$ have distinct neighbourhoods across
  the sides. The robust-pair theorem
  (Theorem~\ref{thm:twelve-words}), on two interleaved $0$--$1$ sequences, gives
  $(xy)^4=1$ for all conjugates $x,y$ of $\rho$ (Theorem~\ref{thm:rs4}); a
  diagonal alignment lemma (Lemma~\ref{lem:diag}), a translation square root (Proposition~\ref{prop:mixed8two}) and a pair-preserving criterion (Lemma~\ref{lem:transfer}) then make
  all conjugates of $\rho$ commute (Theorem~\ref{thm:commute2}).
  The conjugates of $\rho$ generate an
  elementary abelian normal
  subgroup (Theorem~\ref{thm:commute2}), which would be regular if $P$ were primitive, so that $P$ would act
  affinely; the affine lemma for the Pr\"ufer group (Lemma~\ref{lem:affine}) excludes this.
  The opening of Section~\ref{s:two} explains how the robust-pair theorem and the square root of $\rho$ are used.
\end{itemize}

Both branches follow the same pattern: a square root (a supported square root of each conjugate of $t$, or $R(d)$ with $d\in A$ of order four), the bound $(xy)^4=1$, the exclusion of one exceptional configuration, an elementary abelian normal subgroup, and a contradiction with primitivity.

\subsection{Structure of the paper}

Section~\ref{s:setup} sets up the Cayley graph and reduces the theorem to the
case $U$ infinite. Section~\ref{s:H} proves the exact half-stabilizer
statement~(H). Sections~\ref{s:big} and~\ref{s:sync} treat $|H|\ge4$, and
Section~\ref{s:two} treats $|H|=2$. Section~\ref{s:proof} assembles the proof.
Section~\ref{s:boundary}, which does not use Theorem~\ref{thm:main}, extends the
construction of Cameron and Johnson \cite[Proposition~4.1]{CameronJohnson1987} from direct
to semidirect products (Proposition~\ref{prop:product-action}); by
Corollary~\ref{cor:dic-product}, enlarging the abelian part of $\Qg$ by any countably
infinite abelian factor destroys the property, so $\Qg$ sits on the boundary that this
construction draws. The dependencies are shown in Figure~\ref{fig:deps}.

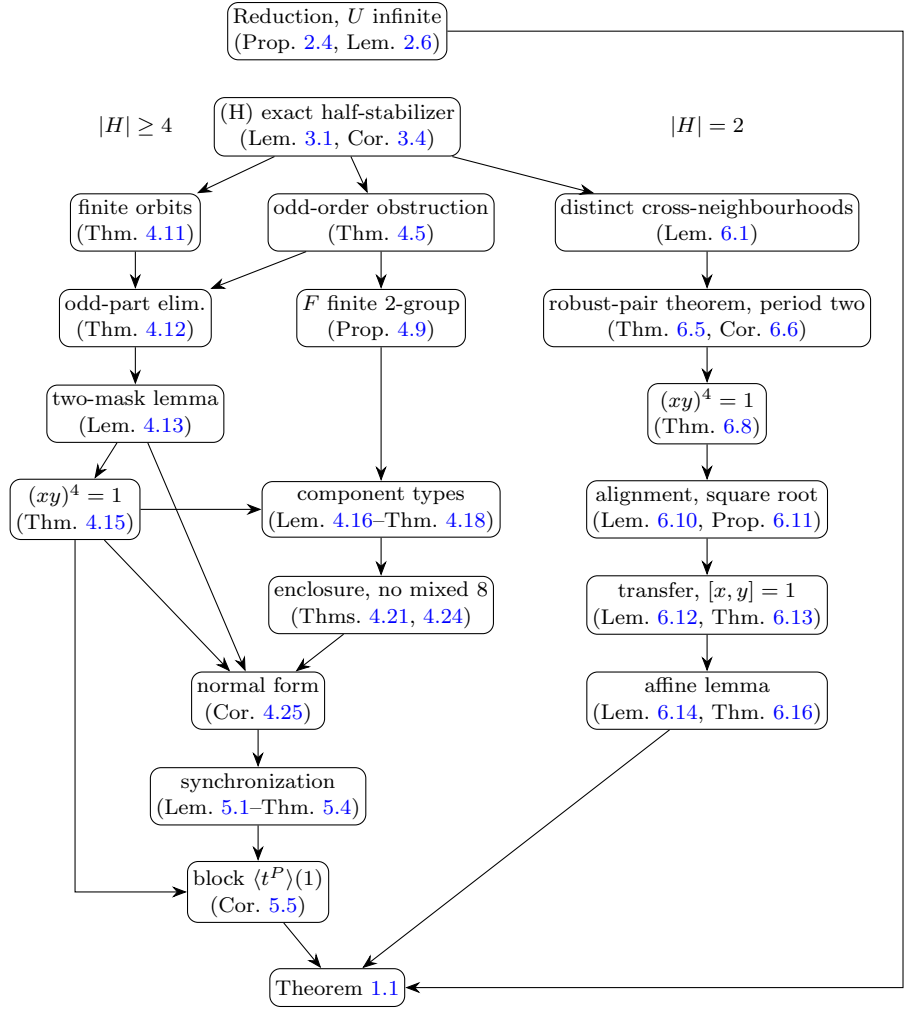
\begin{figure}[ht]
\centering
\begin{tikzpicture}[
  every node/.style={draw, rounded corners, font=\footnotesize, align=center, inner sep=2pt, minimum height=5mm},
  comp/.style={double},
  >={Stealth[length=2mm]}, x=1pt, y=1pt]
\node (red) at (123.3,0.0) {Reduction, $U$ infinite\\(Prop.~\ref{prop:reduction}, Lem.~\ref{lem:Ufinite})};
\node (H) at (123.3,-36.0) {(H) exact half-stabilizer\\(Lem.~\ref{lem:H}, Cor.~\ref{cor:H})};
\node (SR) at (139.8,-72.0) {odd-order obstruction\\(Thm.~\ref{thm:odd-order})};
\node (NR) at (47.3,-72.0) {finite orbits\\(Thm.~\ref{thm:no-rays})};
\node (F) at (139.8,-108.0) {$F$ finite $2$-group\\(Prop.~\ref{prop:two-core})};
\node (OP) at (47.3,-108.0) {odd-part elim.\\(Thm.~\ref{thm:oddprimary})};
\node (TM) at (47.3,-144.0) {two-mask lemma\\(Lem.~\ref{lem:two-mask})};
\node (ID) at (24.6,-180.0) {$(xy)^4=1$\\(Thm.~\ref{thm:fourth-power})};
\node (TY) at (139.8,-180.0) {component types\\(Lem.~\ref{lem:word-fixed-set}--Thm.~\ref{thm:types})};
\node (EN) at (139.8,-216.0) {enclosure, no mixed 8\\(Thms.~\ref{thm:enclosure}, \ref{thm:mixed8})};
\node (NF) at (93.5,-252.0) {normal form\\(Cor.~\ref{cor:normal-form})};
\node (SY) at (93.5,-288.0) {synchronization\\(Lem.~\ref{lem:sync}--Thm.~\ref{thm:pairclosure})};
\node (BL) at (93.5,-324.0) {block $\langle t^P\rangle(1)$\\(Cor.~\ref{cor:block})};
\node (RW) at (262.5,-72.0) {distinct cross-neighbourhoods\\(Lem.~\ref{lem:cross-nbhd})};
\node (WI) at (262.5,-108.0) {robust-pair theorem, period two\\(Thm.~\ref{thm:twelve-words}, Cor.~\ref{cor:period2})};
\node (R4) at (262.5,-144.0) {$(xy)^4=1$\\(Thm.~\ref{thm:rs4})};
\node (AL) at (262.5,-180.0) {alignment, square root\\(Lem.~\ref{lem:diag}, Prop.~\ref{prop:mixed8two})};
\node (CM) at (262.5,-216.0) {transfer, $[x,y]=1$\\(Lem.~\ref{lem:transfer}, Thm.~\ref{thm:commute2})};
\node (AF) at (262.5,-252.0) {affine lemma\\(Lem.~\ref{lem:affine}, Thm.~\ref{thm:two})};
\node (MT) at (123.3,-360.0) {Theorem~\ref{thm:main}};
\node[draw=none] at (47.3,-36.0) {$|H|\ge4$};
\node[draw=none] at (262.5,-36.0) {$|H|=2$};
\draw[->] (H) -- (SR);
\draw[->] (H) -- (NR);
\draw[->] (H) -- (RW);
\draw[->] (SR) -- (F);
\draw[->] (SR) -- (OP);
\draw[->] (NR) -- (OP);
\draw[->] (OP) -- (TM);
\draw[->] (TM) -- (ID);
\draw[->] (F) -- (TY);
\draw[->] (ID) -- (TY);
\draw[->] (TY) -- (EN);
\draw[->] (EN) -- (NF);
\draw[->] (ID) -- (NF);
\draw[->] (TM) -- (NF);
\draw[->] (NF) -- (SY);
\draw[->] (SY) -- (BL);
\draw[->] (ID.south) |- (BL.west);
\draw[->] (RW) -- (WI);
\draw[->] (WI) -- (R4);
\draw[->] (R4) -- (AL);
\draw[->] (AL) -- (CM);
\draw[->] (CM) -- (AF);
\draw[->] (BL) -- (MT);
\draw[->] (AF) -- (MT);
\draw[->] (red.east) -- ++(171.5,0) |- (MT.east);
\end{tikzpicture}
\caption{Logical dependencies. An arrow means ``is used in the proof of''.}
\label{fig:deps}
\end{figure}
\FloatBarrier

\subsection{Historical remarks}\label{ss:history}

For finite groups, $B$-groups go back to Burnside and Schur, as Wielandt
records in the section of his book that collects the families of finite
$B$-groups known in 1964 \cite[\S25]{WielandtFPG}. The finite study led to the theory of Schur rings \cite[\S0]{CameronJohnson1987},
\cite[Chapter~IV]{WielandtFPG}; for example, Wielandt showed that every abelian
group of composite order with at least one cyclic Sylow subgroup is a
$B$-group \cite[Theorem~25.4]{WielandtFPG}, see also \cite{PonomarenkoRyabov2024};
and, after the classification of finite simple groups, Liebeck, Praeger and
Saxl studied the finite primitive groups with a regular subgroup
\cite{LPS2010}. Wielandt records, crediting Scott (1957), that every finite
dicyclic group, with defining relations $x^{2m}=1$, $y^2=x^m$,
$y^{-1}xy=x^{-1}$, is a $B$-group \cite[Theorem~25.6(ii)]{WielandtFPG}. The group $\Qg$ is the union of its finite subgroups $\langle A_n,b\rangle$, of order
$2^{n+1}$, for $n\ge1$, where $A_n$ is the subgroup of order $2^n$ of $A$
(Section~\ref{s:setup}).
Cameron and Johnson proved that no countably infinite abelian group is a
$B$-group \cite[Corollary~4.2]{CameronJohnson1987}, and that no direct product
of two countably infinite groups is one \cite[Proposition~4.1]{CameronJohnson1987}. By the first of these results, being a $B$-group does not pass to unions of chains of
subgroups: the Pr\"ufer $2$-group is the union of its cyclic subgroups of order $2^n$,
$n\ge2$, which are $B$-groups \cite[Theorem~25.2]{WielandtFPG}, but it is not a
$B$-group. 

Wildon records the early history of finite $B$-groups in a survey that includes the
errors in Burnside's proofs \cite[\S6]{Wildon2020}. 

Cameron and Johnson credit G.~Higman with the construction of their Section~1
\cite[p.~231]{CameronJohnson1987}; in Cameron's later accounts Higman's own result has a
stronger hypothesis than their square-root-set theorem, whose condition is a weaker form of
Higman's \cite[\S2.1]{Cameron2007Aspects}, \cite[\S6]{Cameron2026}, and concludes only that
the group is not a $B$-group \cite[\S2.1]{Cameron2007Aspects}. Cameron and Johnson noted that the group
$\langle a,b\mid b^4=1,\ b^{-1}ab=a^{-1}\rangle$ and its direct square are both covered by finitely many
translates of finitely many non-principal square-root sets, which places them outside the hypothesis of
their Theorem~3.1, and that the direct square is not a $B$-group, by their
Proposition~4.1; they wrote that they did not know whether the group itself is a
$B$-group \cite[\S4]{CameronJohnson1987}.\footnote{The 1987 text prints the relation as
$a^{-1}ba=b^{-1}$; with that relation every square-root set is finite, so we follow the form
used in \cite[\S2.3]{Cameron2007Aspects}.} 

Problem~13.30, among the problems from the 13th Issue (1995), reads: ``A group $G$
is called a $B$-group if every primitive permutation group which contains
the regular representation of $G$ is doubly transitive. Are there any
countable $B$-groups?'' \cite{Kourovka}. Countable here means countably infinite, as in Cameron's account of Johnson's
question \cite[\S6]{Cameron2026}. Problem~13.30 is listed without a solution in the 21st edition of the Kourovka
Notebook (2026) \cite{Kourovka}. 

The infinite dicyclic group of Cameron's later accounts is
$\langle x,y\mid y^4=1,\ y^{-1}xy=x^{-1}\rangle$ \cite{Cameron2012blog},
\cite[\S6]{Cameron2026}. Cameron asked in 2007 whether this group, presented there as
$\langle a,b : b^4=1,\ b^{-1}ab=a^{-1}\rangle$, is a $B$-group
\cite[Problem~2.5]{Cameron2007Aspects}. In 2012 he asked ``Is the infinite generalised dicyclic group a $B$-group?'' without defining that group
\cite{Cameron2012blog}. 

The half-flip $t$ is the map $\psi$ of Leemann and de la Salle
\cite[\S1]{LeemannDeLaSalle2022} for $\Qg=\langle A,b\rangle$. For finite groups a map of this kind occurs earlier: Watkins
\cite[proof of Theorem~3]{Watkins1971} uses a nonidentity automorphism $\phi$
with $\phi(x)\in\{x,x^{-1}\}$ for every $x$, which fixes an abelian subgroup of
index two pointwise and inverts every element of the other coset, as $t$ does
with $A$ and $Ab$ since $(db)^{-1}=cdb$ for $d\in A$. He deduces that the finite non-abelian groups with such an automorphism, which his
Theorem~3 characterizes, are in his Class~II \cite[Corollary~3A]{Watkins1971}, a result
he credits also to L.~A.~Nowitz, and that the dicyclic groups, including the generalized
quaternion groups, are in Class~II \cite[Corollary~3B]{Watkins1971}; so none of these
groups acts regularly as the full automorphism group of a graph
\cite[Theorem~1]{Watkins1971}. 

\subsection{Conventions}\label{ss:conv}

Permutations act on the left and products compose from right to left:
$(fg)(x)=f(g(x))$. For permutations $g,h$ we write $g^h=h^{-1}gh$ and
$[g,h]=g^{-1}h^{-1}gh$. For a group $G$ acting on a set and a subset $C$,
$G_{(C)}$ is the pointwise stabilizer and $G_{\{C\}}$ the setwise stabilizer;
$\Fix(g)$ is the set of fixed points of $g$; for an involution $g$, a
\emph{$g$-pair} is a two-point orbit $\{p,g(p)\}$ of $g$. For a subset $Y$ of an
abelian group $D$, its \emph{period subgroup} is $\Pi(Y)=\{h\in D:hY=Y\}$; for
a function $f$ on $D$, $\Pi(f)=\{h\in D:f(hz)=f(z)\text{ for all }z\in D\}$.
For $g\in\Qg$, $R(g)$ is the right
translation $x\mapsto xg$; since products compose from right to left, $R(g)R(h)=R(hg)$ for
$g,h\in\Qg$. The right regular representation is
$R(\Qg)=\{R(g):g\in\Qg\}$. Cayley graphs use right-Cayley adjacency: for an
inverse-closed $S\subseteq\Qg\setminus\{1\}$, the graph $\Cay(\Qg,S)$ has an
edge between $p$ and $q$ if and only if $pq^{-1}\in S$; every $R(g)$ is an
automorphism.

Notation used in more than one section is collected in Appendix~\ref{app:notation},
together with references to its definition. All notation is introduced in the main text as needed.

\section{Setup and reduction}\label{s:setup}

Throughout, $A=C_{2^\infty}$, $c$ is its unique involution,
$\Qg=\langle A,b\mid b^2=c,\ b^{-1}ab=a^{-1}\ (a\in A)\rangle$ and $B=Ab$;
the cosets $A$ and $B$ are the two \emph{sides} of $\Qg$. Every element of $\Qg$ is $a$ or $ab$ for a unique $a\in A$.
Existence follows from the relations below. For uniqueness, $A\times\{0,1\}$ with the product
$(a,i)(a',j)=(a\,a'^{(-1)^i}c^{ij},\,i+j\bmod2)$ is a group (associativity is an identity in the abelian
group $A$, as $c^2=1$), and $a\mapsto(a,0)$, $b\mapsto(1,1)$ respects the defining relations; it sends $a$
to $(a,0)$ and $ab$ to $(a,1)$, so these elements of $\Qg$ are pairwise distinct. Since $b^2=c$ is
central and $b^{-1}ab=a^{-1}$, we have $ba=a^{-1}b$ for $a\in A$ and
$(db)^{-1}=cdb$ for $d\in A$. Write $A=\bigcup_{n\ge0}A_n$, where $A_n$ is cyclic of order $2^n$ and $A_n<A_{n+1}$; as is standard,
$A_n$ is the only subgroup of $A$ of order $2^n$. For $n\ge1$, the set $A_n\cup A_nb$ is a subgroup
of $\Qg$, since $b$ normalizes $A_n$ and $b^2=c\in A_n$. It has order $2^{n+1}$, and these subgroups
exhaust $\Qg$. We record the properties of $A$ used below.

\begin{lemma}[Properties of the Pr\"ufer group]\label{lem:prufer}
\begin{enumerate}
\item Every proper subgroup of $A$ is finite and cyclic. Every nontrivial subgroup of $A$
  contains $c$, and every infinite subset of $A$ generates $A$.
\item $A$ has no nontrivial finite image.
\item If $D$ is a finite subgroup of $A$, then $A/D$ is isomorphic to $A$, so
  (1) and (2) hold for $A/D$.
\end{enumerate}
\end{lemma}

\begin{proof}Let $K\le A$. If the orders of its elements are bounded, $K$ lies in some $A_n$ and equals some $A_j$; otherwise $K$ contains every $A_n$, hence $K=A$; this gives (1). Part (2) holds because $A$ is
divisible, and (3) because the power map with kernel $D$ identifies $A/D$ with $A$.
\end{proof}

\subsection*{The graph}
Let $X,T\subseteq A$ and $S=X\cup Tb$. Write, with the period subgroup $\Pi$ of
Section~\ref{ss:conv},
\[
H=\Pi(T)=\{h\in A:hT=T\},\qquad U=X\symd cX .
\]

\begin{lemma}[Adjacency]\label{lem:adj}
Assume $X=X^{-1}\subseteq A\setminus\{1\}$ and $cT=T$. Then $S=S^{-1}$, and in
$\Gamma=\Cay(\Qg,S)$, for $a,d\in A$, writing $E(p,q)=1$ if $p\sim q$ and $0$
otherwise,
\begin{equation}\label{eq:adj}
E(a,d)=E(ab,db)=\one_X(ad^{-1}),\qquad E(a,db)=\one_T(ad).
\end{equation}
\end{lemma}

\begin{proof}
By $(db)^{-1}=cdb$, $(Tb)^{-1}=(cT)b=Tb$, so $S=S^{-1}$; and \eqref{eq:adj} follows from
$(ab)(db)^{-1}=ad^{-1}$ and $a(db)^{-1}=acd\,b$, with $acd\in T$ if and only if $ad\in cT=T$.
\end{proof}

For a vertex $p$, the set of its neighbours on the side not containing $p$ is its
\emph{cross-neighbourhood}.

Put $P=\Aut\Gamma$, $\rho=R(c)$ and $\Ii=\rho^P$. Define permutations $t,v$ of
$\Qg$, the \emph{half-flips}, by
\begin{equation}\label{eq:tv}
t(a)=a,\quad t(ab)=cab,\qquad v(a)=ca,\quad v(ab)=ab\qquad(a\in A).
\end{equation}

\begin{lemma}[Half-flips]\label{lem:halfflips}
Assume $X=X^{-1}\subseteq A\setminus\{1\}$ and $cT=T$. Then $t,v\in P$ and
\begin{equation}\label{eq:tvprops}
t^2=v^2=1,\quad [t,v]=1,\quad tv=\rho,\quad
\Fix(t)=A,\quad \Fix(v)=B.
\end{equation}
Both sides $A$ and $B$ are nonempty. The right translation $R(b)$ conjugates
$t$ to $v$ and centralizes $\rho$.
\end{lemma}

\begin{proof}
The relations in \eqref{eq:tvprops} hold by \eqref{eq:tv}, since $\rho(x)=xc=cx$.
To see that the half-flips are automorphisms we check the three kinds of pairs
in \eqref{eq:adj}. The map $t$ preserves adjacency: it fixes $A$ pointwise;
$E(cab,cdb)=\one_X(ad^{-1})=E(ab,db)$ because $c$ is central; and
$E(a,cdb)=\one_T(acd)=\one_T(ad)=E(a,db)$ because $cT=T$. The same three
computations apply to $v$. For the last assertion, since $c$ is central,
$R(b)$ commutes with $\rho$, and evaluation on the two sides gives the
conjugation: for $z\in A$ we have $R(b)^{-1}(z)=zb^{-1}=czb$, $t(czb)=zb$ and $R(b)(zb)=zb^2=cz$;
and $R(b)^{-1}(zb)=z$, $t(z)=z$, $R(b)(z)=zb$. Hence $R(b)\,t\,R(b)^{-1}=v$. Since
$R(b)^2=R(b^2)=R(c)=\rho$ and $\rho=tv$ commutes with $v$, also
$t^{R(b)}=R(b)^{-1}tR(b)=R(b)^{-2}vR(b)^2=\rho^{-1}v\rho=v$, so $R(b)$ conjugates $t$ to $v$ in either sense.
\end{proof}

For a permutation group $G$ on a set $\Omega$, an \emph{orbital} is a $G$-orbit
on $\Omega\times\Omega$. It is \emph{self-paired} if it contains $(\beta,\alpha)$
whenever it contains $(\alpha,\beta)$, and then its \emph{undirected orbital
graph} has the edges $\{\alpha,\beta\}$ with $(\alpha,\beta)$ in the orbital.

\begin{proposition}[The orbital graph]\label{prop:reduction}
Let $G\le\Sym(\Qg)$ be primitive, not doubly transitive, and contain $R(\Qg)$.
Let $\Gamma$ be the undirected orbital graph of $G$ containing the pair $(1,c)$.
Then $\Gamma=\Cay(\Qg,X\cup Tb)$ where
\[
X=X^{-1}\subseteq A\setminus\{1\},\quad c\in X,\quad cT=T,\quad
\varnothing\ne T\ne A ,
\]
$\Gamma$ is nonempty and noncomplete, $\Gamma$ and its complement are connected,
$P=\Aut\Gamma$ contains $G$ and is primitive, and $H=\Pi(T)$ is a proper,
finite, cyclic subgroup of $A$ containing $c$.
\end{proposition}

\begin{proof}
The orbital through $(1,c)$ is self-paired, because $R(c)\in G$ maps $(1,c)$ to
$(c,1)$. So its undirected orbital graph $\Gamma$ is defined; it is nonempty,
and it is noncomplete because $G$ is not doubly transitive. Since
$R(\Qg)\le G$, the graph $\Gamma$ is a Cayley graph $\Cay(\Qg,S)$ with
$S=\{g:g\sim1\}$ inverse-closed and $1\notin S$, and $c\in S$. Indeed, $G$ maps the orbital to itself,
so every element of $G$, and its inverse, preserves adjacency in $\Gamma$; thus $G\le\Aut\Gamma=P$, and in
particular every $R(g)$ is an automorphism of $\Gamma$. Applying $R(q^{-1})$ shows that $p\sim q$ if and
only if $pq^{-1}\sim1$, that is, $\Gamma=\Cay(\Qg,S)$. Applying $R(g^{-1})$ to the edge $\{g,1\}$ gives
the edge $\{1,g^{-1}\}$, so $S=S^{-1}$; $1\notin S$ because $\Gamma$ has no loops; and $c\in S$ because $\{c,1\}$ is an edge. Write
$S=X\cup Tb$ with $X,T\subseteq A$. Since $A$ is a subgroup and $B=Ab$ is closed under inversion,
$X=S\cap A$ satisfies $X=X^{-1}\subseteq A\setminus\{1\}$ and $c\in X$; and inverse
closure of $Tb$ gives $cT=T$, because $(Tb)^{-1}=\{cdb:d\in T\}=(cT)b$, as
$(db)^{-1}=cdb$ for $d\in A$.

The connected
components of $\Gamma$, and those of its complement, form $G$-invariant
partitions; the complement has no isolated vertex since $\Gamma$ is noncomplete
and vertex-transitive. By primitivity both graphs are connected. In particular,
if $T=\varnothing$ there is no edge between $A$ and $B$ in $\Gamma$, and if $T=A$
there is none in the complement; hence $\varnothing\ne T\ne A$. Consequently
$H=\Pi(T)$ is a proper subgroup of $A$: it is the stabilizer of $T$ under translation by $A$, hence a
subgroup, and if it were $A$, then $T$, being nonempty, would contain some $d$, and $hd\in T$ for every
$h\in A$ would give $T=A$. So $H$ is finite and cyclic
(Lemma~\ref{lem:prufer}(1)), and it contains $c$ because $cT=T$.

The group $P$
contains $G$, hence is primitive.
\end{proof}

A permutation is \emph{finitary} if it moves only finitely many points, and
$\Alt(\Omega)$ is the group of even finitary permutations of a set $\Omega$. We
use Wielandt's theorem on finitary permutations in the following form.

\begin{theorem}[Wielandt; see {\cite[Theorem~29.2]{Neumann2023}}; see also {\cite[Theorem~3.3D]{dixon_mortimer_1996}} and {\cite[Theorem~6.1]{cameron_1999}}]\label{thm:wielandt}
Let $G$ be primitive on an infinite set $\Omega$. If $G$ contains a nonidentity
permutation of finite support, then $G$ contains the finitary alternating group
$\Alt(\Omega)$.
\end{theorem}

\begin{lemma}[The case $U$ finite]\label{lem:Ufinite}
Let $\Gamma=\Cay(\Qg,X\cup Tb)$ with $X=X^{-1}\subseteq A\setminus\{1\}$,
$c\in X$ and $cT=T$, and suppose that $U=X\symd cX$ is finite. Let $L$ be the
subgroup generated by $U$. The permutation $\eta$ that multiplies by $c$ on $Lb$
and fixes $\Qg\setminus Lb$ is a nonidentity automorphism of $\Gamma$ with finite
support. Consequently, if $\Gamma$ is nonempty and noncomplete, then $\Aut\Gamma$
is not primitive.
\end{lemma}

\begin{proof}
Since $c\in X$ and $1\notin X$, we have $c\in X\setminus cX\subseteq U$.
The set $U$ is finite, so $U\subseteq A_n$ for some $n$ by the definition of $A$
(Section~\ref{s:setup}), and $L\le A_n$. So
$L$ is a finite subgroup of $A$ containing $c$, and $\eta$ is a nonidentity
permutation with support $Lb$. Across the two sides, $\eta$ preserves
adjacency because $cT=T$ (by \eqref{eq:adj}). Adjacency within $Lb$ is preserved
because $\eta$ acts there as multiplication by the central element $c$. For
$p\in Lb$ and $q\in B\setminus Lb$, their difference in the sense of
\eqref{eq:adj} lies outside $L$, hence outside $U$, so its membership in $X$ is
unchanged by multiplication by $c$.

If $\Aut\Gamma$ were primitive, then by
Theorem~\ref{thm:wielandt} it would contain $\Alt(\Qg)$, which is doubly transitive, as $\Qg$ is
infinite \cite[Exercise~2.1.5]{dixon_mortimer_1996}. Since $\Gamma$ is nonempty it has an edge $\{\alpha,\beta\}$, and
since it is noncomplete it has distinct nonadjacent $\gamma,\delta$; an element of
$\Alt(\Qg)\le\Aut\Gamma$ sending $(\alpha,\beta)$ to $(\gamma,\delta)$ would carry an edge to a
non-edge, a contradiction.
\end{proof}

\begin{definition}[Standing hypothesis]\label{def:standing}
We say that \emph{\textup{(S)} holds} if $\Gamma=\Cay(\Qg,X\cup Tb)$ with
\[
X=X^{-1}\subseteq A\setminus\{1\},\quad c\in X,\quad cT=T,\quad
\varnothing\ne T\ne A,
\]
and $U=X\symd cX$ is infinite. Then $H=\Pi(T)$ is a proper, finite, cyclic subgroup of $A$ containing $c$
(as in the proof of Proposition~\ref{prop:reduction}: $\varnothing\ne T\ne A$ makes $H$ proper, so
Lemma~\ref{lem:prufer}(1) applies, and $c\in H$ since $cT=T$), and
we write $P=\Aut\Gamma$, $\rho=R(c)$, $\Ii=\rho^P$ and $t,v$ as in
\eqref{eq:tv}. The branch $|H|\ge4$ means $H=A_\nu$ with $\nu\ge2$; the branch
$|H|=2$ means $H=\langle c\rangle$.
\end{definition}

\begin{remark}
Requiring $U$ to be infinite in \textup{(S)} loses nothing: if $U$ is finite and $\Gamma$ is nonempty and
noncomplete, then $\Aut\Gamma$ is not primitive by Lemma~\ref{lem:Ufinite}.
\end{remark}

\section{The exact half-stabilizer}\label{s:H}

The \emph{half-stabilizer} is $\pst{A}$, the pointwise stabilizer of the side
fixed by the half-flip $t$. Statement (H) determines it exactly.

\begin{lemma}[Half-stabilizer rigidity]\label{lem:H}
Assume \textup{(S)}. For $h\in H$ define $\beta_h(a)=a$ and $\beta_h(ab)=hab$
$(a\in A)$. Then
\begin{equation}\label{eq:H}
\pst{A}=\{\beta_h:h\in H\}\cong H .
\end{equation}
In particular $\pst{A}$ is cyclic, and its unique nonidentity involution is
$t=\beta_c$.
\end{lemma}

\begin{proof}
The proof has three steps: the maps of the statement are automorphisms, the
half-stabilizer is finite, and each of its elements is one of these maps.
Equation \eqref{eq:adj} shows that each $\beta_h$ preserves the graph: within
$B$ it is multiplication by $h\in A$, and $\one_T(ahd)=\one_T(ad)$ because
$h\in\Pi(T)$; it is a bijection, with inverse $\beta_{h^{-1}}$, and fixes $A$, so it lies in $\pst{A}$.
To see what an element fixing the side can do on the other side,
we use cross-neighbourhoods, which such an element preserves. The
cross-neighbourhood of $ub$ in $A$ is $u^{-1}T$.
Consequently the classes of vertices of $B$ with equal such neighbourhoods are
precisely
\begin{equation}\label{eq:fibres}
\mathcal{F}_{\bar u}=Hub,\qquad \bar u\in A/H .
\end{equation}
Every element of $K=\pst{A}$ preserves each of these finite fibres.

Next we show that the fibres are linked rigidly, so that one fibre controls the
others. Between $Hub$ and $Hu'b$, the adjacency matrix has $(h_1,h_2)$-entry
\begin{equation}\label{eq:fibrematrix}
\one_X(h_1h_2^{-1}uu'^{-1}),\qquad h_1,h_2\in H .
\end{equation}
Suppose $Huu'^{-1}\cap U\ne\varnothing$. The corresponding subset
$Z=\{h\in H:huu'^{-1}\in X\}$ is not invariant under multiplication by $c$. Its
period subgroup in $H$ must therefore be trivial, since every nontrivial
subgroup of $H$ contains $c$ (Lemma~\ref{lem:prufer}(1)). As a result the columns of \eqref{eq:fibrematrix}
are pairwise distinct.

An element of $K$ fixing the first fibre pointwise consequently fixes the
second one pointwise. To reach every fibre we use the infinitude in the standing
hypothesis once more. The graph on $A/H$ linking distinct cosets when their
difference coset meets $U$ is connected: its set of connection directions, the image of $U$
in $A/H$, is inverse-closed, as $X=X^{-1}$ gives $U^{-1}=U$, and infinite, as $U$ is infinite and $H$ is
finite; and every infinite subset of $A/H$ generates that
Pr\"ufer group (Lemma~\ref{lem:prufer}(1) and~(3), with $D=H$). An element of $K$ fixing $Hb$ pointwise fixes every fibre $Hub$ pointwise: a path from $H$ to $Hu$ in
this connected graph steps only between linked cosets, and the first sentence of this paragraph carries
pointwise fixing along each step. Such an element therefore fixes $B$ pointwise; it also fixes $A$, so
it is the identity. Restriction to the single fibre $Hb$
is therefore injective, so $K$ is finite.

Finiteness forces every element to commute with the translations by the
Pr\"ufer group, and this determines it. The group $R(A)$ normalizes $K$, since each $R(a)$, $a\in A$, lies in $P$ and maps $A$ onto itself. Its conjugation homomorphism into the finite
group $\Aut(K)$ is trivial (Lemma~\ref{lem:prufer}(2)), so $K$ centralizes
$R(A)$. Let $k\in K$ and $k(b)=hb$; preservation of $Hb$ gives $h\in H$. Since
$ub=R(u^{-1})(b)$, for every $u\in A$,
\[
k(ub)=R(u^{-1})(k(b))=hbu^{-1}=hub .
\]
The permutation also fixes $A$, hence $k=\beta_h$. This proves the equality in \eqref{eq:H}.

Since $\beta_h\beta_{h'}=\beta_{hh'}$ and $\beta_h(b)=hb$, the map $h\mapsto\beta_h$ identifies $H$
with $\pst{A}$. Hence this stabilizer is cyclic, with unique involution $\beta_c=t$.
\end{proof}

The unique involution $t=\beta_c$ of $\pst{A}$ (Lemma~\ref{lem:H}) and the half-flip $v$
split $\rho$: by Lemma~\ref{lem:halves-props} below, $\rho=tv$, the
two factors are commuting involutions, and they fix the complementary sides $A$ and
$B$ pointwise. To use this splitting for an arbitrary matching $x\in\Ii$, the definition
below conjugates the whole triple $(t,v,\rho)$ by a single element of $P$, rather than
$t$ and $v$ separately; this keeps each half attached to the matching it comes from and
carries $A$ and $B$ to the two sides of the pair.

\begin{definition}[Halves]\label{def:halves}
Assume \textup{(S)}. We call each $x\in\Ii$ a \emph{full
matching}. A
\emph{complementary pair} for $x\in\Ii$ is a pair $(t_x,v_x)$ obtained by
simultaneously conjugating $(t,v,\rho)$: $(t_x,v_x,x)=(t^g,v^g,\rho^g)$ for some
$g\in P$. Its members are the \emph{halves} of $x$. Once the pair is fixed,
$t_x$ is the \emph{chosen half} and $v_x$ its \emph{complementary half} (and
$t_x$ is the complementary half of $v_x$); $\Fix(t_x)$ and $\Fix(v_x)$ are the
two \emph{sides} of the pair. A \emph{half-involution} is a $P$-conjugate of $t$.
\end{definition}

\begin{lemma}[Properties of halves]\label{lem:halves-props}
Assume \textup{(S)}. In \textup{(4)}--\textup{(6)}, let $x\in\Ii$, and let $(t_x,v_x)$ be a
complementary pair for $x$, with $(t_x,v_x,x)=(t^g,v^g,\rho^g)$ for some $g\in P$.
\begin{enumerate}
\item The hypotheses of Lemma~\ref{lem:halfflips} hold, so $t,v\in P$ and \eqref{eq:tvprops}
  holds. The set $S=X\cup Tb$ is inverse-closed and $1\notin S$, so every right translation of
  $\Qg$ lies in $P$; in particular $R(b)\in P$.
\item Every $x\in\Ii$ is a fixed-point-free involution.
\item The half-flip $v$ is a half-involution: $v=t^{R(b)}$ with $R(b)\in P$.
\item $t_x$ and $v_x$ are involutions in $P$; in particular $t_x,v_x\ne1$.
\item $x=t_xv_x$ and $[t_x,v_x]=1$; $\Fix(t_x)=g^{-1}(A)$ and $\Fix(v_x)=g^{-1}(B)$, which are
  complementary; and $t_x(p),v_x(p)\in\{p,x(p)\}$ for every vertex $p$, that is, each half sends
  every vertex either to itself or to its $x$-partner.
\item $(v_x,t_x)$ is also a complementary pair for $x$:
  \[
  (v_x,t_x,x)=(t^{R(b)g},v^{R(b)g},\rho^{R(b)g})\qquad\text{with }R(b)g\in P.
  \]
  In particular $t_x$ and $v_x$ are both half-involutions.
\end{enumerate}
\end{lemma}

\begin{proof}
(1),~(3) Hypothesis \textup{(S)} contains the hypotheses of Lemmas~\ref{lem:adj} and~\ref{lem:halfflips}, which
give $S=S^{-1}$, $t,v\in P$, \eqref{eq:tvprops} and $v=t^{R(b)}$, and $1\notin S$ as $X\subseteq A\setminus\{1\}$
and $Tb\subseteq B$; so every right translation, $R(b)$ among them, lies in $P$ (Section~\ref{ss:conv}).
(2),~(4),~(5) Conjugation by $g\in P$ carries $\rho$, $t$, $v$ to $x$, $t_x$, $v_x$ (for~(2), $x\in\Ii=\rho^P$ is
$\rho^g$ for some $g\in P$), so these items are the conjugates of the following facts: $\rho=R(c)$ is a
fixed-point-free involution in $P$; $t$ and $v$ are involutions in $P$ satisfying \eqref{eq:tvprops},
nontrivial as their fixed sets $A$ and $B$ partition $\Qg$ and are both nonempty; and $t$, $v$ agree with
$\rho$ off $A$, $B$ respectively, by \eqref{eq:tv}.
(6) As $R(b)^2=R(c)=\rho=tv$ commutes with $t$, $v^{R(b)}=t^{R(b)^2}=t$, and $\rho^{R(b)}=\rho$ by
Lemma~\ref{lem:halfflips}; so $(t,v,\rho)^{R(b)g}=(v,t,\rho)^g=(v_x,t_x,x)$ with $R(b)g\in P$, and $t_x=t^g$ and
$v_x=t^{R(b)g}$ are half-involutions.
\end{proof}

By Lemma~\ref{lem:halves-props}(2) we identify each full matching $x$ with the perfect matching
formed by its $x$-pairs.

\begin{corollary}[(H) for every half]\label{cor:H}
Assume \textup{(S)}. Let $t'$ be a half-involution with $C=\Fix(t')$ and
complementary half $v'$.
\begin{enumerate}
\item $\pst{C}$ and $\pst{C^\complement}$ are cyclic of order $|H|$, with unique
  involutions $t'$ and $v'$ respectively. Every nonidentity element of
  $\pst{C}$ has fixed set exactly $C$, and every nonidentity element of
  $\pst{C^\complement}$ has fixed set exactly $C^\complement$.
\item $C_P(t')=\sst{C}=C_P(v')$. A permutation in $P$ preserving $C$ setwise
  normalizes $\pst{C}$ and each of its subgroups.
\item If $|H|=2$, then $\pst{C}$ has no element of order four. If $|H|\ge4$, there are
  $\zeta\in\pst{C}$ of order four with $\zeta^2=t'$ (a
  \emph{supported square root} of $t'$; an element of $\pst{C}$ some power of
  which is $t'$ is a \emph{supported root} of $t'$), and more generally
  compatible elements $\zeta_j\in\pst{C}$ with $\zeta_1=t'$, $|\zeta_j|=2^j$ $(1\le j\le\nu)$ and
  $\zeta_j^2=\zeta_{j-1}$ $(2\le j\le\nu)$, where $|H|=2^\nu$.
\end{enumerate}
\end{corollary}

\begin{proof}
For $t'=t$:
Lemma~\ref{lem:H} gives $\pst{A}$, and each $\beta_h$ with $h\ne1$
moves every point $ub$ of $B$, because $hub\ne ub$. Conjugating by $R(b)$
(Lemma~\ref{lem:halfflips}) gives the statements for $B=\Fix(v)$, and
conjugation by $g\in P$ transports them to $t^g$. For (2): an element of $P$ centralizing $t'$
preserves $\Fix(t')=C$; conversely, an element preserving $C$ normalizes the
cyclic group $\pst{C}$, hence each of its characteristic subgroups, and fixes
its unique involution $t'$. The complementary half has the same setwise
stabilizer, hence the same centralizer. Part (3) holds because $\pst{C}$ is
cyclic of order $|H|$, by~(1): for $|H|=2$ it has no element of order four,
and for $|H|=2^\nu\ge4$ it contains a chain of roots of its involution.
\end{proof}

Sections~\ref{s:big}--\ref{s:two} prove the following.
\begin{center}
\emph{Under \textup{(S)}, the full automorphism group $P=\Aut\Gamma$ is imprimitive.}
\end{center}
Hypothesis \textup{(S)} does not assume that $P$ is primitive. The two cases are the branches of
Definition~\ref{def:standing}: for $|H|\ge4$ the goal is the block of Corollary~\ref{cor:block}, and for
$|H|=2$ it is Theorem~\ref{thm:two}. Section~\ref{s:proof} derives Theorem~\ref{thm:main} from it with
Proposition~\ref{prop:reduction} and Lemma~\ref{lem:Ufinite}.

\section{The branch \texorpdfstring{$|H|\ge4$}{|H|>=4}: finite interactions}\label{s:big}

Throughout this section (S) holds and $|H|\ge4$, that is, $H=A_\nu$ with
$\nu\ge2$; these are the hypotheses under which supported square roots of half-involutions
exist (Corollary~\ref{cor:H}(3)). In this section and the next, for named full matchings $x,y\in\Ii$, $(t,v)$ is a
complementary pair for $x$ and $(s,w)$ a complementary pair for $y$, that is, $(t,v)=(t_x,v_x)$ and
$(s,w)=(t_y,v_y)$ in the notation of Definition~\ref{def:halves}; for $x=\rho$ the pair $(t,v)$ is
the pair of half-flips \eqref{eq:tv}.

\emph{Standing facts.} The following facts are proved before they are used, and proofs in this section
and the next cite them by their tags. Here $x,y\in\Ii$ have complementary pairs $(t_x,v_x)$ and
$(t_y,v_y)$, and $t'$ is a half-involution with $C=\Fix(t')$ and complementary half $v'$.
\begin{itemize}
\item $t_x$ and $v_x$ are involutions in $P$ and half-involutions, and $(v_x,t_x)$ is
  also a complementary pair for $x$; $x=t_xv_x$ and $[t_x,v_x]=1$; $\Fix(t_x)$ and $\Fix(v_x)$ are
  complementary; and each half sends every vertex to itself or to its $x$-partner
  (Lemma~\ref{lem:halves-props}(4)--(6)).
\item $x$ is a fixed-point-free involution; every right translation of $\Qg$ lies in
  $P$, in particular $\rho$, so that the $P$-conjugacy class $\Ii=\rho^P$ lies in $P$
  (Definition~\ref{def:standing}, Lemma~\ref{lem:halves-props}(1),(2)).
\item $\pst{C}$ and $\pst{C^\complement}$ are cyclic of order $|H|=2^\nu$, with unique
  involutions $t'$ and $v'$; every nonidentity element of $\pst{C}$ has fixed set exactly $C$, and every
  nonidentity element of $\pst{C^\complement}$ has fixed set exactly $C^\complement$; and $t'$ has a supported square root
  (Definition~\ref{def:standing}, Corollary~\ref{cor:H}(1),(3)).
\end{itemize}

Let $L$ be a group of permutations generated by named elements
$g_1,\dots,g_k$. Its orbits are also called its \emph{components}, and the
\emph{marked type} of a component is the component together with the
restrictions of $g_1,\dots,g_k$ to it, up to a bijection carrying each
restriction to the corresponding one. 

\begin{lemma}[Finite marked actions]
\label{lem:finite-marked-actions}
Let $L=\langle g_1,\dots,g_k\rangle\le\Sym(\Omega)$ with $k<\infty$, and suppose there is an integer $N$ with
$|L\omega|\le N$ for every $\omega\in\Omega$. Then finitely many marked types of orbits occur, and restriction to
one representative $\Omega_j$ of each occurring type gives an injective homomorphism
$L\to\prod_{j=1}^{s}\Sym(\Omega_j)$. In particular $L$ is finite.
\end{lemma}

\begin{proof}
On an $n$-point set a marked action is an ordered $k$-tuple in $\Sym(n)^k$, so there are finitely many marked
types with at most $N$ points. Restriction to a union of orbits is a homomorphism. If a word
$w(g_1,\dots,g_k)$ is the identity on each representative, a bijection between components of the same marked
type that carries each restriction to the corresponding one shows that $w$ is the identity on every component,
hence on $\Omega$. So restriction is injective, and its target is finite.
\end{proof}

The following lemma is standard, as is the homogenization of an affine action used in the proof of Lemma~\ref{lem:affine}; it finds three-cycles in such groups; it is used in Lemma~\ref{lem:order3} and
Theorem~\ref{thm:no-rays}.

\begin{lemma}[Alternating-factor lemma]\label{lem:altfactor}
Let $\Omega$ be a finite set with $|\Omega|\ge5$, let $G$ be a finite group, and let $J\le\Sym(\Omega)\times G$ be a
subgroup whose projection to $\Sym(\Omega)$ is surjective. If a prime divides $|\Omega|!$ but not $|G|$, then
$\Alt(\Omega)\times1\le J$.
\end{lemma}

\begin{proof}
Let $K=\{k\in\Sym(\Omega):(k,1)\in J\}$, so that $K\times1$ is the kernel of the projection of $J$ to $G$.
Conjugation by an element of $J$ acts on $K\times1$ through the first coordinate, and every element of
$\Sym(\Omega)$ is the first coordinate of an element of $J$; so $K$ is normal in $\Sym(\Omega)$. The projection
of $J$ onto $\Sym(\Omega)$ maps $K\times1$ onto $K$, so it induces a surjection of $J/(K\times1)$ onto
$\Sym(\Omega)/K$; and $J/(K\times1)$ is isomorphic to the projection of $J$ to $G$, a subgroup of $G$. Hence
$|\Sym(\Omega):K|$ divides $|G|$, and the prime of the hypothesis, which divides $|\Sym(\Omega)|=|\Omega|!$ but
not $|G|$, divides $|K|$; in particular $K\ne1$. The only normal subgroups of a symmetric group of degree at
least five are the trivial group, the alternating group and the whole group \cite[3.2.3]{Robinson1996}, so
$\Alt(\Omega)\le K$, that is, $\Alt(\Omega)\times1\le J$.
\end{proof}

\subsection{Supported square roots and the odd-order obstruction}\label{ss:odd-order}

In this subsection $t$ denotes a half-involution with fixed set $C$, and $v$
its complementary half-involution; by Corollary~\ref{cor:H} everything proved
for the original pair applies to them.

\begin{lemma}[Odd cycles through the fixed set]\label{lem:oddrot}
Assume \textup{(S)} and $|H|\ge4$, and let $t$ be a half-involution with fixed
set $C$ and complementary half $v$. Let $z\in P$ be an involution such that
$r=tz$ has odd finite order.
Then $r^3=1$. Moreover, if a point $\alpha$ of $C$ lies on a nontrivial $r$-cycle,
then $v(\alpha)$ is fixed by $r$ \textup{(}equivalently, by $z$\textup{)}.
\end{lemma}

\begin{proof}
The proof bounds the odd cycles through the fixed set by building automorphisms
that fix the fixed set pointwise and comparing them with Corollary~\ref{cor:H}(1). An orbit of
$\langle t,z\rangle$ meeting $C$ is an odd $r$-cycle with a unique $t$-fixed
point. Indeed, since $trt=r^{-1}$, $t$ permutes the $r$-cycles. An $r$-cycle through a point of $C$
is $t$-invariant, so it is the whole $\langle t,z\rangle$-orbit. The half $t$ acts on it as a
reflection of a cycle of odd length (the length divides the odd order of $r$), which fixes exactly
one point. The complementary half $v$ pairs this fixed point $\alpha$ with the
fixed point $\beta$ of a different odd cycle: $v$ commutes with $t$
(Lemma~\ref{lem:halves-props}(5)), so $\beta=v(\alpha)$ is fixed by $t$, that is, $\beta\in C$;
$\beta\ne\alpha$ because $\Fix(v)=C^\complement$; and each $\langle t,z\rangle$-orbit
meeting $C$ meets it in exactly one point, so $\beta$ lies on a different odd cycle. Write their
lengths $\ell,\ell'$. Set
$v_j=r^jvr^{-j}$; these conjugates are the letters of the words below. Since $\Fix(v)=C^\complement$ and each
$\langle t,z\rangle$-orbit meeting $C$ meets it in exactly one point, $v$ acts on the union of two paired
cycles as the transposition $(\alpha\ \beta)$ and fixes every $\langle t,z\rangle$-orbit disjoint from $C$
pointwise; as $r$ preserves each $r$-cycle, $v_j$ acts on a paired component as
$(r^j\alpha\ \ r^j\beta)$ and trivially on every orbit disjoint from $C$.

The first word rules out two nontrivial paired cycles. The global word
$W_1=(v_1v_{-1})^3$ fixes $C$ pointwise; on a paired component it is computed as follows. If
$\ell=\ell'=1$, then $v_1=v_{-1}=(\alpha\ \beta)$ and $W_1$ is the identity. If exactly one of $\ell,\ell'$
is one, then $v_1$ and $v_{-1}$ share exactly one point, so $v_1v_{-1}$ is a three-cycle and $W_1$ is the
identity. If $\ell,\ell'>1$, then $v_1$ and $v_{-1}$ are disjoint and
$W_1=(r\alpha\ \ r\beta)(r^{-1}\alpha\ \ r^{-1}\beta)$, a nontrivial double transposition exchanging points
of the two cycles. Globally it is therefore an involution,
but its restriction there is neither the identity nor $t$, contradicting the
uniqueness of the involution in the cyclic group $\pst{C}$. So one of $\ell,\ell'$ is one, which is
the second assertion. Since $t$ fixes $\beta$, $r(\beta)=t(z(\beta))$ equals $\beta$ if and only if
$z(\beta)=\beta$, which is its parenthetical form.

The second word excludes an $r$-fixed point paired with a long cycle. The word
$W_2=(v_1v_2v_1)(v_{-1}v_{-2}v_{-1})$ also fixes $C$. On a component
consisting of an $r$-fixed point and an $\ell$-cycle through $\beta$ it is
\[
(r\beta,\ r^2\beta)(r^{-1}\beta,\ r^{-2}\beta).
\]
It is the identity for $\ell=1,3$ and is a double transposition differing from $t$
for odd $\ell\ge5$. The same contradiction excludes $\ell\ge5$. Consequently $r^3$
fixes $C$ pointwise. It has odd order and lies in the $2$-group $\pst{C}$, so $r^3=1$.
\end{proof}

\begin{lemma}[No product of order three]\label{lem:order3}
Assume \textup{(S)} and $|H|\ge4$, and let $t$ be a half-involution. There is
no involution $z\in P$ with $|tz|=3$.
\end{lemma}

\begin{proof}
Suppose $|tz|=3$. Let $C=\Fix(t)$ and
$v$ its complementary half. Every vertex is moved by exactly one of $t,v$, and $(tz)^3=1$ leaves four kinds of
$\langle t,z\rangle$-orbit: a pair on which $t=z$, a path $t=(1\,2)$,
$z=(2\,3)$ whose end $3$ lies in $C$, an alternating hexagon, and fixed
points. A path's end $3$ is $v$-paired with a point of $C$ which, by
Lemma~\ref{lem:oddrot}, is fixed by $z$; so two paths cannot be joined by a
$v$-edge. The marked types of $L=\langle t,z,v\rangle$ are therefore of sizes
$4,2,2,6$: On a four-point component its generators can be
labelled
\[
t=(1\,2),\qquad z=(2\,3),\qquad v=(3\,4).
\]
The two-point types are $(t,z,v)=(1,1,(1\,2))$ and $((1\,2),(1\,2),1)$. A
six-point regular $S_3$ component is excluded by the word $[(vz)^3,t]$: it has
order three there, is the identity on the other types, and fixes $C$ pointwise,
while $\pst{C}$ is a $2$-group. There is at least one four-point component
because $t\ne z$.

Let $\zeta$ be a supported square root of $t$, so $\zeta^2=t$; as $\zeta\in\pst{C}$
is not the identity, $\Fix(\zeta)=C$. It pairs the two-cycles of $t$. Since every
remaining $L$-component contains at most one $t$-two-cycle, adjoining $\zeta$ gives
components of sizes $8,6,4,2$, with the following restrictions (unlisted points
fixed):
\[
\begin{array}{c|llll}
\text{size} & t & z & v & \zeta\\\hline
8 & (1\,2)(5\,6) & (2\,3)(6\,7) & (3\,4)(7\,8) & (1\,5\,2\,6)\\
6 & (1\,2)(5\,6) & (2\,3)(5\,6) & (3\,4) & (1\,5\,2\,6)\\
4 & (1\,2)(3\,4) & (1\,2)(3\,4) & 1 & (1\,3\,2\,4)\\
2 & 1 & 1 & (1\,2) & 1
\end{array}
\]
The two possible root orientations are isomorphic as marked actions: interchange
identical four-point components, or swap the two vertices in a two-point
component. So there are four marked types. The group $\langle t,z,v,\zeta\rangle$ has four generators, and its orbits have at most eight points by the list above. It is therefore finite by Lemma~\ref{lem:finite-marked-actions}; this finiteness is what the
group-theoretic arguments below require.

The eight-point image is $S_8$. It is transitive, contains the five-cycle
$z\,\zeta\,z\,\zeta^{-1}=(1\,6\,2\,3\,7)$, and cannot have a block system with blocks of
size $2$ or
$4$: a five-cycle would have to fix all blocks and act trivially inside them.
Jordan's prime-cycle theorem \cite[Theorem~13.9]{WielandtFPG}, \cite[Theorem~1.1]{Jones2014} (a primitive group of degree $n$
containing a cycle of prime length $p\le n-3$ contains the alternating group)
gives $\Alt(\{1,\dots,8\})$, and the odd generator $\zeta$ gives $S_8$. The six-point image is
$S_6$: $v^z=(2\,4)$, and $(2\,4)^{\zeta^k}$ for $k=1,2,3$ is $(5\,4)$, $(1\,4)$,
$(6\,4)$; together with $v=(3\,4)$ these are all transpositions $(i\,4)$, which
generate $S_6$. The other images are $C_4$ and $C_2$.

By Lemma~\ref{lem:finite-marked-actions}, restriction to one component of each occurring marked type embeds
$\langle t,z,v,\zeta\rangle$ in the product of its images on those components; let $J$ be that image. If an eight-point component occurs, then
$J\le S_8\times G$, where $S_8$ is the image on the eight-point type and $G$ is the product of the images
on the other types, of order dividing $6!\cdot4\cdot2$ and so prime to $7$. The alternating-factor lemma
(Lemma~\ref{lem:altfactor}), with $\Omega=\{1,\dots,8\}$ and the prime $7$, gives
$\Alt(\{1,\dots,8\})\times1\le J$. A three-cycle on three of $\{1,2,5,6\}$ that is the identity on
all other types therefore lies in $\langle t,z,v,\zeta\rangle$; it fixes $C$
pointwise in every component, contrary to $\pst{C}$ being a $2$-group. If there
is no eight-point type, a six-point type occurs, because the $t$-two-cycle of a
four-point component is paired by $\zeta$ with a $t$-two-cycle of another
component. Then $J\le S_6\times G$, where $G$, the product of the images $C_4$ and $C_2$ on the other
types, has order dividing $8$, and the alternating-factor lemma with $\Omega=\{1,\dots,6\}$ and the
prime $3$ gives $\Alt(\{1,\dots,6\})\times1\le J$; so a three-cycle on three of $\{1,2,5,6\}$ that is the
identity on all other types gives the same contradiction.
\end{proof}

\begin{theorem}[Odd-order obstruction]\label{thm:odd-order}
Assume \textup{(S)} and $|H|\ge4$. A finite-order product of a half-involution
and an arbitrary involution of $P$ has $2$-power order.
\end{theorem}

\begin{proof}
Let $t$ be a half-involution, $z\in P$ an involution,
and $|tz|=2^ed$ with $d$ odd. Replace $z$ by $(tz)^{2^e}t$, an involution whose product with $t$ has
order $d$. Lemma~\ref{lem:oddrot} gives $d\in\{1,3\}$, and
Lemma~\ref{lem:order3} excludes $d=3$.
\end{proof}

\subsection{Finite subgroups and full matchings}\label{ss:finite}

We use the finite Baer--Suzuki theorem \cite[6.7.6]{kurzweil_stellmacher_2004}, \cite[Lemma~7]{Suzuki1965} in the
following form, and only inside finite groups: \emph{if an involution $z$ of a
finite group $J$ generates a $2$-group with every conjugate of itself, then
$z\in\Opt(J)$.} Here $\Opt(J)$ is the \emph{$2$-core} of $J$, its largest normal
$2$-subgroup.

\begin{definition}[Two-matchings graph]\label{def:twomatch}
For $x,y\in\Ii$, the \emph{two-matchings graph} of $x$ and $y$ is the multigraph on $\Qg$ whose
edges are the $x$-pairs $\{p,x(p)\}$ and the $y$-pairs $\{p,y(p)\}$, $p\in\Qg$; a pair that is both an
$x$-pair and a $y$-pair gives two parallel edges, one in each matching (a \emph{doubled edge}). Its
components are the \emph{two-matching components} of $x$ and $y$; a path or cycle in it is
\emph{alternating} if consecutive edges belong to different matchings.
A component that is a doubled edge, a finite alternating cycle or a doubly infinite alternating path is
\emph{labelled} by $e_i,o_i$ with $x(e_i)=o_i$, $y(o_i)=e_{i+1}$, where $i\in\Z/m\Z$ on a
cycle of $2m$ vertices \textup{(}a doubled edge being the case $m=1$\textup{)} and $i\in\Z$ on a
doubly infinite path \textup{(}origin and orientation arbitrary\textup{)}, so that
\[
x(e_i)=o_i,\quad x(o_i)=e_i;\qquad y(o_i)=e_{i+1},\quad y(e_i)=o_{i-1};
\]
on a component of size $2m$ we also write $e_i=2i$ and $o_i=2i+1$. The matching edges
\[
R_i=\{e_i,o_i\},\qquad B_i=\{o_i,e_{i+1}\}
\]
are the \emph{red} edges, of $x$, and the \emph{blue} edges, of $y$. For complementary pairs
$(t_x,v_x)$ for $x$ and $(t_y,v_y)$ for $y$ \textup{(}Definition~\ref{def:halves}\textup{)}, the
\emph{red mask} $M=(M_i)$ and the \emph{blue mask} $N=(N_i)$, the masks of $t_x$ and $t_y$, are the
binary sequences with $M_i=1$ exactly when $t_x$ exchanges $e_i,o_i$ and $N_i=1$ exactly when $t_y$
exchanges $o_i,e_{i+1}$. A mask is \emph{constant} if all its entries are equal and
\emph{alternating} if any two cyclically consecutive entries differ; a component
is \emph{mixed} if one of its masks is alternating and the other constant.
\end{definition}

Figure~\ref{fig:twomatch} shows an eight-vertex alternating cycle with the labelling and the masks used from Section~\ref{ss:two-masks} on. Under \textup{(S)} no mixed eight-point component occurs: Theorem~\ref{thm:mixed8} excludes it when $|H|\ge4$, and Proposition~\ref{prop:mixed8two} when $H=\langle c\rangle$.

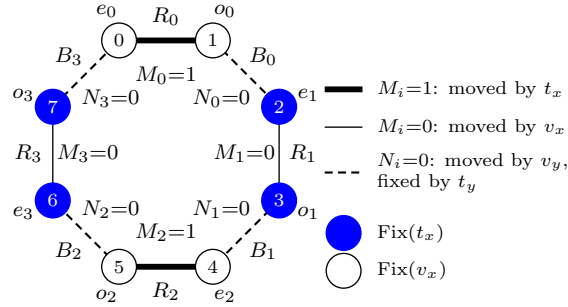
\begin{figure}[!htb]
\centering
\begin{tikzpicture}[x=1cm,y=1cm,font=\footnotesize,
  vtx/.style={circle,draw,line width=0.5pt,minimum size=4.6mm,inner sep=0pt,font=\scriptsize},
  fixt/.style={fill=blue,text=white,draw=blue},   fixv/.style={fill=white,text=black,draw=black}, redmv/.style={line width=2.2pt,black},          redfx/.style={line width=0.5pt,black},          bluedash/.style={line width=0.8pt,black,dash pattern=on 3pt off 2pt}, lab/.style={inner sep=1pt}]
\draw[redmv] (-0.620,1.497) -- (0.620,1.497);
\draw[redfx] (1.497,0.620) -- (1.497,-0.620);
\draw[redmv] (0.620,-1.497) -- (-0.620,-1.497);
\draw[redfx] (-1.497,-0.620) -- (-1.497,0.620);
\draw[bluedash] (0.620,1.497) -- (1.497,0.620);
\draw[bluedash] (1.497,-0.620) -- (0.620,-1.497);
\draw[bluedash] (-0.620,-1.497) -- (-1.497,-0.620);
\draw[bluedash] (-1.497,0.620) -- (-0.620,1.497);
\node[vtx,fixv] at (-0.620,1.497) {0};
\node[lab] at (-0.781,1.885) {$e_{0}$};
\node[vtx,fixv] at (0.620,1.497) {1};
\node[lab] at (0.781,1.885) {$o_{0}$};
\node[vtx,fixt] at (1.497,0.620) {2};
\node[lab] at (1.885,0.781) {$e_{1}$};
\node[vtx,fixt] at (1.497,-0.620) {3};
\node[lab] at (1.885,-0.781) {$o_{1}$};
\node[vtx,fixv] at (0.620,-1.497) {4};
\node[lab] at (0.781,-1.885) {$e_{2}$};
\node[vtx,fixv] at (-0.620,-1.497) {5};
\node[lab] at (-0.781,-1.885) {$o_{2}$};
\node[vtx,fixt] at (-1.497,-0.620) {6};
\node[lab] at (-1.885,-0.781) {$e_{3}$};
\node[vtx,fixt] at (-1.497,0.620) {7};
\node[lab] at (-1.885,0.781) {$o_{3}$};
\node[lab] at (0.000,1.817) {$R_{0}$};
\node[lab] at (0.000,1.037) {$M_{0}{=}1$};
\node[lab] at (1.817,0.000) {$R_{1}$};
\node[lab] at (1.037,0.000) {$M_{1}{=}0$};
\node[lab] at (0.000,-1.817) {$R_{2}$};
\node[lab] at (0.000,-1.037) {$M_{2}{=}1$};
\node[lab] at (-1.817,-0.000) {$R_{3}$};
\node[lab] at (-1.037,-0.000) {$M_{3}{=}0$};
\node[lab] at (1.285,1.285) {$B_{0}$};
\node[lab] at (0.733,0.733) {$N_{0}{=}0$};
\node[lab] at (1.285,-1.285) {$B_{1}$};
\node[lab] at (0.733,-0.733) {$N_{1}{=}0$};
\node[lab] at (-1.285,-1.285) {$B_{2}$};
\node[lab] at (-0.733,-0.733) {$N_{2}{=}0$};
\node[lab] at (-1.285,1.285) {$B_{3}$};
\node[lab] at (-0.733,0.733) {$N_{3}{=}0$};
\draw[redmv] (2.10,0.85) -- ++(0.5,0);
\node[anchor=west,inner sep=1pt,font=\scriptsize,align=left] at (2.75,0.85) {$M_i{=}1$: moved by $t_x$};
\draw[redfx] (2.10,0.35) -- ++(0.5,0);
\node[anchor=west,inner sep=1pt,font=\scriptsize,align=left] at (2.75,0.35) {$M_i{=}0$: moved by $v_x$};
\draw[bluedash] (2.10,-0.25) -- ++(0.5,0);
\node[anchor=west,inner sep=1pt,font=\scriptsize,align=left] at (2.75,-0.25) {$N_i{=}0$: moved by $v_y$,\\fixed by $t_y$};
\node[vtx,fixt] at (2.35,-1.05) {};
\node[anchor=west,inner sep=1pt,font=\scriptsize,align=left] at (2.75,-1.05) {$\Fix(t_x)$};
\node[vtx,fixv] at (2.35,-1.55) {};
\node[anchor=west,inner sep=1pt,font=\scriptsize,align=left] at (2.75,-1.55) {$\Fix(v_x)$};
\end{tikzpicture}
\caption{A mixed eight-vertex configuration satisfying Definition~\ref{def:twomatch}: an alternating cycle of
  the perfect matchings $x$ and $y$, labelled
  $e_i=2i$, $o_i=2i+1$ ($i\in\Z/4\Z$). The red edges $R_i$, of $x$, are solid and the blue edges $B_i$,
  of $y$, dashed. Drawn: $M=(1,0,1,0)$, alternating, so $t_x$ moves
  $R_0,R_2$ (heavy) and its complementary half $v_x$ moves $R_1,R_3$; filled vertices form
  $\Fix(t_x)$ on the component. $N=(0,0,0,0)$, constant, so $t_y$ is the identity here and $v_y=y$.}
\label{fig:twomatch}
\end{figure}

\begin{lemma}[Component structure]\label{lem:components}
Assume \textup{(S)}, and let $x,y\in\Ii$.
\begin{enumerate}
\item Each two-matching component of $x$ and $y$ is a doubled edge (an $x$-pair that is also a
  $y$-pair), a finite alternating even cycle, or a doubly infinite alternating path.
\item Every half of $x$ or of $y$ preserves each component.
\item Label a component as in Definition~\ref{def:twomatch}. Then $xy(e_i)=e_{i-1}$ and $xy(o_i)=o_{i+1}$; so $xy$ preserves the
  component and acts on it with order $m$ on a cycle of $2m$ vertices and with infinite order on a
  doubly infinite path. In particular, if $xy$ has finite order, every component is finite and $m$
  divides $|xy|$.
\end{enumerate}
\end{lemma}

\begin{proof}
(1) Each vertex lies on exactly one edge of each matching, so every vertex has degree two with its two
edges from different matchings; a connected graph of this kind is a doubled edge, a finite cycle,
necessarily even, or a doubly infinite path. (2) Every half of $x$ or of $y$ sends a vertex to itself
or to its $x$-partner (respectively $y$-partner), by Lemma~\ref{lem:halves-props}(5). (3) In the
labelling, $xy(e_i)=x(o_{i-1})=e_{i-1}$ and $xy(o_i)=x(e_{i+1})=o_{i+1}$, so $xy$ moves each vertex two
steps along the component, forward on one of the two classes of alternate vertices and backward on
the other. On a cycle of $2m$ vertices each class has $m$ points and is permuted cyclically, so the
order is $m$ (for $m=1$ it is the identity on the doubled edge). On a doubly infinite path
$(xy)^n(e_0)=e_{-n}\ne e_0$ for every $n\ge1$, so the order is infinite. If $xy$ has finite order,
$(xy)^{|xy|}=1$ restricts to the identity on each component, so no component is a doubly infinite path
and $m$ divides $|xy|$.
\end{proof}

\begin{remark}
The two-matching components are the orbits of $\langle x,y\rangle$. On each component $\Theta$ the induced group
$\langle x|_\Theta,y|_\Theta\rangle$ is regular: of order $2$ on a doubled edge, dihedral of order $2m$ on a cycle of
$2m$ vertices with $m\ge2$, and infinite dihedral on a doubly infinite path. The restriction of $\langle x,y\rangle$
to one component need not be faithful.
Figure~\ref{fig:twomatch} shows such an orbit with $m=4$.
\end{remark}
\FloatBarrier

\begin{proposition}[Finite subgroups and the $2$-core]\label{prop:two-core}
Assume \textup{(S)} and $|H|\ge4$.
\begin{enumerate}
\item If $J\le P$ is finite and contains a half-involution $t$,
  then $t\in\Opt(J)$.
\item For every finite subgroup $L$ of $P$, $\Ii\cap L\subseteq\Opt(L)$.
\item If $x,y\in\Ii$ and $xy$ has finite order, then
  $F=\langle t_x,v_x,t_y,v_y\rangle$ is a finite $2$-group, and $|xy|$ is a
  power of two.
\end{enumerate}
\end{proposition}

\begin{proof}
(1) A half-involution is the chosen half of a full matching (Definition~\ref{def:halves}), so $t$
and each $t^g$ with $g\in J$ are involutions in $P$. For $g\in J$ the
product $tt^g$ has finite order, so by Theorem~\ref{thm:odd-order} the dihedral group
$\langle t,t^g\rangle$ is a $2$-group. The finite Baer--Suzuki theorem then gives
$t\in\Opt(J)$.

(2) If $L$
is finite and contains a full matching $x\in\Ii$, adjoin its two halves. They preserve every $L$-orbit because each sends a vertex either to
itself or its $x$-partner, and $x\in L$. Since $L$ also preserves its own orbits, the orbits of the
enlarged group are exactly the $L$-orbits. Put $\widehat L=\langle L,t_x,v_x\rangle$. It is finitely generated, since $L$ is finite and two generators have been adjoined. Its orbits are the $L$-orbits and have at most $|L|$ points. Hence $\widehat L$
is finite by Lemma~\ref{lem:finite-marked-actions}. The group $\widehat L$ is a finite subgroup of $P$, and
both halves belong to its $2$-core by (1). Their product $x=t_xv_x$
therefore lies in the $2$-core of $\widehat L$, and in $L$. The intersection of that $2$-core with
$L$ is normal in $L$, since the $2$-core is normal in $\widehat L$, which contains $L$; and it is a
$2$-group, as a subgroup of one. So it lies in $\Opt(L)$, the largest normal $2$-subgroup of $L$, and
$x\in\Opt(L)$.

(3) If $xy$
has finite order, the components of the two-matchings graph are
finite with uniformly bounded sizes: by Lemma~\ref{lem:components}(3), each is a cycle of $2m$
vertices with $m$ dividing $|xy|$, so every component has at most $2|xy|$ vertices. The four
halves $t_x,v_x,t_y,v_y$ preserve these components (Lemma~\ref{lem:components}(2)), so every $F$-orbit lies
in one of them and has at most $2|xy|$ vertices. As $F$ has four generators and every $F$-orbit has at most $2|xy|$ points, $F$ is finite by Lemma~\ref{lem:finite-marked-actions}, and $F\le P$. It is
generated by four half-involutions, each in $\Opt(F)$ by (1), so $F$ is a
$2$-group. In particular $|xy|$ is a power of two, since $xy=t_xv_xt_yv_y\in F$.
\end{proof}

\subsection{Finite orbits through the fixed set}\label{ss:no-rays}

The group $P$ carries the topology of pointwise convergence on the discrete
countable set $\Qg$.

\begin{lemma}[Closedness]\label{lem:closed}
If $g_j\in P$ and, for every vertex $p$, the sequence $g_j(p)$ is eventually
constant with value $g(p)$, and $g$ is a bijection, then $g\in P$.
\end{lemma}

\begin{proof}
Adjacency is tested on two vertices at a time, so eventual constancy at each
vertex suffices. For any two vertices $p,q$, once both $g_j(p)$ and $g_j(q)$
have stabilized,
$g(p)\sim g(q)$ if and only if $g_j(p)\sim g_j(q)$, if and only if $p\sim q$.
\end{proof}

Lemma~\ref{lem:closed} is the case of a graph of the fact that the automorphism group of a
first-order structure is closed in the topology of pointwise convergence
\cite[Theorem~2.8(b)]{Cameron2007Aspects}. 

\begin{theorem}[Finite orbits through the fixed set]\label{thm:no-rays}
Assume \textup{(S)} and $|H|\ge4$. For every half-involution $t$, every
involution $z\in P$ and every $\alpha\in C=\Fix(t)$, the orbit of $\alpha$ under
$r=tz$ is finite.
\end{theorem}

The powers $r^{j!}$ eventually fix each point with finite $r$-orbit and carry each point of $C$ with infinite orbit out of $C$, so suitable words in $t$, the complementary half $v$ of $t$, a square root $\zeta$ of $v$ and their conjugates by $r^{j!}$ have pointwise limits in $\pst{C^\complement}$ (Lemma~\ref{lem:closed}); a point of $C$ with infinite orbit would make one such limit an element of order three in this cyclic $2$-group.

\begin{proof}
We show that no point of the fixed set has an infinite orbit. Let $v$ be the
complementary half of $t$, so that $\Fix(v)=C^\complement$, and let
\[
C_\infty=\{\alpha\in C:\langle r\rangle\alpha\ \text{is infinite}\}.
\]
Since $trt=r^{-1}$, an infinite $r$-orbit contains at most one point of $C$.
Indeed, if $\alpha$ and $r^j\alpha$ are fixed by $t$, then $r^{2j}\alpha=\alpha$,
so $j=0$. Thus $r^iC_\infty$ and $r^jC_\infty$ are disjoint for $i\ne j$, a
fact used at the end of Step~1.

\emph{Step 1: no $v$-pair has two infinite endpoints.} The idea is to build, as
a pointwise limit, an automorphism that acts only on such pairs. Put $n_j=j!$ for $j\ge1$, so that the $n_j$ are distinct,
$v_j=r^{-n_j}vr^{n_j}$ and $g_j=(vv_j)^3$. The permutation $v_j$ moves a vertex
$p$ exactly when $r^{n_j}p\in C$. Fix a $v$-pair $\{\alpha,\beta\}$ in $C$, and
take $j$ so large that $r^{n_j}$ fixes every endpoint of this pair whose
$r$-orbit is finite. There are three cases:
\begin{itemize}
\item both endpoints have finite orbits: $v_j=v$ on the pair, and the cube is
  the identity there;
\item exactly one endpoint, say $\beta$, has an infinite orbit: $v$ and $v_j$
  act on $\{\alpha,\beta,r^{-n_j}\beta\}$ as $(\alpha\,\beta)$ and
  $(\alpha\ r^{-n_j}\beta)$; they share exactly the finite endpoint, so their
  product has order three and the cube is the identity;
\item both endpoints have infinite orbits: the old transposition
  $(\alpha\,\beta)$ and its shifted copy $(r^{-n_j}\alpha\ r^{-n_j}\beta)$ are
  disjoint, so the cube is their product.
\end{itemize}
Different $v$-pairs give disjoint configurations. For every fixed vertex
$p\notin C$, both $v$ and $v_j$ fix $p$ for all sufficiently large $j$: $v$
fixes $C^\complement$, and $v_j$ moves $p$ only if $r^{n_j}p\in C$. If the $r$-orbit of
$p$ is finite, this fails as soon as its length divides $n_j$ (factorial
divisibility). If it is infinite, it holds for at most one $j$, because that
orbit contains at most one point of $C$. Combining the two observations, $g_j$ converges pointwise, eventually per
point, to the permutation $g$
which equals $v$ on exactly those $v$-pairs with both endpoints in $C_\infty$,
and is the identity elsewhere. This limit is a bijection, and the
approximants lie in $P$: $t,v\in P$ and $z\in P$, so $r=tz\in P$
and $v_j,g_j\in P$. So the limit
lies in $P$ by Lemma~\ref{lem:closed}. It fixes $C^\complement$ pointwise, so it
lies in $\pst{C^\complement}$. If it is nonidentity, its fixed set is exactly $C^\complement$;
then it moves every point of $C$, every $v$-pair in $C$ has both endpoints in
$C_\infty$, and $C_\infty=C$. But $C_\infty=C$ would make $r^{-1}C$, $C$ and
$r^{-2}C$ pairwise disjoint. The nonidentity permutation $r^{-1}vr\in P$ has
support $r^{-1}C$, fixes $C$ pointwise, and also fixes the nonempty subset
$r^{-2}C$ of $C^\complement$. That is impossible:
every nonidentity element of $\pst{C}$ has fixed set exactly $C$. Therefore
$g=1$, and
\[
C_\infty\cap v(C_\infty)=\varnothing .
\]
This disjointness restricts the local types in Step~2.

\emph{Step 2: a supported root leaves three local types.} To exploit Step~1 we
use a supported square root of the complementary half, whose orbits group the
fixed set into four-cycles. By Lemma~\ref{lem:halves-props}(6), $(v,t)$ is also a complementary
pair, so $v$ is a half-involution with $\Fix(v)=C^\complement$ and complementary half $t$, and
Corollary~\ref{cor:H}(3) applies to it with $C^\complement$ in place of $C$: choose
$\zeta\in\pst{C^\complement}$ with $\zeta^2=v$ and $|\zeta|=4$. Since $\zeta$ fixes $C^\complement$ pointwise it preserves $C$, and for
$p\in C$, $\zeta^2p=vp\ne p$; as $\zeta^4=1$, the $\zeta$-orbit of $p$ has length four, and its opposite points
$p,\zeta^2p$ form a $v$-pair. So the orbits of $\zeta$ in $C$ are four-cycles. By Step~1, on such an orbit $O$ the set
$C_\infty$ meets $O$ in either zero points, one point, or two adjacent points:
two opposite points would be a $v$-pair, and larger subsets also contain such a
pair.

Put $\tau_j=r^{-n_j}\zeta r^{n_j}$.
Fix one $\zeta$-orbit $O$, labelled $0,1,2,3$ (after a
rotation) so that $\zeta=(0\,1\,2\,3)$, and let $O_\infty=O\cap C_\infty$. For all
sufficiently large $j$ (depending on $O$),
\[
O\cup r^{-n_j}O_\infty\cup r^{n_j}O_\infty
\]
is invariant under $\zeta,\tau_j,t$, and the restrictions are the following exact
finite models.
{\small
\[
\begin{array}{llll}
\text{Type 0: } O_\infty=\varnothing,\ \text{degree }4: & \zeta=(0\,1\,2\,3), & \tau=(0\,1\,2\,3), & t=1.\\
\text{Type 1: } O_\infty=\{0\},\ \text{degree }6: & \zeta=(0\,1\,2\,3), & \tau=(4\,1\,2\,3), & t=(4\,5).\\
\text{Type 2: } O_\infty=\{0,1\},\ \text{degree }8: & \zeta=(0\,1\,2\,3), & \tau=(4\,5\,2\,3), & t=(4\,6)(5\,7).
\end{array}
\]}
Figure~\ref{fig:type2} shows the Type~2 model. In every type the old vertices are $0,1,2,3$. The other vertices are shifted
copies: in Type~1, $4=r^{-n_j}0$ and $5=r^{n_j}0$; in Type~2, $4=r^{-n_j}0$,
$5=r^{-n_j}1$, $6=r^{n_j}0$, $7=r^{n_j}1$. Finite-orbit old vertices are
fixed by $r^{n_j}$; infinite-orbit old vertices never return to $C$. Each
$\alpha\in O_\infty$ lies in an infinite $r$-orbit that meets $C$ only in $\alpha$. So the points
$r^{\pm n_j}\alpha$ ($n_j\ne0$) lie outside $C$, are fixed by $\zeta\in\pst{C^\complement}$, and are distinct from
$O$. Also $r^{-n_j}\alpha\ne r^{n_j}\alpha$, because the orbit is infinite. In Type~2 the points
$0,1\in C_\infty$ lie in different $r$-orbits, because an infinite orbit contains at most one point
of $C$. So $r^{\pm n_j}0$ and $r^{\pm n_j}1$ are four distinct points, and the set has exactly
$4+2|O_\infty|$ elements.

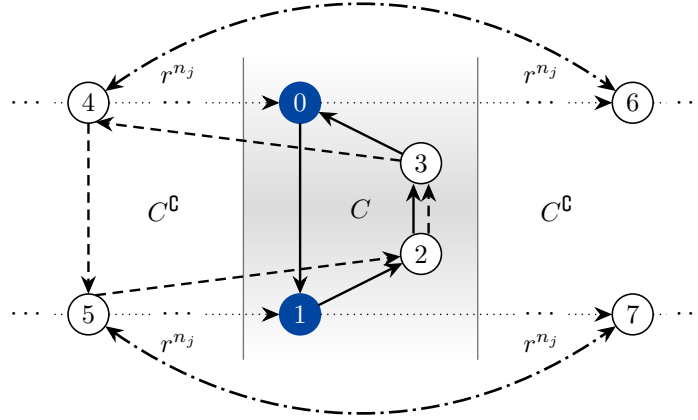
\begin{figure}[!htb]
\centering
\providecolor{figaccent}{RGB}{0,70,160}\begin{tikzpicture}[x=1cm,y=1cm,font=\normalsize,
  every node/.style={inner sep=0pt,outer sep=0pt},
  vopen/.style={circle,draw=black,fill=white,line width=0.6pt,minimum size=5.4mm,inner sep=0pt},
  vfill/.style={circle,draw=figaccent,fill=figaccent,text=white,line width=0.6pt,minimum size=5.4mm,inner sep=0pt},
  head/.style={-{Stealth[length=2.4mm,width=2.0mm]}},
  zarr/.style={head,line width=0.9pt},
  tarr/.style={head,line width=0.9pt,dash pattern=on 4pt off 2.5pt},
  tarc/.style={{Stealth[length=2.4mm,width=2.0mm]}-{Stealth[length=2.4mm,width=2.0mm]},line width=1.0pt,dash pattern=on 7pt off 2pt on 1.2pt off 2pt},
  rline/.style={head,line width=0.5pt,dotted},
  cband/.style={line width=0.5pt,black!55}]
\shade[top color=white,bottom color=black!14] (-1.55,0) rectangle (1.55,2.00);
\shade[top color=black!14,bottom color=white] (-1.55,-2.00) rectangle (1.55,0);
\draw[cband] (-1.55,-2.00) -- (-1.55,2.00) (1.55,-2.00) -- (1.55,2.00);
\node at (0,0) {$C$};
\node at (-2.6,0) {$C^{\complement}$};
\node at (2.6,0) {$C^{\complement}$};
\node[vfill] (v0) at (-0.80,1.40) {$0$};
\node[vfill] (v1) at (-0.80,-1.40) {$1$};
\node[vopen] (v2) at (0.80,-0.60) {$2$};
\node[vopen] (v3) at (0.80,0.60) {$3$};
\node[vopen] (v4) at (-3.60,1.40) {$4$};
\node[vopen] (v5) at (-3.60,-1.40) {$5$};
\node[vopen] (v6) at (3.60,1.40) {$6$};
\node[vopen] (v7) at (3.60,-1.40) {$7$};
\draw[dotted,line width=0.5pt] (v4) -- (-2.75,1.40);
\node at (-2.40,1.40) {$\cdots$};
\draw[rline] (-2.05,1.40) -- (v0);
\node at (-2.40,1.78) {$r^{n_j}$};
\draw[dotted,line width=0.5pt] (v0) -- (2.05,1.40);
\node at (2.40,1.40) {$\cdots$};
\draw[rline] (2.75,1.40) -- (v6);
\node at (2.40,1.78) {$r^{n_j}$};
\draw[dotted,line width=0.5pt] (-4.05,1.40) -- (v4);
\draw[dotted,line width=0.5pt] (v6) -- (4.05,1.40);
\node at (-4.40,1.40) {$\cdots$};
\node at (4.40,1.40) {$\cdots$};
\draw[dotted,line width=0.5pt] (v5) -- (-2.75,-1.40);
\node at (-2.40,-1.40) {$\cdots$};
\draw[rline] (-2.05,-1.40) -- (v1);
\node at (-2.40,-1.78) {$r^{n_j}$};
\draw[dotted,line width=0.5pt] (v1) -- (2.05,-1.40);
\node at (2.40,-1.40) {$\cdots$};
\draw[rline] (2.75,-1.40) -- (v7);
\node at (2.40,-1.78) {$r^{n_j}$};
\draw[dotted,line width=0.5pt] (-4.05,-1.40) -- (v5);
\draw[dotted,line width=0.5pt] (v7) -- (4.05,-1.40);
\node at (-4.40,-1.40) {$\cdots$};
\node at (4.40,-1.40) {$\cdots$};
\draw[zarr] (v0) -- (v1);
\draw[zarr] (v1) -- (v2);
\draw[zarr] ([xshift=-1.0mm]v2.north) -- ([xshift=-1.0mm]v3.south);
\draw[zarr] (v3) -- (v0);
\draw[tarr] ([xshift=1.0mm]v2.north) -- ([xshift=1.0mm]v3.south);
\draw[tarr] (v3) -- (v4.-70);
\draw[tarr] (v4) -- (v5);
\draw[tarr] (v5.70) -- (v2);
\draw[tarc] (v4) to[bend left=35] (v6);
\draw[tarc] (v5) to[bend right=35] (v7);
\end{tikzpicture}
\caption{The Type~2 model in the proof of Theorem~\ref{thm:no-rays}, for one sufficiently large $j$. The $\zeta$-orbit
  $O=\{0,1,2,3\}$ lies in the fixed set $C$ of $t$ (shaded), with $\zeta=(0\,1\,2\,3)$; its opposite points form
  $v$-pairs, and the filled points $0,1$ have infinite $r$-orbits (dotted), which meet $C$ only there. Their
  shifted copies are $4=r^{-n_j}0$, $5=r^{-n_j}1$, $6=r^{n_j}0$, $7=r^{n_j}1$. Solid arrows: $\zeta$, which
  fixes $4,\dots,7$. Dashed arrows: $\tau_j=r^{-n_j}\zeta r^{n_j}$, here $(4\,5\,2\,3)$, which is $\zeta$ with
  $0,1$ replaced by $4,5$. Dash-dotted double-headed arcs: $t=(4\,6)(5\,7)$. Step~5 shows that no $\zeta$-orbit
  is of this type.}
\label{fig:type2}
\end{figure}

Vertices outside the fixed set are handled by a single parity. For a fixed vertex
outside $C$, both $\zeta$ and $\tau_j$ eventually fix that vertex
and its $t$-partner. A word in $\zeta,\tau_j,t$ therefore eventually acts there as its
total $t$-exponent modulo two. Record this parity by a dummy degree-two model
$\zeta=\tau=1$, $t=(0\,1)$.

\emph{Step 3: the word-limit principle.} This step converts computations in the
finite models into automorphisms that fix the complementary side pointwise. Let
$W(\zeta,\tau,t)$ be any fixed word
which, in each of the Types~0, 1, 2, fixes the new vertices pointwise, and is
the identity in the dummy model. Then $W(\zeta,\tau_j,t)$ converges pointwise to a
permutation in $\pst{C^\complement}$. More precisely, for each vertex $p$ there is
$j_0(p)$ such that for $j\ge j_0(p)$ every letter of $W$ acts on the trajectory
of $p$ as in the model of the orbit containing $p$ (or in the dummy model);
$j_0(p)$ depends on $p$, and no uniform bound on finite $r$-orbit sizes is
needed, since each fixed old four-cycle has only finitely many finite-orbit
endpoints. The restriction of the limit to each old four-cycle is the
restriction of $W$ in that type. The limit is a bijection because $W$ preserves
the old four-point set in each model. Outside $C$ the dummy parity makes its
eventual action trivial. Since $\zeta,\tau_j,t\in P$ (with $r,\zeta\in P$, so $\tau_j=r^{-n_j}\zeta r^{n_j}\in P$),
every word $W(\zeta,\tau_j,t)$ lies in $P$; by Lemma~\ref{lem:closed} the limit lies in $P$, hence
in $\pst{C^\complement}$.

\FloatBarrier
\emph{Step 4: the finite images.} The contradiction will come from three-cycles,
so we need the images to be full symmetric groups. The image of the Type~1 model
is $S_6$. The
first two generators act transitively on $\{0,1,2,3,4\}$ and fix $5$.
Conjugates of $t=(4\,5)$ therefore include every $(i\,5)$, and these generate
$S_6$.

For Type~2 a transposition is produced explicitly. A
direct right-to-left multiplication gives
\[
t\,\tau^{-1}\,t\,\tau^{-1}\,t\,\zeta\,\tau\,t\,\zeta^{-1}\,\tau\,\zeta=(0\,1).
\]
Conjugating by $\zeta$ yields transpositions connecting $0,1,2,3$; conjugation by
$\tau^{-1}$ then supplies $(3\,4)$ and $(4\,5)$; conjugation by $t$ supplies $(3\,6)$
and $(6\,7)$. Transpositions along a connected graph generate the full
symmetric group. Thus the Type~2 image is $S_8$.

\emph{Step 5: Type~2 does not occur.} The alternating-factor lemma produces the
three-cycle we need. Form the finite diagonal marked group
$J_2$ generated on one copy of each of Types~0, 1, 2 and the dummy. Its
projections are $C_4,S_6,S_8,C_2$, so $J_2\le S_8\times G$, where $S_8$ is the projection to Type~2 and
$G=C_4\times S_6\times C_2$, whose order is not divisible by seven. The alternating-factor lemma
(Lemma~\ref{lem:altfactor}), with $\Omega=\{0,\dots,7\}$ and the prime $7$, gives
$\Alt(\{0,\dots,7\})\times1\le J_2$, and with it every three-cycle on Type~2 that is the identity on the
other types.

In
particular there is an element of $J_2$ which is $(0\,1\,2)$ on Type~2 and
the identity on the other types. Choose a word $W$ representing it. It fixes all
new vertices and is trivial on the dummy. By Step~3, if even one Type~2 old
orbit occurred, $W$ would give a nonidentity permutation of order three in
$\pst{C^\complement}$, which is cyclic of order $|H|=2^\nu$, so this
cyclic group has no element of order three. Thus Type~2 does not occur.

\emph{Step 6: Type~1 does not occur.} The same argument, with the alternating
group of degree six, excludes Type~1. Form the finite diagonal marked group $J_1$
on one copy of each of Types~0, 1 and the dummy. Its projections are $C_4$, $S_6$ (Step~4) and $C_2$, so
$J_1\le S_6\times G$ with $G=C_4\times C_2$ of order eight, and the alternating-factor lemma, with
$\Omega=\{0,\dots,5\}$ and the prime $3$, gives $\Alt(\{0,\dots,5\})\times1\le J_1$. Choose a word giving $(0\,1\,2)$ on Type~1
and the identity on the other models; it fixes the new vertices $4,5$ of
Type~1. Step~3 applies to this word although it is chosen without Type~2: by
Step~5 no $\zeta$-orbit is of Type~2, and Step~3 uses the condition on a type only at the
$\zeta$-orbits of that type, since for $j\ge j_0(p)$ every letter of the word acts on the
trajectory of $p$ as in the model of the orbit containing $p$ (or in the dummy model).
By Step~3 its pointwise limit gives a nonidentity element of order three
of $\pst{C^\complement}$ if a Type~1 orbit occurs, which is again impossible because
$\pst{C^\complement}$ is a $2$-group.

Hence no $\zeta$-orbit meets $C_\infty$; as these orbits cover the fixed set,
$C_\infty$ is empty.
\end{proof}

\subsection{Odd-part elimination, two masks, and the fourth-power bound}\label{ss:two-masks}

\begin{theorem}[Odd-part elimination]\label{thm:oddprimary}
Assume \textup{(S)} and $|H|\ge4$. For every half-involution $t$, with
$C=\Fix(t)$, and every involution $z\in P$, every orbit of $r=tz$ that meets $C$
has finite $2$-power size.
\end{theorem}

The proof takes the odd part of $r$ on its finite cycles as a pointwise limit of powers $r^{k_n}$, conjugates the complementary half $v$ of $t$ by it to get an involution $b\in P$, and reads off the cycles of $tb$ one $v$-pair at a time, so that Theorem~\ref{thm:odd-order} forces the odd part to fix $C$ pointwise.

\begin{proof}
Finiteness is Theorem~\ref{thm:no-rays}, applied to $t$ and $z$; it remains to exclude
an odd factor. Let $\Delta$ be the union of the finite $r$-cycles; then $C\subseteq\Delta$.
Let $v$ be the complementary half of $t$. $v\in P$ is an involution with
$\Fix(v)=C^\complement$, so $C$ is a union of $v$-pairs, and $v$ preserves $\Delta$ and fixes
$\Delta^\complement$ pointwise. As $t,z$ are involutions, $trt=zt=r^{-1}$, so $tr^mt=r^{-m}$ for all $m$.

\emph{The odd part.} Enumerate the finite $r$-cycles, with lengths $\ell_i=2^{\eta_i}d_i$,
$d_i$ odd. Let $\lambda_i\in\{0,\dots,\ell_i-1\}$ satisfy $\lambda_i\equiv0\pmod{2^{\eta_i}}$
and $\lambda_i\equiv1\pmod{d_i}$, and choose $k_n$ with $k_n\equiv0\pmod{2^{E_n}}$ and
$k_n\equiv1\pmod{O_n}$, where $E_n=\max_{i\le n}\eta_i$ and $O_n=\operatorname{lcm}_{i\le n}d_i$.
For $n\ge i$ we have $k_n\equiv\lambda_i\pmod{\ell_i}$, so $r^{k_n}=r^{\lambda_i}$ on the $i$-th cycle.
Let $r_{2'}$ be $r^{\lambda_i}$ on the $i$-th cycle, for every $i$, and the identity on $\Delta^\complement$.
Since $\gcd(\lambda_i,\ell_i)=2^{\eta_i}$, every $r_{2'}$-orbit is finite of odd length.

(i) \emph{For $\alpha\in C$ and all $j$, $t(r_{2'}^j\alpha)=r_{2'}^{-j}\alpha$; hence an
$r_{2'}$-orbit contains at most one point of $C$.} On the $r$-cycle of $\alpha$,
$r_{2'}=r^{\lambda}$ for its $\lambda=\lambda_i$, so $tr_{2'}^j\alpha=tr^{\lambda j}t\alpha=r^{-\lambda j}\alpha
=r_{2'}^{-j}\alpha$. If also $r_{2'}^j\alpha\in C$, then $r_{2'}^j\alpha=r_{2'}^{-j}\alpha$,
so $r_{2'}^{2j}\alpha=\alpha$, and $r_{2'}^j\alpha=\alpha$ because the orbit is odd.

(ii) \emph{$b=v^{r_{2'}}=r_{2'}^{-1}vr_{2'}$ lies in $P$.} Each $v^{r^{k_n}}$ lies in $P$.
For $\gamma\in\Delta$ take $n$ so large that $r^{k_n}=r_{2'}$ on the cycles of $\gamma$ and of
$vr_{2'}\gamma\in\Delta$; then $v^{r^{k_n}}(\gamma)=r^{-k_n}vr_{2'}\gamma=b(\gamma)$.
For $\gamma\in\Delta^\complement$, $r^{k_n}\gamma\in\Delta^\complement$ is fixed by $v$, so
$v^{r^{k_n}}(\gamma)=\gamma=b(\gamma)$. As $b$ is a bijection, $b\in P$ by
Lemma~\ref{lem:closed}. It is an involution that fixes every point outside
$r_{2'}^{-1}(C)$ and exchanges $r_{2'}^{-1}\alpha$, $r_{2'}^{-1}\beta$ for each $v$-pair $\{\alpha,\beta\}$.

(iii) \emph{The cycles of $\tilde t=tb\in P$.} For a $v$-pair $\Psi=\{\alpha,\beta\}$ put
$\alpha^\pm=r_{2'}^{\pm1}\alpha$, $\beta^\pm=r_{2'}^{\pm1}\beta$ and
$K_\Psi=\Psi\cup r_{2'}\Psi\cup r_{2'}^{-1}\Psi$. By (i) the sets $K_\Psi$ are pairwise disjoint
(a common point puts two $C$-points of distinct $v$-pairs in one $r_{2'}$-orbit), and
$t$ fixes $\alpha$ and exchanges $\alpha^+$, $\alpha^-$; so $t$ and $b$ preserve each $K_\Psi$,
and off their union $b=1$ and $\tilde t=t$ has order at most $2$. If $r_{2'}\alpha\ne\alpha$, then
$\alpha^+\ne\alpha^-$ (the orbit is odd), and $b$ fixes $\alpha$ and $\alpha^+$, since by (i)
neither $r_{2'}\alpha$ nor $r_{2'}^2\alpha$ is in $C$. Hence:
(a) if $r_{2'}$ fixes $\alpha$ and $\beta$, then $K_\Psi=\Psi$ and $\tilde t=(\alpha\ \beta)$ there;
(b) if $r_{2'}$ moves $\alpha$ and fixes $\beta$, then $\tilde t$ fixes $\alpha$ and has the $3$-cycle
$\beta\mapsto\alpha^+\mapsto\alpha^-\mapsto\beta$;
(c) if $r_{2'}$ moves both, then $\tilde t$ fixes $\alpha,\beta$ and has the $4$-cycle
$\alpha^+\mapsto\alpha^-\mapsto\beta^+\mapsto\beta^-\mapsto\alpha^+$
(the points from $\alpha$ and from $\beta$ lie in distinct $r_{2'}$-orbits, by (i)).

\begin{center}
\begin{tabular}{lll}
Case & Assumed for the $v$-pair $\Psi=\{\alpha,\beta\}$ & Where it is closed\\
\hline
(a) & $r_{2'}$ fixes $\alpha$ and $\beta$ & Conclusion: the only case left; it gives $d_i=1$\\
(b) & $r_{2'}$ moves $\alpha$ and fixes $\beta$ & Conclusion: $|\tilde t|$ is a power of two\\
(c) & $r_{2'}$ moves $\alpha$ and $\beta$ & Conclusion: $\tilde t^2=t$ fails at $\alpha^+$\\
\end{tabular}
\end{center}

\emph{Conclusion.} So $\tilde t^{12}=1$, and Theorem~\ref{thm:odd-order}, applied to the
half-involution $t$ and the involution $b$, makes $|\tilde t|$ a power of two: case (b) never
occurs, and $\tilde t^4=1$. Then $\tilde t^2$ fixes $C$ pointwise, so $\tilde t^2\in\pst{C}$. If case (c)
occurs, $\tilde t^2$ is a nonidentity involution, hence $\tilde t^2=t$; but
$\tilde t^2(\alpha^+)=\beta^+\ne\alpha^-=t(\alpha^+)$, since $\beta^+$ and $\alpha^-$ lie in distinct
$r_{2'}$-orbits. So only case (a) occurs: $r_{2'}$ fixes $C$ pointwise. If the $i$-th
$r$-cycle meets $C$ in $\alpha$, then $r^{\lambda_i}\alpha=\alpha$, so $\ell_i\mid\lambda_i$,
hence $\lambda_i=0$, and $0\equiv1\pmod{d_i}$ gives $d_i=1$.
\end{proof}

\begin{lemma}[Two-mask lemma]\label{lem:two-mask}
Assume \textup{(S)} and $|H|\ge4$. Let $x,y\in\Ii$ with complementary pairs
$(t,v)$ for $x$ and $(s,w)$ for $y$. On every connected component $\Theta$ of
the two-matchings graph of $x$ and $y$,
\[
\{t|_\Theta,v|_\Theta\}=\{1,x|_\Theta\}
\quad\text{or}\quad
\{s|_\Theta,w|_\Theta\}=\{1,y|_\Theta\}.
\]
In words, at least one of the two complementary pairs is constant on the
component.
The component can be finite or bi-infinite.
\end{lemma}

If both masks change along an alternating component, a red change followed by a blue change with none between selects a half of $x$ and a half of $y$ whose product has a finite odd orbit through a point fixed by the half of $x$, which Theorem~\ref{thm:oddprimary} excludes.

\begin{proof}
We treat one component at a time; by Lemma~\ref{lem:components}(1) it is a doubled edge, a finite
alternating even cycle or a doubly infinite alternating path. On a doubled edge $\Theta=\{p,x(p)\}$, $x$ and $y$ agree, and each half is the identity or the
transposition there, and exactly one of $t,v$ fixes $p$; so
$\{t|_\Theta,v|_\Theta\}=\{1,x|_\Theta\}$, and likewise
$\{s|_\Theta,w|_\Theta\}=\{1,y|_\Theta\}$: both complementary pairs are constant. For an alternating cycle
or bi-infinite path, take the labelling, the red edges $R_i$, of $x$, the blue edges $B_i$, of $y$, and
the masks $M$ and $N$ of Definition~\ref{def:twomatch}, with $(t,v)$ and $(s,w)$ in place of
$(t_x,v_x)$ and $(t_y,v_y)$; so $M$ and $N$ are the \emph{masks} of $t$ and $s$. On each $R_i$ exactly
one of $t,v$ exchanges $e_i,o_i$ and the other fixes both, so
$M_i=0$ exactly when $v$ does; likewise for $N_i$ and $s,w$ on $B_i$. The half selected by a label is
the one that exchanges the edge.

We argue by contradiction, and first locate two nearby changes of different
colours. Suppose both masks change. Along the alternating component, place a
red-change
event between $R_i$ and $R_{i+1}$ whenever $M_i\ne M_{i+1}$; place a blue-change
event between $B_j$ and $B_{j+1}$ whenever $N_j\ne N_{j+1}$. There are
consecutive change events of different kinds. This follows by choosing two
opposite-kind events at minimum positive distance; an intervening event would
yield a closer opposite-kind pair. The argument works on a cycle after a cyclic
choice of origin, and on a bi-infinite path.

Between these two changes we can choose halves that cut out a short path. Reverse
the orientation or exchange the colours if necessary. Let the first
event be a red change between $R_i$ and $R_{i+1}$, and the next a blue change
between $B_j$ and $B_{j+1}$, with $j\ge i$. Since there are no intervening
changes, the labels are constant on
\[
R_{i+1},\dots,R_{j+1}\quad\text{and}\quad B_i,\dots,B_j .
\]
Choose the red half $t_0$ matching the new red label and the blue half $s_0$
matching the old blue label: $t_0$ is the half of $x$ that exchanges the ends of $R_{i+1}$ (so of
$R_{i+1},\dots,R_{j+1}$, and not of $R_i$, as $M_i\ne M_{i+1}$), and $s_0$ the half of $y$ that
exchanges the ends of $B_j$ (so of $B_i,\dots,B_j$, and not of $B_{j+1}$, as $N_{j+1}\ne N_j$). In
the graph of the two partial matchings
$t_0,s_0$, the consecutive edges
\[
B_i,R_{i+1},B_{i+1},\dots,B_j,R_{j+1}
\]
form a complete finite component: $R_i$ and $B_{j+1}$ are unselected (Figure~\ref{fig:oddpath}). On a cycle of length $2m$
take $i\le j\le i+m-1$. The arc $B_i,R_{i+1},\dots,R_{j+1}$ cannot contain $R_i$, because
$M_{i+1}=\dots=M_{j+1}$ while $M_i\ne M_{i+1}$; so $j+1\le i+m-1$, the arc has $2(j-i+1)\le2m-2$
edges, it is a simple path, and $R_i$ and $B_{j+1}$ ($N_{j+1}\ne N_j=N_i$) are edges of the cycle
outside it. This path
has $2(j-i+1)$ edges and therefore $2(j-i+1)+1$ vertices, an odd number at
least $3$. One endpoint is fixed by $t_0$.

\begin{figure}[!htb]
\centering
\providecolor{figaccent}{RGB}{0,70,160}\begin{tikzpicture}[x=1cm,y=1cm,font=\normalsize,
  every node/.style={inner sep=0pt,outer sep=0pt},
  vopen/.style={circle,draw=black,fill=white,line width=0.6pt,minimum size=0.44cm,inner sep=0pt},
  vfill/.style={circle,draw=figaccent,fill=figaccent,line width=0.6pt,minimum size=0.44cm,inner sep=0pt},
  redheavy/.style={line width=2.2pt,black},
  redlight/.style={line width=0.5pt,black!55},
  blueheavy/.style={line width=1.6pt,black,dash pattern=on 5pt off 3pt},
  bluelight/.style={line width=0.5pt,black!55,dash pattern=on 3pt off 2pt},
  bar/.style={line width=0.8pt,black}]
\coordinate (v0) at (0.000,0);
\coordinate (v1) at (1.250,0);
\coordinate (v2) at (2.500,0);
\coordinate (v3) at (3.750,0);
\coordinate (v4) at (5.000,0);
\coordinate (v5) at (6.250,0);
\coordinate (v6) at (7.500,0);
\coordinate (v7) at (8.750,0);
\coordinate (v8) at (10.000,0);
\draw[bar] (1.250,-0.340) -- (1.250,0.340);
\draw[bar] (8.750,-0.340) -- (8.750,0.340);
\draw[redlight] (v0) -- (v1);
\draw[blueheavy] (v1) -- (v2);
\draw[redheavy] (v2) -- (v3);
\draw[blueheavy] (v3) -- (v4);
\draw[redheavy] (v4) -- (v5);
\draw[blueheavy] (v5) -- (v6);
\draw[redheavy] (v6) -- (v7);
\draw[bluelight] (v7) -- (v8);
\draw[dotted,line width=0.5pt] (v0) -- (-0.650,0);
\draw[dotted,line width=0.5pt] (v8) -- (10.650,0);
\node[vopen] at (v0) {};
\node[vfill] at (v1) {};
\node[vopen] at (v2) {};
\node[vopen] at (v3) {};
\node[vopen] at (v4) {};
\node[vopen] at (v5) {};
\node[vopen] at (v6) {};
\node[vopen] at (v7) {};
\node[vopen] at (v8) {};
\node at (0.000,-1.080) {$e_{i}$};
\node at (1.250,-1.080) {$o_{i}$};
\node at (1.250,-1.580) {$p_{0}$};
\node at (2.500,-1.080) {$e_{i+1}$};
\node at (2.500,-1.580) {$p_{1}$};
\node at (3.750,-1.080) {$o_{i+1}$};
\node at (3.750,-1.580) {$p_{2}$};
\node at (5.000,-1.080) {$e_{j}$};
\node at (5.000,-1.580) {$p_{3}$};
\node at (6.250,-1.080) {$o_{j}$};
\node at (6.250,-1.580) {$p_{4}$};
\node at (7.500,-1.080) {$e_{j+1}$};
\node at (7.500,-1.580) {$p_{5}$};
\node at (8.750,-1.080) {$o_{j+1}$};
\node at (8.750,-1.580) {$p_{6}$};
\node at (10.000,-1.080) {$e_{j+2}$};
\node at (0.625,0.440) {$R_{i}$};
\node at (0.625,0.940) {$M_{i}\!=\!0$};
\node at (1.875,0.440) {$B_{i}$};
\node at (1.875,-0.600) {$N_{i}\!=\!1$};
\node at (3.125,0.440) {$R_{i+1}$};
\node at (3.125,0.940) {$M_{i+1}\!=\!1$};
\node at (4.375,0.440) {$B_{i+1}$};
\node at (4.375,-0.600) {$N_{i+1}\!=\!1$};
\node at (5.625,0.440) {$R_{j}$};
\node at (5.625,0.940) {$M_{j}\!=\!1$};
\node at (6.875,0.440) {$B_{j}$};
\node at (6.875,-0.600) {$N_{j}\!=\!1$};
\node at (8.125,0.440) {$R_{j+1}$};
\node at (8.125,0.940) {$M_{j+1}\!=\!1$};
\node at (9.375,0.440) {$B_{j+1}$};
\node at (9.375,-0.600) {$N_{j+1}\!=\!0$};
\node at (-1.000,0.000) {$\cdots$};
\node at (11.000,0.000) {$\cdots$};
\end{tikzpicture}
\caption{The odd path in the proof of Lemma~\ref{lem:two-mask}, drawn for $j=i+2$, with the labels
  and edges of Definition~\ref{def:twomatch}: red edges $R_l$ (solid) carry the bit $M_l$, blue edges $B_l$
  (dashed) the bit $N_l$. The red change $M_i\ne M_{i+1}$ is followed by the blue change $N_j\ne N_{j+1}$ with
  no change between. The half $t_0$ of $x$ exchanges the ends of $R_{i+1},\dots,R_{j+1}$ and the half $s_0$
  of $y$ those of $B_i,\dots,B_j$ (heavy); $R_i$ and $B_{j+1}$ are unselected. In the graph of $t_0$ and
  $s_0$ the heavy edges form a component, a path with $2(j-i+1)+1=7$ vertices $p_0,\dots,p_6$, and its
  endpoint $p_0=o_i$ (filled) is fixed by $t_0$. The lemma shows that the two masks never both change.}
\label{fig:oddpath}
\end{figure}
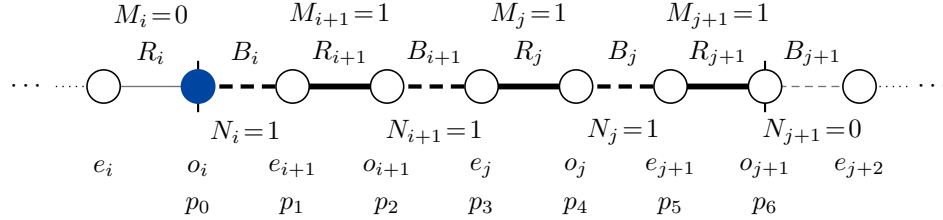

This path yields an odd cycle through the fixed set, which the previous theorem
forbids. The product $t_0s_0$ is a single cycle on the vertices of this
alternating finite path. With vertices $p_0,\dots,p_{2k}$ in path order, its cycle can be
written
\[
(p_0,p_2,\dots,p_{2k},p_{2k-1},p_{2k-3},\dots,p_1),
\]
up to inversion. It is a finite odd orbit through a point fixed by the
half-involution $t_0$, contrary to Theorem~\ref{thm:oddprimary} applied to the
half $t_0$ and the involution $s_0$. Thus the two masks cannot both change. By the first paragraph, $M\equiv1$ gives $t|_\Theta=x|_\Theta$ and $v|_\Theta=1$,
and $M\equiv0$ gives the reverse. Hence $M$ is constant if and only if
$\{t|_\Theta,v|_\Theta\}=\{1,x|_\Theta\}$, likewise for $N$, $(s,w)$ and $y$, and since one of $M,N$ is
constant, the displayed alternative holds.
\end{proof}

\begin{proposition}[Trivial fixed-side commutator]\label{prop:fixed-side}
Assume \textup{(S)} and $|H|\ge4$. Let $x=tv$ and $y=sw$ be full matchings in
$\Ii$ with their complementary halves, $C=\Fix(t)$ and $k=[s,w^t]$. Then $k=1$.
\end{proposition}

\begin{proof}
For $\alpha\in C$, inspect its two-matchings
component $\Theta$. If $(t,v)$ is
constant on $\Theta$, then $t$ must be the identity there (it fixes $\alpha$), so
$k=[s,w]=1$ on that component. Otherwise, Lemma~\ref{lem:two-mask}, applied to
$x$, $y$ and the pairs $(t,v)$, $(s,w)$ on $\Theta$, makes one of $s,w$ the
identity on $\Theta$, and again $k$ fixes $\alpha$. Hence
\[
k\in\pst{C}.
\]

Suppose
$k\ne1$. $\Fix(k)=C$. The commutator of the
involutions $s$ and $w^t$ is inverted by $s$, so $sks=k^{-1}$. Thus $s$
preserves $C$ setwise and, by Corollary~\ref{cor:H}(2), commutes with $t$. It
already commutes with $w$, so it commutes with $w^t$. Therefore $k=1$, a
contradiction.
\end{proof}

\begin{theorem}[Fourth-power bound]\label{thm:fourth-power}
Assume \textup{(S)} and $|H|\ge4$. For all $x,y\in\Ii$, $(xy)^4=1$.
\end{theorem}

\begin{proof}
Let $(t,v)$ and $(s,w)$ be complementary
pairs for $x$ and $y$. Since $t\in P$ and conjugation by an element
of $P$ carries a complementary pair to a complementary pair (Definition~\ref{def:halves}),
$(s^t,w^t)$ is a complementary pair for $y^t\in\Ii$. The complementary centralizer identities
(Corollary~\ref{cor:H}(2)) give
\[
C_P(s)=C_P(w),\qquad C_P(s^t)=C_P(w^t).
\]
From $[s,w^t]=1$ (Proposition~\ref{prop:fixed-side}), all four commutations between
$\{s,w\}$ and $\{s^t,w^t\}$ follow: conjugating by the involution $t$ gives $[s^t,w]=1$; since
$w^t\in C_P(s)=C_P(w)$, $[w,w^t]=1$; and since $s\in C_P(w^t)=C_P(s^t)$, $[s,s^t]=1$.
Consequently $[y,y^t]=1$. Since
$yy^t=(yt)^2$, this gives $(yt)^4=1$, or $[t,t^y]=1$.

$(t^y,v^y)$ is a complementary pair for $x^y\in\Ii$
(Definition~\ref{def:halves}). By Corollary~\ref{cor:H}(2), $C_P(t)=C_P(v)$ and $C_P(t^y)=C_P(v^y)$.
From $[t,t^y]=1$, $t^y\in C_P(t)=C_P(v)$, so $[v,t^y]=1$. Also $t,v\in C_P(t^y)=C_P(v^y)$, so
$[t,v^y]=[v,v^y]=1$. Hence $x=tv$ commutes with $x^y=t^yv^y$, and since $xx^y=(xy)^2$, $(xy)^4=1$.
\end{proof}

What remains to be determined is how the halves act on a single component, which has $2$, $4$ or $8$
points; Section~\ref{ss:types} determines which masks occur on it.

\subsection{Component types}\label{ss:types}

This subsection determines the possible restrictions of two full matchings and
their chosen halves to a single component. Let $x,y\in\Ii$. By Theorem~\ref{thm:fourth-power}, $|xy|$ divides $4$, so Proposition~\ref{prop:two-core}(3) makes $F=\langle t_x,v_x,t_y,v_y\rangle$ a finite $2$-group. By Lemma~\ref{lem:components}(1),(3), every
component is a doubled edge or an alternating cycle of $2m$ vertices with $m\mid4$, and $xy$ has order $m$ on it.

As throughout this section, the red half is $t=t_x$ and the blue half is $s=t_y$; the labelling and the
red and blue masks $M$, $N$ are those of Definition~\ref{def:twomatch}.
The complementary halves are $v=xt$ and $w=ys$. Products are compositions, the rightmost factor acting first
(Section~\ref{ss:conv}). The words used below are
\[
Y=s\,t\,s\,y,\qquad \mathcal{W}=Y\,t\,Y .
\]
The word $\mathcal{W}$ is not a conjugate of $t$: its second factor is $Y$, not $Y^{-1}$.

\begin{lemma}[The word $\mathcal{W}$ on the fixed set]\label{lem:word-fixed-set}
Assume \textup{(S)} and $|H|\ge4$, and let $x,y\in\Ii$, in the notation of
Section~\ref{ss:types}. On every component where $t\ne x$, the word $\mathcal{W}=YtY$ equals $t$.
Therefore $\mathcal{W}$ fixes $\Fix(t)$ pointwise in the full action.
\end{lemma}

\begin{proof}
Each vertex is joined to its $x$-partner and to its $y$-partner, so $x$ and $y$, and by Lemma~\ref{lem:components}(2) their halves, preserve every two-matching component. So $\mathcal{W}$ preserves each
component, and its restriction is the same word in the restrictions. On each component, the two-mask
lemma (Lemma~\ref{lem:two-mask}), applied to $x$, $y$ and the complementary pairs $(t,v)$ and $(s,w)$,
makes $t$ equal to $1$ or $x$ there, or $s$ equal to $1$ or $y$ there; so it leaves three possibilities
when $t\ne x$.
\begin{enumerate}
\item If $t=1$, then $Y=s^2y=y$, and $\mathcal{W}=y^2=1=t$.
\item If $s=1$, then $Y=ty$ and $\mathcal{W}=(ty)t(ty)=t$.
\item If $s=y$, then $Y=yt$ and $\mathcal{W}=(yt)t(yt)=t$.
\end{enumerate}
On a component with $t=x$, the full matching $x$ has no fixed vertex, so there
is no fixed-set condition to check there.
\end{proof}

\begin{lemma}[Commutation of $\mathcal{W}$ with the matching]\label{lem:word-commutation}
In the model of Section~\ref{ss:types}, on an alternating $2m$-cycle with $m\in\{1,2,4\}$
\textup{(}a doubled edge when $m=1$\textup{)} and $t=x$, $[x,\mathcal{W}]=1$ if
and only if the blue
mask is constant, or $m$ is $2$ or $4$ and the blue mask is alternating. This
lemma is purely combinatorial: it uses neither \textup{(S)} nor that
$\langle x,y,s\rangle$ is a $2$-group.
\end{lemma}

\begin{proof}
Label the component as in Definition~\ref{def:twomatch}, with $e_i=2i$ and $o_i=2i+1$, indices modulo $2m$;
on it $t=x$ exchanges $2i$ and $2i+1$, $y$ exchanges $2i+1$ and $2i+2$, and $s$ is the product of the
transpositions $(2i+1\ \,2i+2)$ with $N_i=1$. Replacing the labels by $e'_i=e_{i+k}$, $o'_i=o_{i+k}$ keeps
these descriptions of $x$ and $y$ and replaces $N$ by its rotation $N'_i=N_{i+k}$, while $\mathcal{W}$ and the condition
$[x,\mathcal{W}]=1$ do not depend on the labels; so it suffices to treat one mask in each rotation class. If $N$ is
constant, then $s=1$ or $s=y$ on the component, so $Y=ty$ or $Y=yt$, and $\mathcal{W}=(ty)t(ty)=t$ or
$\mathcal{W}=(yt)t(yt)=t$; in both cases $\mathcal{W}=t=x$, and $[x,\mathcal{W}]=1$. This settles $m=1$, where the mask is a single bit.

For $m=2$ the nonconstant masks form the rotation class of $(N_0,N_1)=(1,0)$. For $m=4$ there are six
rotation classes of masks, those of $(0,0,0,0)$, $(1,1,1,1)$, $(1,0,1,0)$, $(1,0,0,0)$, $(1,1,0,0)$ and
$(1,1,1,0)$; the first two are constant, the third is alternating, and the last three are neither constant
nor alternating. Computing $Y=sxsy$ and then $\mathcal{W}(p)=Y(x(Y(p)))$, that is, applying the nine letters
$y,s,x,s,x,y,s,x,s$ in this order, gives on the component
\[
\begin{array}{c|c|c|c|c}
m & (N_0,\dots,N_{m-1}) & s & Y & \mathcal{W}\\\hline
2 & (1,0) & (1\,2) & (0\,1)(2\,3) & (0\,1)(2\,3)\\
4 & (1,0,1,0) & (1\,2)(5\,6) & (0\,5\,4\,1)(2\,3\,6\,7) & (0\,1)(2\,3)(4\,5)(6\,7)\\
4 & (1,0,0,0) & (1\,2) & (0\,6\,4\,1)(2\,3\,5\,7) & (0\,2\,3\,1)(4\,6\,7\,5)\\
4 & (1,1,0,0) & (1\,2)(3\,4) & (0\,6\,3\,1)(2\,4\,5\,7) & (0\,2\,7\,1)(3\,6\,4\,5)\\
4 & (1,1,1,0) & (1\,2)(3\,4)(5\,6) & (0\,5\,3\,1)(2\,4\,6\,7) & (0\,6\,7\,1)(2\,3\,5\,4)
\end{array}
\]
In the two alternating rows $\mathcal{W}=x$, so $[x,\mathcal{W}]=1$. In the last three rows $\mathcal{W}(1)=0$ and $\mathcal{W}(0)\in\{2,6\}$, so
$\mathcal{W}(x(1))=\mathcal{W}(0)\ne1=x(\mathcal{W}(1))$ and $[x,\mathcal{W}]\ne1$. Hence, for $m\in\{1,2,4\}$, $x$ commutes with $\mathcal{W}$ exactly when
$N$ is constant or alternating, and an alternating mask occurs only for $m\in\{2,4\}$.
\end{proof}

\subsubsection*{The substitutions}
The proof of Theorem~\ref{thm:types} applies Lemmas~\ref{lem:word-fixed-set}
and~\ref{lem:word-commutation} also after the following substitutions: replacing $t$ by $v$; replacing
$s$ by $w$; and exchanging the named pairs $(x,t)$ and $(y,s)$, with the vertex rotation
$\kappa(j)=j+1$ modulo $2m$.

\begin{theorem}[Component types]\label{thm:types}
Assume \textup{(S)} and $|H|\ge4$, and let $x,y\in\Ii$.
\begin{enumerate}
\item $|xy|\in\{1,2,4\}$.
\item The components have $m\in\{1,2,4\}$, and on each component
  either both masks are constant or exactly one is alternating and the other
  constant. In particular, on an eight-point component the masks are either
  both constant or exactly one is alternating and the other constant; a pair
  cannot have both alternating masks.
\item $F=\langle t_x,v_x,t_y,v_y\rangle$ is a finite $2$-group of nilpotency
  class at most two.
\end{enumerate}
\end{theorem}

Parts (2) and (3) are proved one component at a time: two symmetries of the pairs reduce first to $t=x$ on the component, where Lemma~\ref{lem:word-commutation} makes the blue mask constant or alternating, and then to an alternating red mask with a constant blue mask, where the commutators of $x$, $y$, $t$ are checked to be central.

\begin{proof}
Write $(t,v)=(t_x,v_x)$ and $(s,w)=(t_y,v_y)$.
(1) By Theorem~\ref{thm:fourth-power}, $(xy)^4=1$, so $|xy|\in\{1,2,4\}$.

\emph{Two symmetries.} (a) $(v,t)$ is also a complementary pair for $x$; choosing it
replaces $M$ by its complement and leaves $N$ and $F$ unchanged. (b) Exchanging $(x,t,v)$ with
$(y,s,w)$ leaves the components and $F$ unchanged; relabelling a component by $e'_i=o_{i-1}$,
$o'_i=e_i$ satisfies Definition~\ref{def:twomatch} for the pair $(y,x)$, and the new red and
blue masks are $(N_{i-1})_i$ and $M$. Both operations again give two full matchings with
complementary pairs, so every statement used below applies after them, and both preserve the
properties ``constant'' and ``alternating'' of each mask. Since (2) and the bound in (3) are
proved one component at a time, we may apply (a) and (b) separately on each component.

(2) By (1), $|xy|$ divides $4$, so by Lemma~\ref{lem:components}(1),(3) every component is finite, a cycle of $2m$ points with $m$ dividing $|xy|$ (a doubled edge when $m=1$); hence $m\in\{1,2,4\}$. Fix a component
$\Theta$. If $m=1$, both masks are single bits, hence constant. Let $m\in\{2,4\}$. By
Lemma~\ref{lem:two-mask}, $\{t,v\}=\{1,x\}$ or $\{s,w\}=\{1,y\}$ on $\Theta$; by (b) and then (a)
we may assume $t=x$ on $\Theta$, so $M$ is constant. The word $\mathcal{W}=YtY$ lies in $P$. By Lemma~\ref{lem:word-fixed-set}, $\mathcal{W}$ fixes $\Fix(t)$
pointwise, so $\mathcal{W}$ and $t$ lie in $\pst{\Fix t}$, which is cyclic by Corollary~\ref{cor:H}(1) with $t'=t$, since $t$ is a
half-involution by Lemma~\ref{lem:halves-props}(6); hence $[t,\mathcal{W}]=1$. As $s$, $t$ and $y$
preserve $\Theta$, so does $\mathcal{W}$, and restricting to $\Theta$ gives $[x,\mathcal{W}]=1$ there. By
Lemma~\ref{lem:word-commutation}, $N$ is constant or alternating. Since $m\ge2$, a constant mask
is not alternating, so either both masks are constant or exactly one is alternating and the
other constant; since (a) and (b) preserve this, it holds for the original pairs.

(3) By (1), $xy$ has finite order, so $F$ is a finite $2$-group by Proposition~\ref{prop:two-core}(3). $x=tv$ and $y=sw$, where $t,v$ and $s,w$ are
commuting pairs of involutions; so $F=\langle x,y,t,s\rangle$ and $[x,t]=[y,s]=1$. Every element
of $F$ preserves each component, and a permutation that is the identity on every
component is the identity; so it suffices that the restriction $\bar F$ of $F$ to each component
$\Theta$ has class at most two. By (2) and (b), either both masks are constant on $\Theta$, or
$m\in\{2,4\}$, $M$ alternates and $N$ is constant. In the first case $t,v\in\{1,x\}$ and
$s,w\in\{1,y\}$ on $\Theta$, so $\bar F=\langle x,y\rangle$ is dihedral of order $2m$ with $m\mid4$,
of class at most two. In the second, $s,w\in\{1,y\}$, so $\bar F=\langle x,y,t\rangle$. We show that
the commutators of these three generators are central in $\bar F$; then $\bar F/Z(\bar F)$ is generated by
commuting elements, hence is abelian, and $\bar F$ has class at most two.

First, $[x,t]=1$ and $[x,y]=(xy)^2$. For $m=2$ this is $1$. For $m=4$ it is the central involution
of the dihedral group $\langle x,y\rangle$ of order $8$; it maps $e_i\mapsto e_{i-2}$ and
$o_i\mapsto o_{i+2}$ (Lemma~\ref{lem:components}(3)), so it conjugates $t$ to the involution
exchanging $e_j,o_j$ exactly when $M_{j+2}=1$, which is $t$ because $M_{j+2}=M_j$. So $[x,y]$ is
central, and $[x,y]^2=1$.

Next, $[y,t]=yty\,t=t^yt$. For each $i$ with $M_i=1$, $t$ exchanges $e_i,o_i$, so $t^y$ exchanges
$y(e_i)=o_{i-1}$ and $y(o_i)=e_{i+1}$; as $M$ alternates, $M_{i\pm1}=0$. So the supports of $t^y$
and $t$ are the unions of the red pairs with $M_j=0$ and with $M_j=1$, respectively, and are
disjoint. Hence $t^y$ and $t$ commute, and $t^yt$ commutes with $t$ and with $t^y$. Conjugation
by $y$ exchanges $t^y$ and $t$, so fixes $t^yt$. Since $xtx=t$ and $xyx=[x,y]y$,
\[
xt^yx=(xyx)\,t\,(xyx)=([x,y]y)\,t\,([x,y]y)=yty=t^y ,
\]
so $x$ fixes $t^yt$ as well. Thus $[y,t]$ is central in $\bar F$.
\end{proof}

Theorem~\ref{thm:types}(2) still allows mixed components, one mask alternating and the other constant.
On eight points such a component is excluded next, with a supported root (Theorem~\ref{thm:mixed8}); on
four points it remains, and the normal form allows it (Corollary~\ref{cor:normal-form}(1)).

\subsection{Supported-root enclosure and the mixed eight-point component}\label{ss:root}

\begin{lemma}[Dietzmann's lemma {\cite[14.5.7]{Robinson1996}}]\label{lem:dietzmann}
A finite conjugacy-invariant set of torsion elements of a group generates a
finite group.
\end{lemma}

\begin{remark}
Conjugacy invariance cannot be dropped. Two involutions $s,t$ whose product $st$
has infinite order generate the infinite dihedral group, although $\{s,t\}$ is a
finite set of torsion elements; in that group the conjugacy class of $s$ is infinite.
\end{remark}

The enclosure theorem controls the group generated by the four halves together
with a supported root of one of them.

\begin{theorem}[Supported-root enclosure]\label{thm:enclosure}
Assume \textup{(S)} and $|H|\ge4$, and let $x,y\in\Ii$. Let $F=\langle t_x,v_x,t_y,v_y\rangle$, the
finite group of Theorem~\ref{thm:types}(3), and let $t$ be one of its four named half-involutions
(each is one by Lemma~\ref{lem:halves-props}(6)). Let
$C=\Fix(t)$ and $K_C=\pst{C}$, and choose $\zeta\in K_C$ of order four with $\zeta^2=t$, a supported square
root of $t$ (Corollary~\ref{cor:H}(3); it exists since $|H|\ge4$). Set
\[
W_t=[t,F]=\{[t,f]:f\in F\},\qquad V_t=\langle t,W_t\rangle,\qquad L=\langle F,\zeta\rangle .
\]
Then $V_t=\langle t\rangle\times W_t$ is elementary abelian, normal in $F$,
and of order at most $8$; $L$ is a finite $2$-group; $[\zeta,V_t]=1$; and the normal closure
$\langle \zeta^L\rangle$ of $\zeta$ in $L$ is abelian of exponent $4$.
\end{theorem}

The proof of the theorem occupies the two lemmas that follow. The first describes the elementary
abelian group generated by the half and its commutators, and shows that the enlarged group is
finite, of order a power of two. The second shows that the root centralizes that group, and that
the normal closure of the root is abelian of exponent four. The proof of the theorem, placed after
the two lemmas, assembles them.

\begin{lemma}[Finiteness and the $2$-group enclosure]\label{lem:enc-finite}
Assume \textup{(S)} and $|H|\ge4$, let $x,y\in\Ii$, and take $F$, $t$, $C$, $K_C$, $\zeta$, $W_t$, $V_t$
and $L$ in the notation of Theorem~\ref{thm:enclosure}. Put $n=|W_t|$.
\begin{enumerate}
\item $W_t$ is a central subgroup of $F$ of exponent at most two, $n\in\{1,2,4\}$, and
  $t\notin W_t$; $V_t=\langle t\rangle\times W_t$ is elementary abelian of order $2n\le8$
  and normal in $F$.
\item \textup{(a)} $\omega^\zeta\in\{\omega,\omega t\}$ for every $\omega\in V_t$, so $\zeta$ normalizes $V_t$
  and $V_t$ is normal in $L$; and \textup{(b)} $\langle \zeta^L\rangle$ is finite,
  $L=\langle \zeta^L\rangle F$, and $L$ is a finite $2$-group.
\end{enumerate}
\end{lemma}

\begin{proof}
(1) By Theorem~\ref{thm:types}(3), $F$ has class at most two, so its commutators are central, and for
$f,g\in F$ we have $[t,fg]=[t,g][t,f]^g=[t,f][t,g]$ and $[t,f]^2=[t,f]^t[t,f]=[t^2,f]=1$. Hence
$f\mapsto[t,f]$ is a homomorphism $F\to Z(F)$, and its image $W_t$ is a central subgroup of $F$ of
exponent at most two, generated by the commutators of $t$ with the four generators of $F$. Two of
these are trivial: $[t,t]=1$, and $[t,v']=1$ for the complementary half $v'$ of $t$, which is also one
of the four generators. So $W_t$ is generated by two commuting elements of order at most two,
and $n\in\{1,2,4\}$. If $[t,f]=t$, then $t^f=t[t,f]=t^2=1$, which is impossible; so $t\notin W_t$.
As $t$ is an involution commuting with $W_t$, $V_t=\langle t\rangle\times W_t$ is elementary abelian of
order $2n\le8$; it is normal in $F$, since $W_t$ is central and $t^f=t[t,f]\in V_t$ for $f\in F$.

(2)(a) Let $\omega\in V_t$. As $V_t$ is abelian and contains $t$, $\omega$ centralizes $t$, so $\omega$
preserves $C$ and normalizes the subgroup $\langle\zeta\rangle$ of $K_C$ (Corollary~\ref{cor:H}(2)).
Thus $\zeta^\omega\in\langle\zeta\rangle$ and $(\zeta^\omega)^2=t^\omega=t$. The elements of
$\langle\zeta\rangle$ with square $t$ are $\zeta$ and $\zeta^3=\zeta t$, so $\zeta^\omega=\zeta t^\varepsilon$
with $\varepsilon\in\{0,1\}$, and then
\[
\omega^\zeta=\omega\,(\zeta^\omega)^{-1}\zeta=\omega\,t^\varepsilon\zeta^{-1}\zeta=\omega t^\varepsilon .
\]
Since $t\in V_t$, $\zeta$ normalizes $V_t$; as $F$ does too, $V_t$ is normal in $L=\langle F,\zeta\rangle$.

(b) Let $g\in L$ and $u=\zeta^g$. Then $u^2=t^g\in V_t$ by~(a), and $t^g$ is a half-involution
(Definition~\ref{def:halves}), as $g\in L\le P$. Since $\zeta\in K_C$ fixes $C$ pointwise, $u$ fixes
$g^{-1}(C)=\Fix(t^g)$ pointwise, so $u\in\pst{\Fix(t^g)}$, a group of order $|H|$. As $|V_t|\le8$, the class $\zeta^L$ lies in a union of at most eight finite
groups, so it is finite; its elements have order four, so $\langle\zeta^L\rangle$ is finite by
Lemma~\ref{lem:dietzmann}. As $\langle\zeta^L\rangle$ is normal in $L$, $\langle\zeta^L\rangle F$ is a
subgroup containing $F$ and $\zeta$; so $L=\langle\zeta^L\rangle F$, and $L$ is finite. The four
generators of $F$ are half-involutions, so Proposition~\ref{prop:two-core}(1), applied to $L$
and each of them, gives $F\le\Opt(L)$. Then $L/\Opt(L)$ is generated by the image of $\zeta$, so its
order divides four; as $\Opt(L)$ is a $2$-group, so is $L$.
\end{proof}

\begin{lemma}[Root centralization and abelianness]\label{lem:enc-central}
Assume \textup{(S)} and $|H|\ge4$, let $x,y\in\Ii$, and take $F$, $t$, $C$, $K_C$, $\zeta$, $W_t$, $V_t$
and $L$ in the notation of Theorem~\ref{thm:enclosure}. Put $n=|W_t|$.
\begin{enumerate}
\item The root $\zeta$ centralizes $V_t$.
\item \textup{(a)} $t^{L}=t^F=tW_t$, a set of $n$ elements; \textup{(b)} every
  $L$-conjugate $u$ of $\zeta$ centralizes $V_t$, its involutory power
  $\hat{u}=u^2$ lies in $V_t$, and $u$ lies in the cyclic group $\pst{\Fix(\hat{u})}$; and
  \textup{(c)} any two $L$-conjugates of $\zeta$ commute, and $\langle \zeta^L\rangle$ is abelian of
  exponent $4$.
\end{enumerate}
\end{lemma}

\begin{proof}
(1) By Lemma~\ref{lem:enc-finite}, $L$ is a finite $2$-group, $W_t$ is central in $F$, $t\notin W_t$,
$V_t=\langle t\rangle\times W_t$ is elementary abelian, and $z^{\zeta}\in\{z,zt\}$ for every $z\in W_t$.
As $\zeta$ centralizes $t=\zeta^2$, it suffices to show that $\zeta$ centralizes $W_t$. Otherwise choose
$z\in W_t$ with $z^{\zeta}=zt$, and $f\in F$ with $[t,f]=z$. Then $z\ne1$ (as $z^{\zeta}\ne z$) and
$z\ne t$ (as $t\notin W_t$), so $\langle t,z\rangle=\{1,t,z,tz\}$ is a four-group. Since
$[g,h]=g^{-1}g^h$ (Section~\ref{ss:conv}), $t^f=t[t,f]=tz$; and $z^f=z$ as $W_t$ is central in $F$.
So $f$ and $\zeta$ normalize $\langle t,z\rangle$, acting on it by
\[
t^f=tz,\quad z^f=z;\qquad t^{\zeta}=t,\quad z^{\zeta}=zt .
\]
These transformations generate $\GL_2(\F_2)$, of order six, contradicting that
$\langle f,\zeta\rangle\le L$ is a finite $2$-group.

(2)(a) For $f\in F$, $t^f=t[t,f]$, so $t^F=tW_t$, a set of $|W_t|=n$ elements contained in $V_t$.
This set is invariant under conjugation by $F$, and by $\zeta$, which fixes each element of $V_t$
by (1); as $L=\langle F,\zeta\rangle$, it is invariant under $L$, so $t^L\subseteq t^F\subseteq t^L$.

(b) Let $u=\zeta^g$ with $g\in L$. Here $g\in P$, as $F\le P$ and $\zeta\in K_C\le P$. Since
$V_t$ is normal in $L$ (Lemma~\ref{lem:enc-finite}(2)(a)), conjugating $[\zeta,V_t]=1$ by $g$ gives
$[u,V_t]=1$. Next, $\hat{u}=u^2=(\zeta^2)^g=t^g$ lies in $t^L\subseteq V_t$ by (a), and is a
half-involution since $g\in P$. Conjugation by $g$ carries $t$ to $\hat{u}$, hence
$\zeta\in K_C=\pst{\Fix(t)}$ to $u\in\pst{\Fix(\hat{u})}$, a cyclic group.

(c) Let $u,u'$ be $L$-conjugates of $\zeta$; both have order four. By (b), $u$ commutes with
$\hat{u}'\in V_t$, so $u\in C_P(\hat{u}')=\sst{\Fix(\hat{u}')}$, and $u$ normalizes the subgroup
$\langle u'\rangle$ of $\pst{\Fix(\hat{u}')}$ (Corollary~\ref{cor:H}(2), applied to the half-involution
$\hat{u}'$); symmetrically $u'$ normalizes $\langle u\rangle$. Hence
$[u,u']=u^{-1}u^{u'}=(u'^{-1})^{u}u'$ lies in $\langle u\rangle\cap\langle u'\rangle$. If
$\hat{u}\ne\hat{u}'$, this intersection is trivial: a nontrivial subgroup of a cyclic $2$-group
contains its unique involution, so a nontrivial intersection would contain both $\hat{u}$ and
$\hat{u}'$, two distinct involutions of the cyclic group $\langle u\rangle$. If $\hat{u}=\hat{u}'$, then
$u,u'$ lie in the cyclic group $\pst{\Fix(\hat{u})}$ of (b). Either way $u$ and $u'$ commute, so
$\langle\zeta^L\rangle$ is abelian; it is generated by elements of order four and contains $\zeta$,
so its exponent is four.
\end{proof}

\begin{proof}[Proof of Theorem~\ref{thm:enclosure}]
By Lemma~\ref{lem:enc-finite}(1), $V_t=\langle t\rangle\times W_t$ is elementary abelian, normal in
$F$, and of order $2n\le8$, where $n=|W_t|$. By Lemma~\ref{lem:enc-finite}(2)(b), $L$ is a finite
$2$-group. By Lemma~\ref{lem:enc-central}(1), $\zeta$ centralizes $V_t$, and by
Lemma~\ref{lem:enc-central}(2)(c), $\langle \zeta^L\rangle$ is abelian of exponent $4$.
\end{proof}

Plan: apply Theorem~\ref{thm:enclosure} to a half-involution $t$ acting on a mixed component $\Theta$ as $x$
does; $V_t$ is transitive on $\Theta$, so every permutation of $\Theta$ commuting with $V_t$ has square $1$, while
$d=\zeta^{-1}\zeta^s$, for the half $s$ of $y$, preserves $\Theta$, commutes with $V_t$ and has square $[t,s]\ne1$
on $\Theta$.

\begin{theorem}[No mixed eight-cycle: supported root]\label{thm:mixed8}
Assume \textup{(S)} and $|H|\ge4$. For
$x,y\in\Ii$, no eight-point component has one mask alternating and the
other constant. Consequently every eight-point component has both masks
constant.
\end{theorem}

\begin{proof}
Suppose that an eight-point component $\Theta$ is mixed. Interchanging $x$ and $y$, which leaves
$F$ unchanged, we may assume that the red mask is constant and the blue mask alternating. Each half
sends every vertex to itself or to its partner, so a constant red mask makes $t_x$ equal to
$x$ or to $1$ on $\Theta$; as $x=t_xv_x$, one of $t_x,v_x$ equals $x$ on $\Theta$, and we call
it $t$. Put $s=t_y$, $w=v_y$, and label $\Theta$ as in Definition~\ref{def:twomatch}, with $e_i=2i$,
$o_i=2i+1$ $(i\in\Z/4\Z)$. The alternating blue mask makes $s$ exchange exactly the blue edges $B_i$
with $i$ of one parity and fix the other points of $\Theta$; replacing each $e_i,o_i$ by
$e_{i+1},o_{i+1}$ if necessary (again a labelling), these are $B_0,B_2$. As $y=sw$ and $[s,w]=1$, on $\Theta$
\[
t=x=(0\,1)(2\,3)(4\,5)(6\,7),\qquad s=(1\,2)(5\,6),\qquad w=sy=(3\,4)(7\,0),
\]
\[
[t,s]=(0\,3)(1\,2)(4\,7)(5\,6),\qquad [t,w]=(0\,7)(1\,6)(2\,5)(3\,4).
\]
Let $\zeta$, $W_t$, $V_t$ and $L$ be as in Theorem~\ref{thm:enclosure} for this $t$, one of the four
named half-involutions of $F$.

\emph{$V_t$ is transitive on $\Theta$.} $F$ preserves $\Theta$, and so does $V_t\le F$.
Since $s,w\in F$, we have $[t,s],[t,w]\in W_t\le V_t$. The $V_t$-orbit of $0$ contains $1=t(0)$,
$3=[t,s](0)$, $7=[t,w](0)$, $6=[t,w](1)$, $4=[t,w](3)$, $2=t(3)$ and $5=t(4)$, so it is $\Theta$.
Hence $\Theta=V_t\gamma$ for every $\gamma\in\Theta$.

\emph{Every $\xi\in\Sym(\Theta)$ commuting with $V_t|_\Theta$ satisfies $\xi^2=1$.} Fix
$\gamma_0\in\Theta$ and $g_0\in V_t$ with $\xi\gamma_0=g_0\gamma_0$. Every $\gamma\in\Theta$ is
$g\gamma_0$ with $g\in V_t$, and $\xi\gamma=g\xi\gamma_0=gg_0\gamma_0=g_0g\gamma_0=g_0\gamma$, since
$V_t$ is abelian. So $\xi=g_0|_\Theta$, and $g_0^2=1$ as $V_t$ is elementary abelian
(Theorem~\ref{thm:enclosure}).

\emph{$\zeta(\Theta)$ is an $F$-orbit.} As $x,y\in F$ move along the edges, the $F$-orbits are the components, which have at most
eight points. Since $[\zeta,V_t]=1$ (Theorem~\ref{thm:enclosure}), for $\gamma\in\Theta$ we
have $\zeta(\Theta)=\zeta(V_t\gamma)=V_t\,\zeta\gamma$, which lies in the $F$-orbit of $\zeta\gamma$;
having eight points, it is that orbit.

\emph{Contradiction.} Put $d=\zeta^{-1}\zeta^s=\zeta^{-1}s\zeta s$. As $s\in F$ preserves $\Theta$
and $\zeta(\Theta)$, $d(\Theta)=\Theta$. As $s$ normalizes $V_t$ and $\zeta$ centralizes it
(Theorem~\ref{thm:enclosure}), $\zeta^s$ centralizes $V_t^s=V_t$, and so does $d$; hence
$d|_\Theta$ commutes with $V_t|_\Theta$, and $(d|_\Theta)^2=1$ by the previous claims. But
$\zeta^s\in\langle\zeta^L\rangle$ since $s\in F\le L$, and $\langle\zeta^L\rangle$ is abelian
(Theorem~\ref{thm:enclosure}), so $\zeta$ and $\zeta^s$ commute and
$d^2=\zeta^{-2}(\zeta^2)^s=t\,t^s=[t,s]$, which is nontrivial on $\Theta$. Finally,
Theorem~\ref{thm:types}(2) leaves, on an eight-point component, only both masks constant or the
mixed case just excluded.
\end{proof}

\begin{corollary}[Normal form]\label{cor:normal-form}
Assume \textup{(S)} and $|H|\ge4$. For $x,y\in\Ii$, every
component of their two-matchings graph has size $2$, $4$ or $8$. For any chosen complementary
pairs:
\begin{enumerate}
\item on a four-point component, at least one pair is constant (its
  restrictions are the identity and the full matching);
\item on an eight-point component, both pairs are constant.
\end{enumerate}
\end{corollary}

\begin{proof}
Every component has $2$,
$4$ or $8$ points (Theorem~\ref{thm:fourth-power} and Lemma~\ref{lem:components}(1),(3)). Part (1) is Lemma~\ref{lem:two-mask}. Part (2) is Theorem~\ref{thm:mixed8}
together with Theorem~\ref{thm:types}(2), which excludes both masks alternating. The masks of these
two theorems are those of the chosen halves $t_x$, $t_y$ of complementary pairs for $x$ and $y$ that
are not otherwise restricted, so both apply to any chosen complementary pairs; and on a component
$\Theta$, $t_x$ swaps $e_i,o_i$ exactly when $M_i=1$ and fixes them otherwise, and $v_x=xt_x$, so $M$
is constant exactly when $\{t_x|_\Theta,v_x|_\Theta\}=\{1,x|_\Theta\}$, that is, when the pair of $x$
is constant on $\Theta$. Likewise $N$ is constant exactly when the pair of $y$ is constant on
$\Theta$.
\end{proof}

The normal form still allows an eight-point component on which both masks are constant, and there $x$
and $y$ do not commute: $xy$ has order four on it (Lemma~\ref{lem:components}(3)), so $[x,y]=(xy)^2$ is
not the identity there. This is the obstruction that remains, and Section~\ref{s:sync} removes it by
showing that all halves commute (Theorem~\ref{thm:pairclosure}).

\section{The branch \texorpdfstring{$|H|\ge4$}{|H|>=4}: synchronization and the block}\label{s:sync}

This section turns the normal form of the previous section into the
commutation of all conjugates of the half-flip, and then into a block. The first
two results (Lemmas~\ref{lem:sync} and \ref{lem:commutator-criterion}) concern a single side
and assume only (S); they hold for every $H$.

\begin{lemma}[One-sided synchronization]\label{lem:sync}
Assume \textup{(S)}. Define
\[
\Ea=\{g\in P:g(a)\in\{a,ca\}\text{ for every }a\in A\}.
\]
Then
\begin{equation}\label{eq:EA}
\Ea=\langle v\rangle\times\pst{A}\cong C_2\times H .
\end{equation}
Every $g\in\Ea$ has one of the two restrictions
\begin{equation}\label{eq:EArestr}
g|_A=1_A\qquad\text{or}\qquad g|_A=\rho|_A .
\end{equation}
\end{lemma}

\begin{proof}
Let $g\in\Ea$ and write $g(a)=c^{\varepsilon(a)}a$ with $\varepsilon(a)\in\F_2$ $(a\in A)$.
Let $a\in A$ and $u\in U$. Since $g\in P$ and $g(A)\subseteq A$, Lemma~\ref{lem:adj} gives
\[
\one_X(u)=E(ua,a)=E(g(ua),g(a))=\one_X\bigl(c^{\varepsilon(ua)+\varepsilon(a)}u\bigr).
\]
As $u\in X\symd cX$ and $c^2=1$, exactly one of $u$ and $cu$ lies in $X$, so this forces
$\varepsilon(ua)=\varepsilon(a)$. Thus $U$ lies in the period subgroup $\Pi(\varepsilon)$, which
is a subgroup of $A$: for $h,h'\in\Pi(\varepsilon)$ and $z\in A$,
$\varepsilon(hh'^{-1}z)=\varepsilon(h'^{-1}z)=\varepsilon(h'h'^{-1}z)=\varepsilon(z)$. Since $U$ is
infinite, Lemma~\ref{lem:prufer}(1) gives $\Pi(\varepsilon)=A$, so $\varepsilon$ is constant, with
value $\varepsilon$ say, and $g|_A=\rho^\varepsilon|_A$; this is \eqref{eq:EArestr}.

By \eqref{eq:tvprops}, $v(a)\in\Fix(t)=A$ for $a\in A$ (as $[t,v]=1$), so $v(a)=tv(a)=\rho(a)$,
that is, $v|_A=\rho|_A$. Hence $v^\varepsilon g$ fixes $A$ pointwise, and Lemma~\ref{lem:H} gives
$g=v^\varepsilon\beta_h$ for some $h\in H$. Conversely, every such element belongs to $\Ea$. The two
factors commute and meet trivially, because they have disjoint supports ($v$ fixes $B$, $\beta_h$
fixes $A$) and $v|_A=\rho|_A\ne1$; with $v^2=1$ and \eqref{eq:H} this proves \eqref{eq:EA}.
\end{proof}

\begin{lemma}[Half-commutator criterion]\label{lem:commutator-criterion}
Assume \textup{(S)}. Let $s\in P$ be an involution with at least one fixed
point. If
\begin{equation}\label{eq:kAc}
[t,s](a)\in\{a,ca\}\qquad(a\in A),
\end{equation}
then $[t,s]=1$.
\end{lemma}

\begin{proof}
Put $k=[t,s]$. By Lemma~\ref{lem:sync},
$k\in\Ea=\langle v\rangle\times\pst{A}$. This group is abelian and contains $t$,
so $tkt=k$. On the other hand, for two involutions, $k=tsts$ and
$tkt=stst=k^{-1}$. Combining the two, $k^2=1$.

An element of order at most two in that group has only four possible values,
because the cyclic factor has a single involution. The finite cyclic $2$-group
$H$ has the unique nonidentity involution $c$. Equation~\eqref{eq:EA} leaves only
\[
k\in\{1,t,v,tv\}=\{1,t,v,\rho\}.
\]
Each nonidentity value is excluded by the identity $t^s=tk$. If $k=t$, then $t^s=1$, impossible because
$t\ne1$. If $k=v$, then $t^s=\rho$; this is impossible because $t$ fixes all of
$A$, whereas $\rho$ is fixed-point-free. If $k=\rho$, then $t^s=v$; taking fixed
sets gives $s(A)=B$, and since $s^2=1$, it also gives $s(B)=A$, making $s$
fixed-point-free. That contradicts the hypothesis on $s$, so $k=1$.
\end{proof}

\begin{lemma}[The half-commutator is one-sided]\label{lem:onesided}
Assume \textup{(S)} and $|H|\ge4$. Suppose $x,y\in\Ii$. Choose one half of $x$ and
one half of $y$. After conjugation in $P$, arrange that $x=\rho$ and the first
half is $t$; call the second half $s$. Then $k=[t,s]$ satisfies $k(a)\in\{a,ca\}$
for every $a\in A$.
\end{lemma}

The conjugation is possible even when the selected half of $x$ is the
complementary one.

\begin{proof}
We work component by component in the normal form of
Corollary~\ref{cor:normal-form}, by which components have $2$, $4$ or $8$ vertices.

\emph{Two vertices.} Every restriction is the identity or the sole
transposition; these commute, so $k=1$ there.

\emph{Four vertices.} If $t=1$ there, then $k=1$. If $t=\rho$ there, that component has
no point in $A=\Fix(t)$. Otherwise $t$ is exactly one of the two transpositions
of $\rho$, and by Corollary~\ref{cor:normal-form}(1), $s$ is the identity or its full
matching $y$. In the first case $k=1$. In the second, $y$ interchanges the two
$\rho$-edges, so $t^s$ is the complementary transposition and
\[
k=t\,t^s=\rho
\]
on that component. In particular its action on the points of $A$ is
multiplication by $c$.

\emph{Eight vertices.} Corollary~\ref{cor:normal-form}(2) makes $t$ constant. If the
component meets $A$, then $t$ is the identity on the entire component, so $k=1$
there. Otherwise there is no point of $A$ to check.

In every case $k$ sends a point of
$A$ either to itself or to its $\rho$-partner, as claimed.
\end{proof}

\begin{theorem}[Commuting halves]\label{thm:pairclosure}
Assume \textup{(S)} and $|H|\ge4$. For any $x,y\in\Ii$, all four complementary halves commute. In particular,
\[
[x,y]=1,\qquad \langle t_x,v_x,t_y,v_y\rangle\cong C_2^e,\ e\le4 .
\]
\end{theorem}

\begin{proof}
For any chosen cross-pair of halves normalize as in
Lemma~\ref{lem:onesided}, and put $k=[t,s]$. By Lemma~\ref{lem:onesided}, $k$
satisfies \eqref{eq:kAc}; the element $s$ is an involution in
$P$ and
has a nonempty fixed set, so Lemma~\ref{lem:commutator-criterion} applies and gives $k=1$.
The argument applies to every chosen cross-pair of halves, and complementary
halves already commute; so all four half-generators commute. Four commuting
involutions generate an elementary abelian group of order at most $2^4$. In
particular $x=t_xv_x$ and $y=t_yv_y$ commute.
\end{proof}

With $[x,y]=1$ for all $x,y\in\Ii$, what remains in this branch is a nontrivial block; the corollary below
finds it among the orbits of the group generated by the half-involutions.

\begin{corollary}[The block]\label{cor:block}
Assume \textup{(S)} and $|H|\ge4$, so that the fourth-power bound holds
(Theorem~\ref{thm:fourth-power}). Then every pair of conjugate
half-involutions commutes (Theorem~\ref{thm:pairclosure}). The subgroup $\langle t^P\rangle$ is elementary abelian and normal in $P$, and its vertex orbit $\langle t^P\rangle(1)$
satisfies
\begin{equation}\label{eq:block}
\{1,c\}\subseteq\langle t^P\rangle(1)\subseteq A .
\end{equation}
In particular $P$ is imprimitive and $c^{P_1}\subseteq A$; here $P_1$ is the stabilizer of the
vertex $1$ and $c^{P_1}$ the orbit of $c$ under it.
\end{corollary}

\begin{proof}
Every two conjugates of $t$ belong to
complementary pairs for some full matchings $x,y\in\Ii$. Theorem~\ref{thm:fourth-power}
makes $xy$ of finite order, and Theorem~\ref{thm:pairclosure} then makes those
halves commute. All generators of $\langle t^P\rangle$ therefore commute; they are involutions,
so $\langle t^P\rangle$ is elementary abelian. It is normal by definition.

For $n\in\langle t^P\rangle$, using $tn=nt$ and $t(1)=1$,
\[
t(n(1))=n(t(1))=n(1).
\]
Hence $\langle t^P\rangle(1)\subseteq\Fix(t)=A$. By Lemma~\ref{lem:halves-props}(3), $v=t^{R(b)}$ with $R(b)\in P$, so $v\in\langle t^P\rangle$ and $c=v(1)\in\langle t^P\rangle(1)$, which proves
\eqref{eq:block}.

The orbits of a normal
subgroup form a $P$-invariant partition. Its class of $1$ contains two distinct
points and is contained in the proper subset $A$, so the partition is
nontrivial. Any element of $P_1$ preserves that class, yielding
$c^{P_1}\subseteq\langle t^P\rangle(1)\subseteq A$.
\end{proof}

\section{The branch \texorpdfstring{$|H|=2$}{|H|=2}}\label{s:two}

Throughout this section (S) holds and $H=\Pi(T)=\langle c\rangle$. For
$x\in\Ii$ we write $(t_x,v_x)$ for a complementary pair of $x$ in
the notation of Definition~\ref{def:halves}, with
$C_x=\Fix(t_x)$ and $C_x^\complement=\Fix(v_x)$; for $x=\rho$ the pair is $(t,v)$ with
$C_\rho=A$. Likewise $(t_y,v_y)$ and $C_y=\Fix(t_y)$ for $y\in\Ii$.

\emph{Standing facts.} The following facts are proved before they are used, and proofs in this section
cite them by their tags. Here $x,y\in\Ii$ have complementary pairs $(t_x,v_x)$ and $(t_y,v_y)$.
\begin{itemize}
\item $t_x$ and $v_x$ are involutions in $P$ and half-involutions, and $(v_x,t_x)$ is
  also a complementary pair for $x$; $x=t_xv_x$ and $[t_x,v_x]=1$; $C_x=\Fix(t_x)$ and $C_x^\complement=\Fix(v_x)$
  are complementary; and each half sends every vertex to itself or to its $x$-partner
  (Lemma~\ref{lem:halves-props}(4)--(6)).
\item $x$ is a fixed-point-free involution; every right translation of $\Qg$ lies in
  $P$, in particular $\rho$, so that the $P$-conjugacy class $\Ii=\rho^P$ lies in $P$
  (Definition~\ref{def:standing}, Lemma~\ref{lem:halves-props}(1),(2)).
\item $x$, $y$ and every half of $x$ or of $y$ preserve each two-matching component of
  $x$ and $y$; each component is a doubled edge, a finite alternating even cycle or a doubly infinite
  alternating path; and, labelled as in Definition~\ref{def:twomatch}, $xy(e_i)=e_{i-1}$ and
  $xy(o_i)=o_{i+1}$, so that $xy$ has order $m$ on a cycle of $2m$ vertices and infinite order on a doubly
  infinite path (Definition~\ref{def:twomatch}, Lemma~\ref{lem:components}).
\end{itemize}

This branch runs parallel to the branch $|H|\ge4$ of Sections~\ref{s:big} and~\ref{s:sync},
with a translation square root in place of a supported one: supported square roots exist only
when $|H|\ge4$ (Corollary~\ref{cor:H}(3)). The bound $(xy)^4=1$ (Theorem~\ref{thm:rs4})
corresponds to the fourth-power bound (Theorem~\ref{thm:fourth-power}); the mixed eight-point
component is excluded by a translation root (Proposition~\ref{prop:mixed8two}) as it was by a
supported root (Theorem~\ref{thm:mixed8}); the full matchings commute
(Theorem~\ref{thm:commute2}) as they did there (Theorem~\ref{thm:pairclosure}); and the affine
lemma (Lemma~\ref{lem:affine}) ends this branch where the block (Corollary~\ref{cor:block})
ended the other.

The robust-pair theorem (Theorem~\ref{thm:twelve-words}) concerns two interleaved $0$--$1$ sequences.
The two sequences are the masks of $x$ and $y$. Along a component of the graph formed by their
pairs, they record which pairs the corresponding conjugates of $t$ swap. The masks have period two
(Corollary~\ref{cor:period2}). In every interleaved sequence without period two, the robust-pair
theorem finds two vertices which, where they stay distinct, lie in distinct pairs of a conjugate
of $\rho$ and have the same neighbourhood across the corresponding conjugates of the sides
(Lemma~\ref{lem:robust-pair}(3)). Distinct cross-neighbourhoods forbid this
(Lemma~\ref{lem:cross-nbhd}(3)). On eight-point components a parity count is also needed, and
one component is left open: eight points, one mask constant and the other alternating. The right
translation $R(d)$ by an element $d\in A$ of order four, a square root of $\rho$, excludes it.
Once $x$ is conjugated to $\rho$, the translation preserves that component by the diagonal
alignment lemma (Lemma~\ref{lem:diag}). It conjugates $y$ to a matching whose product with $y$
has order three there, against $(xy)^4=1$ (Proposition~\ref{prop:mixed8two}).

\subsection{Distinct cross-neighbourhoods}\label{ss:cross-nbhd}

\begin{lemma}[Distinct cross-neighbourhoods]\label{lem:cross-nbhd}
Assume \textup{(S)} and $H=\langle c\rangle$.
\begin{enumerate}
\item For $a,z\in A$, $E(a,zb)=\one_T(az)$. Consequently distinct $\rho$-pairs
  on either side have distinct cross-neighbourhoods: equality would give a
  period of $T$ outside $\langle c\rangle$.
\item $\pst{A}=\langle t\rangle$ and $\pst{B}=\langle v\rangle$.
\item For every $x\in\Ii$ with complementary pair $(t_x,v_x)$,
  $\pst{C_x}=\langle t_x\rangle$ and $\pst{C_x^\complement}=\langle v_x\rangle$. In
  particular, every element of $P$ preserving $C_x$ setwise centralizes $t_x$, $v_x$
  and $x$. Distinct $x$-pairs inside $C_x$ have distinct neighbourhoods in
  $C_x^\complement$, distinct $x$-pairs inside $C_x^\complement$ have distinct neighbourhoods in
  $C_x$, and the two points of an $x$-pair on either side have the same
  neighbourhood on the other side.
\end{enumerate}
\end{lemma}

\begin{proof}
Part (1) is \eqref{eq:adj}: the cross-neighbourhood of $zb$ in $A$ is $z^{-1}T$
and that of $a$ in $B$ is $(a^{-1}T)b$; $z^{-1}T=z'^{-1}T$ means $z'z^{-1}\in\Pi(T)=\langle c\rangle$,
and $a^{-1}T=a'^{-1}T$ means $a'a^{-1}\in\Pi(T)$. So two vertices on the same side have equal
cross-neighbourhoods exactly when they differ by $\Pi(T)=\langle c\rangle$, that is, lie in one
$\rho$-pair (as in \eqref{eq:fibres} with $H=\langle c\rangle$): $\rho(zb)=zbc=czb$, since $c=b^2$
is central, and $\rho(a)=ac$. (2) is
Lemma~\ref{lem:H} and Corollary~\ref{cor:H} with $H=\langle c\rangle$: Lemma~\ref{lem:H} gives
$\pst{A}=\{\beta_1,\beta_c\}=\{1,t\}$, and Corollary~\ref{cor:H}(1), applied to the half-involution $t$
with complementary half $v$, $C=\Fix(t)=A$ and
$C^\complement=B$, makes $\pst{B}$ cyclic of order $|H|=2$ with involution $v$, so $\pst{B}=\langle v\rangle$.
(3) follows by conjugating (1) and (2) by $g$, where $(t_x,v_x,x)=(t^g,v^g,\rho^g)$ for some
$g\in P$: the graph automorphism $g^{-1}\in P$ maps $A$, $B$ and the $\rho$-pairs onto
$C_x=g^{-1}(A)$, $C_x^\complement=g^{-1}(B)$ and the $x$-pairs. So, by (1), distinct $x$-pairs in $C_x$
(respectively $C_x^\complement$) have distinct neighbourhoods in $C_x^\complement$ (respectively $C_x$); and by (2),
$\pst{C_x}=g^{-1}\pst{A}g=\langle t_x\rangle$ and $\pst{C_x^\complement}=g^{-1}\pst{B}g=\langle v_x\rangle$.
The two points of a $\rho$-pair on either side have the same cross-neighbourhood, since
$E(ac,zb)=E(a,czb)=\one_T(acz)=\one_T(az)$ ($cT=T$); transported by $g^{-1}$, the two points of an
$x$-pair on either side have the same neighbourhood on the other side.
Finally, by Corollary~\ref{cor:H}(2), applied to the half-involution $t_x$ with complementary half
$v_x$, an element of $P$ preserving $C_x$ setwise centralizes $t_x$ and $v_x$, hence also
$x=t_xv_x$ (Lemma~\ref{lem:halves-props}(5)).
\end{proof}

\subsection{Period two}\label{ss:period-two}

This subsection shows that the masks of two full matchings have period two on
every component. For $x,y\in\Ii$, the components of the two-matchings graph of $x$ and $y$
(Definition~\ref{def:twomatch}) are described by the standing facts. Label a component as in
Definition~\ref{def:twomatch}. For chosen halves $t_x$ of $x$ and $t_y$ of $y$, the masks of that definition
are written $M(i)$ and $N(i)$. The conjugates of the halves by powers of $q=xy$ are described in
Lemma~\ref{lem:robust-pair}(1).

Two periodicity conditions occur in this subsection, one for the interleaved sequence $\mu$ of
Lemma~\ref{lem:robust-pair}(1), given by $\mu(2i)=M(i)$ and $\mu(2i+1)=N(i)$, and one for the masks:
\[
\mu(p+2)=\mu(p)\quad\text{for all }p,\qquad\qquad M(i+2)=M(i)\ \text{and}\ N(i+2)=N(i)\quad\text{for all }i.
\]
The first holds exactly when both masks are constant; the second also allows alternating masks.
Theorem~\ref{thm:twelve-words} concerns the first, written ``$\mu$ is $2$-periodic''.
Corollary~\ref{cor:period2} asserts the second, \eqref{eq:period2}, on every component, and the first on
bi-infinite components and on cycles with $m\ge3$, $m\ne4$.

A robust pair, defined in Definition~\ref{def:grid}, is what links the grid to the masks. Interleaving the red and blue masks of a
component gives one binary sequence, and the grid joins an even vertex to an odd vertex exactly when
the value of the sequence at their cell equals the bit of the cell's row. A robust pair is a pair of
distinct cells in one row, both carrying the bit of that row, that stay joined by a rook path
whatever bits the other rows carry. Figure~\ref{fig:grid} shows the grid of Definition~\ref{def:grid}
for the word $W_A$ of Table~\ref{tab:words}, with its robust pair and the rook paths of
Appendix~\ref{app:paths}. On a component the row bits are not free: for a vertex on one
side, the bit of a row records whether the chosen half on that diagonal fixes the vertex, and every
edge of the resulting grid joins two vertices equally adjacent to it. A robust pair whose two
targets stay distinct on the component would therefore give two vertices outside that side, in
distinct pairs of one full matching, with the same neighbourhood in the side; distinct
cross-neighbourhoods forbid this. The theorem below shows that every sequence without period two has
a robust pair, and the corollary after it concludes that the masks have period two on every
component.

\begin{definition}[The grid graph]\label{def:grid}
Let $\mu,\sigma:\Z\to\{0,1\}$. The graph $G(\mu,\sigma)$ is the bipartite graph
with vertices $e_j$ and $o_i$, for all integer indices, in which
\begin{equation}\label{eq:grid}
e_j\sim o_i\quad\Longleftrightarrow\quad \mu(i+j)=\sigma(i-j).
\end{equation}
We call the edge $e_jo_i$ the \emph{cell} $(p,d)=(i+j,\,i-j)$; its \emph{row} is
$d$, and $p\equiv d\pmod2$. Two cells are \emph{rook-connected} if the
corresponding edges lie in the same component of $G(\mu,\sigma)$. For $d_0\in\Z$ and a bit
$\varepsilon$, two cells $(p,d_0)$ and $(p',d_0)$ with $p\ne p'$ and $\mu(p)=\mu(p')=\varepsilon$ form a
\emph{robust pair} (in row $d_0$, with bit $\varepsilon$) if they are rook-connected in $G(\mu,\sigma)$
for every $\sigma:\Z\to\{0,1\}$ with $\sigma(d_0)=\varepsilon$.
\end{definition}

The indices of the edge in cell $(p,d)$ are recovered as $i=(p+d)/2$ and $j=(p-d)/2$. Two edges of
$G(\mu,\sigma)$ share a vertex exactly when they have the same $j$ (the vertex $e_j$) or the same $i$ (the
vertex $o_i$). This incidence is a rook move in the coordinates $(i,j)$ and a diagonal move in the
coordinates $(p,d)$, where it keeps $p-d$ or $p+d$ fixed.

\begin{lemma}[Mask-to-grid lemma]\label{lem:robust-pair}
Assume \textup{(S)} and $H=\langle c\rangle$. Let $x,y\in\Ii$, with chosen halves
$t_x,t_y$ and masks $M,N$ on a component $\Theta$ of their two-matchings graph,
labelled as in Definition~\ref{def:twomatch}, with masks $M,N$ as in Section~\ref{ss:period-two}. If $\Theta$ is doubly infinite, indices lie in $\Z$ and $\pi$
is the identity; if $\Theta$ is a cycle of $2m$ vertices, the masks are lifted
periodically to $\Z$ and $\pi$ reduces indices modulo $m$; a doubled edge is the case $m=1$.
\begin{enumerate}
\item Interleave $\mu(2i)=M(i)$, $\mu(2i+1)=N(i)$. For every $d\in\Z$ the
  matching $\pi e_j\leftrightarrow\pi o_{j+d}$, the \emph{diagonal} $d$, is the
  restriction to $\Theta$ of a
  conjugate of $x$ ($d$ even) or of $y$ ($d$ odd) in $\Ii$, and its chosen half,
  $q^{d/2}t_xq^{-d/2}$ for $d$ even and $q^{(d+1)/2}t_yq^{-(d+1)/2}$ for $d$ odd, where $q=xy$,
  swaps the pair $(\pi e_j,\pi o_{j+d})$ exactly when $\mu(2j+d)=1$.
\item Let $h$ be the chosen half on diagonal $d_0$ or its complement, with fixed
  set $C$, and let $\varepsilon=1$ or $0$ accordingly. For $\alpha\in C$ define
  $\sigma_\alpha(d)=1$ if the chosen half on diagonal $d$ fixes $\alpha$ and
  $\sigma_\alpha(d)=0$ otherwise. Then $\sigma_\alpha(d_0)=\varepsilon$, and the
  adjacency column $f_\alpha(\beta)=E(\alpha,\beta)$ satisfies
  $f_\alpha(\pi e_j)=f_\alpha(\pi o_i)$ whenever $\mu(i+j)=\sigma_\alpha(i-j)$.
\item Let $p\ne p'$ with $p\equiv p'\equiv d_0\pmod2$ and
  $\mu(p)=\mu(p')=\varepsilon$, put $j=(p-d_0)/2$ and $j'=(p'-d_0)/2$, and suppose
  $\pi e_j\ne\pi e_{j'}$. If the cells $(p,d_0)$ and $(p',d_0)$ are
  rook-connected in $G(\mu,\sigma)$ for every $\sigma$ with $\sigma(d_0)=\varepsilon$,
  then $\pi e_j$ and $\pi e_{j'}$ lie in $C^\complement$, in distinct pairs of the full
  matching of $h$, and have the same neighbourhood in $C$. This contradicts
  Lemma~\ref{lem:cross-nbhd}(3). Hence no such pair of cells exists.
\end{enumerate}
\end{lemma}

The proof conjugates the chosen halves of $x$ and $y$ by powers of $q=xy$ to obtain every diagonal, and uses the halves that fix $\alpha$ to carry adjacency to $\alpha$ across each edge of $G(\mu,\sigma_\alpha)$; part (3) then follows along a rook path.

\begin{proof}
Let $(t_x,v_x)$ and $(t_y,v_y)$ be the complementary pairs of the chosen halves, and put $q=xy\in P$. By Lemma~\ref{lem:components}(3), $q$ preserves $\Theta$, with $q^ke_i=e_{i-k}$ and
$q^ko_i=o_{i+k}$ for $k\in\Z$ (indices read through $\pi$).

(1) For $d=2k$ put $(g_d,g_d',z_d)=(q^kt_xq^{-k},q^kv_xq^{-k},q^kxq^{-k})$; for $d=2k-1$ put the same
with $t_y,v_y,y$. If $(t_x,v_x,x)=(t^g,v^g,\rho^g)$ with $g\in P$, then
$(g_d,g_d',z_d)=(t^{gq^{-k}},v^{gq^{-k}},\rho^{gq^{-k}})$ with $gq^{-k}\in P$; so $z_d\in\Ii=\rho^P$ and
$(g_d,g_d')$ is a complementary pair for $z_d$ (Definition~\ref{def:halves}), whose chosen half $g_d$ is
the half named in (1); likewise for $y$. For $d=2k$ and $j=i-k$, $z_d$ sends $e_j=q^ke_i$ to
$q^kx(e_i)=q^ko_i=o_{j+d}$, and $g_d$ swaps these two points exactly when $t_x$ swaps $e_i,o_i$, that is,
when $M(i)=\mu(2i)=\mu(2j+d)$ is $1$. For $d=2k-1$ and $j=i+1-k$, $z_d$ sends $e_j=q^ke_{i+1}$ to
$q^ky(e_{i+1})=q^ko_i=o_{j+d}$, and $g_d$ swaps these two points exactly when $t_y$ swaps $o_i,e_{i+1}$,
that is, when $N(i)=\mu(2i+1)=\mu(2j+d)$ is $1$. $z_d$ exchanges $\pi e_j$ and
$\pi o_{j+d}$ for every $j$; so it preserves $\Theta$, and its restriction to $\Theta$ is the diagonal $d$.

For a bit $\beta$ let $h_{d,\beta}=g_d$ if $\beta=1$ and $h_{d,\beta}=g_d'$ if $\beta=0$. We claim:
($\ast$) $h_{d,\beta}$ swaps $\pi e_j$ and $\pi o_{j+d}$ if and only if $\mu(2j+d)=\beta$.
For $\beta=1$ this is (1). The $z_d$-partner of $\pi e_j$ is $\pi o_{j+d}$; exactly one of $g_d,g_d'$ moves $\pi e_j$, and that one swaps it with
$\pi o_{j+d}$. So $g_d'$ swaps the pair exactly when $g_d$ does not, that is, when $\mu(2j+d)=0$.

(2) $h_{d,\beta}$ fixes $\alpha$ exactly when
$\beta=\sigma_\alpha(d)$. As $h=h_{d_0,\varepsilon}$ fixes $\alpha\in C$, $\sigma_\alpha(d_0)=\varepsilon$.
Now let $\mu(i+j)=\sigma_\alpha(i-j)$, and put $d=i-j$, so that $i+j=2j+d$. By ($\ast$),
$h'=h_{d,\sigma_\alpha(d)}$ swaps $\pi e_j$ and $\pi o_i$, and $h'$ fixes $\alpha$. As $h'\in P=\Aut\Gamma$,
$f_\alpha(\pi o_i)=E(h'\alpha,h'\pi e_j)=E(\alpha,\pi e_j)=f_\alpha(\pi e_j)$.

(3) Suppose that such cells exist, and let $z=z_{d_0}$; then $h=h_{d_0,\varepsilon}$ is a half of $z$.
As $p=2j+d_0$ and $\mu(p)=\varepsilon$, by ($\ast$) $h$ swaps $\pi e_j$ and $\pi o_{j+d_0}$, and likewise
$\pi e_{j'}$ and $\pi o_{j'+d_0}$. These two $z$-pairs are moved by $h$, so they lie in $C^\complement$; they are
distinct, since $\pi e_{j'}\ne\pi e_j$ and no $e$-vertex is an $o$-vertex. Let $\alpha\in C$. By (2),
$\sigma_\alpha(d_0)=\varepsilon$, so by hypothesis the cells $(p,d_0)$ and $(p',d_0)$, which are the edges
$e_jo_{j+d_0}$ and $e_{j'}o_{j'+d_0}$ of $G(\mu,\sigma_\alpha)$ (as $\mu(p)=\mu(p')=\varepsilon=\sigma_\alpha(d_0)$),
are rook-connected in $G(\mu,\sigma_\alpha)$. So a path in $G(\mu,\sigma_\alpha)$ joins $e_j$ to $e_{j'}$;
each of its edges $e_ao_b$ has $\mu(a+b)=\sigma_\alpha(b-a)$, so $f_\alpha(\pi e_a)=f_\alpha(\pi o_b)$ by (2).
Hence $E(\alpha,\pi e_j)=E(\alpha,\pi e_{j'})$: the points $\pi e_j,\pi e_{j'}$ have the same neighbourhood
in $C$. $(h,h_{d_0,1-\varepsilon})$ is a complementary pair for $z$, with $C_z=\Fix(h)=C$. By
Lemma~\ref{lem:cross-nbhd}(3), the two points of a $z$-pair inside $C^\complement$ have the same neighbourhood in $C$,
and distinct $z$-pairs inside $C^\complement$ have distinct neighbourhoods in $C$. Applied to the $z$-pairs of
$\pi e_j$ and $\pi e_{j'}$, this contradicts the equality just proved.
\end{proof}

The proof of the theorem reads the sequence through its change sequence, defined next, which marks
the positions where a value differs from the value two places later; it vanishes exactly when the
sequence has period two.
In outline, local forbidden patterns force changes far apart, and an isolated-change argument removes
the last one.
Each of the twelve words in the table below is a short pattern of changes
whose presence alone forces a robust pair: these are the local obstructions. When none of them
occurs, distinct changes are at least eight apart, and a single isolated change still forces a
robust pair. Where no robust pair exists, the change sequence therefore vanishes.

For $\mu:\Z\to\{0,1\}$, put $\delta_i=\mu(i)+\mu(i+2)\pmod2$, the \emph{change sequence} of $\mu$.
Table~\ref{tab:words} lists twelve finite $\delta$-words, labelled $W_A$--$W_L$. In the row of a word, the
$\mu$-column is its normalized reconstruction: the $\mu$-values at consecutive positions starting at
zero, with $\mu(0)=\mu(1)=0$ and $\mu(i+2)=\mu(i)+\delta_i$. The column $(d_0,\varepsilon)$ gives the
base row $d_0$ with $\sigma(d_0)=\varepsilon$; the base cells are $(p,d_0),(p',d_0)$, and the targets
are $e_j,e_{j'}$, where $j=(p-d_0)/2$ and $j'=(p'-d_0)/2$; the last column gives $|j-j'|$.
In the language of symbolic dynamics, the binary sequences on $\Z$ that contain none of the twelve words
form a shift of finite type, the one with $W_A$--$W_L$ as its forbidden words; the proof of
Theorem~\ref{thm:twelve-words} shows that, when no robust pair exists, the change sequence is a point of
this shift and must vanish. 

\begin{table}[!htbp]
\caption{The twelve change words $W_A$--$W_L$: normalized $\mu$-word, base row and bit
$(d_0,\varepsilon)$, base cells, and $|j-j'|$.}\label{tab:words}
\centering
{\small
\begin{tabular}{clllll}
label & $\delta$-word & normalized $\mu$-word & $(d_0,\varepsilon)$ & base cells & $|j-j'|$\\ \hline
$W_A$ & \texttt{1101} & \texttt{001110} & $(0,1)$ & $(2,0),(4,0)$ & 1\\
$W_B$ & \texttt{11001} & \texttt{0011110} & $(-1,1)$ & $(3,-1),(5,-1)$ & 1\\
$W_C$ & \texttt{110000} & \texttt{00111111} & $(0,1)$ & $(2,0),(4,0)$ & 1\\
$W_D$ & \texttt{110001} & \texttt{00111110} & $(0,1)$ & $(2,0),(6,0)$ & 2\\
$W_E$ & \texttt{11111} & \texttt{0011001} & $(-1,0)$ & $(1,-1),(5,-1)$ & 2\\
$W_F$ & \texttt{10100} & \texttt{0010000} & $(-1,0)$ & $(1,-1),(3,-1)$ & 1\\
$W_G$ & \texttt{101010101} & \texttt{00100010001} & $(-1,0)$ & $(1,-1),(9,-1)$ & 4\\
$W_H$ & \texttt{100100} & \texttt{00101111} & $(0,1)$ & $(2,0),(4,0)$ & 1\\
$W_I$ & \texttt{1000100} & \texttt{001010000} & $(1,0)$ & $(1,1),(5,1)$ & 2\\
$W_J$ & \texttt{10000100} & \texttt{0010101111} & $(0,1)$ & $(2,0),(4,0)$ & 1\\
$W_K$ & \texttt{100000100} & \texttt{00101010000} & $(-1,0)$ & $(5,-1),(7,-1)$ & 1\\
$W_L$ & \texttt{1000000100} & \texttt{001010101111} & $(0,1)$ & $(4,0),(8,0)$ & 2\\
\end{tabular}}

\end{table}

\FloatBarrier

\begin{lemma}[Spacing of ones]\label{lem:spacing}
Let $\delta$ be a binary sequence on $\Z$, and let $W_A$--$W_L$ be the words of Table~\ref{tab:words}.
\begin{enumerate}
\item If $\delta$ contains none of the words $W_A$--$W_E$, no two adjacent entries of $\delta$ equal one.
\item If $\delta$ contains none of the words $W_A$--$W_G$, the word $101$ does not occur in $\delta$.
\item If $\delta$ contains none of the words $W_A$--$W_E$ and $W_H$--$W_L$, and $101$ does not occur in
  $\delta$, then distinct ones in $\delta$ have distance at least eight.
\end{enumerate}
This lemma is purely combinatorial.
\end{lemma}

\begin{proof}
(1) Suppose two adjacent entries equal one. Because $W_E=11111$ does not occur, a finite run of at least
two ones must end. Its last two ones and the following zero give $110$. The next bit either gives
$1101$, or gives $1100$. In the latter case the next bit gives $11001$ or $11000$; in the last case one
more bit gives $110000$ or $110001$. These four outcomes are $W_A$, $W_B$, $W_C$ and $W_D$, none of which
occurs. This also covers a run unbounded on the left; a run unbounded on the right would itself contain
$11111$.

(2) After $101$, part~(1) forces a zero. The next bit must be one, because $W_F=10100$ does not occur.
Iterating gives $101010101=W_G$, which does not occur. Thus two ones cannot be at distance two.

(3) If two consecutive ones have distance $k\in\{3,4,5,6,7\}$, part~(1) and the absence of $101$ force
two zeroes after the second one. This gives the word $1\,0^{k-1}\,1\,0\,0$, which is respectively
$W_H$, $W_I$, $W_J$, $W_K$, $W_L$, none of which occurs. If two distinct ones are at distance less than
eight, the first of them and the next one after it are consecutive ones at distance at most that, hence
at distance $1$, $2$ or $k\in\{3,\dots,7\}$, excluded respectively by part~(1), the absence of $101$
and the words $W_H$--$W_L$.
\end{proof}

\begin{theorem}[Robust-pair theorem]\label{thm:twelve-words}
Let $\mu:\Z\to\{0,1\}$ not be $2$-periodic (some $p$ with $\mu(p)\ne\mu(p+2)$).
Then there are integers $d_0,p,p'$ and a bit $\varepsilon$ such that
\[
p\ne p',\qquad p\equiv p'\equiv d_0\pmod2,\qquad \mu(p)=\mu(p')=\varepsilon,
\]
and, for every $\sigma$ satisfying $\sigma(d_0)=\varepsilon$, the base cells
$(p,d_0),(p',d_0)$ are rook-connected. More precisely, put $\delta_i=\mu(i)+\mu(i+2)\pmod2$, and call
$|j-j'|=|p-p'|/2$, where $j=(p-d_0)/2$ and $j'=(p'-d_0)/2$, the \emph{target difference} of the
pair. If $\delta$ contains one of the twelve words $W_A$--$W_L$ of Table~\ref{tab:words}, then $d_0,p,p'$ and
$\varepsilon$ can be chosen with the target difference given for that word in the last column of
Table~\ref{tab:words}, which is $1$ or $2$,
and $4$ for $W_G=101010101$; if $\delta$ contains none of them, they can be chosen with
target difference $1$.
This theorem is purely combinatorial.
\end{theorem}

The proof is by contraposition, through the change sequence $\delta$: each word of Table~\ref{tab:words} forces a robust pair by a local computation; if none occurs, Lemma~\ref{lem:spacing} puts the changes of $\delta$ at least eight apart, and an isolated change forces a robust pair with target difference $1$.

\begin{proof}
\emph{Symmetries.} We first record two symmetries of the grid, used to normalize
local words. Independent complementation on the two parities is a symmetry. For bits $\eta_0,\eta_1$, define
\[
\widetilde\mu(p)=\mu(p)+\eta_{p\bmod2},\qquad
\widetilde\sigma(d)=\sigma(d)+\eta_{d\bmod2}.
\]
Since $p$ and $d$ have the same parity in a cell, both sides of the edge rule of the grid change by
the same bit, and its status as an edge is
unchanged. At the beginning of a finite word, we may therefore normalize its
first two $\mu$-values to $0$ independently. Translation of $p$ by an integer
$\ell$, together with the same translation of $d$, is also an isomorphism: it
only relabels the $o$-indices. Thus a local word may start at any integer, not
only at zero.
Both symmetries carry robust pairs back to $\mu$. If $(p,d_0),(p',d_0)$ is a robust pair for
$\widetilde\mu$ with bit $\widetilde\varepsilon$, then it is a robust pair for $\mu$ with bit
$\varepsilon=\widetilde\varepsilon+\eta_{d_0\bmod2}$: indeed $\mu(p)=\mu(p')=\varepsilon$ since
$p\equiv p'\equiv d_0\pmod2$, and $\sigma\mapsto\widetilde\sigma$ is a bijection from
$\{\sigma:\sigma(d_0)=\varepsilon\}$ onto $\{\widetilde\sigma:\widetilde\sigma(d_0)=\widetilde\varepsilon\}$
with $G(\mu,\sigma)=G(\widetilde\mu,\widetilde\sigma)$. \emph{The change sequence.} The argument is phrased in terms of the changes of
the sequence, which the parity symmetries do not see. Recall the change sequence of $\mu$, defined before
Table~\ref{tab:words}:
\begin{equation}\label{eq:delta}
\delta_i=\mu(i)+\mu(i+2)\pmod2 .
\end{equation}
Once
$\mu(0)=\mu(1)=0$ is fixed, a word in $\delta$ determines a unique $\mu$-word by
$\mu(i+2)=\mu(i)+\delta_i$.

Suppose, towards a contradiction, that no robust pair as in the theorem exists.
We will show that $\delta$ must vanish.

\emph{Local forbidden-word lemma.} The first ingredient is a list of local
configurations that force robust pairs. Each of the twelve $\delta$-words of
Table~\ref{tab:words} forces a robust pair, with the base row, base cells and targets given in its
row (the columns are described before the table). In every row
$p\ne p'$ and the base cells have the displayed common $\mu$-value $\varepsilon$, so both targets lie on
base edges whenever $\sigma(d_0)=\varepsilon$; thus connecting the targets connects the two base cells.
Precisely: for each word, with $\mu$ given at the displayed consecutive
positions, for every assignment of the rows listed in its path table (the base
row being fixed), the two targets are connected using only cells of the
displayed $\mu$-word. Appendix~\ref{app:paths} displays the path tables of
the local lemmas for all twelve words. Each path table is read as follows.
\begin{enumerate}
\item Fix the base row bit $\sigma(d_0)=\varepsilon$, and assume that the two
  target $e$-vertices are disconnected.
\item A row of the table that assumes a value $\sigma(d)=b$ displays a path
  joining the targets, valid under that value and the values forced in the
  rows above it. By the assumption, $\sigma(d)=1-b$ is forced.
\item At the terminal row, both values of the remaining row bit give a joining
  path. This is a contradiction.
\item Two lemmas require one explicit binary case split; each branch is treated
  by steps (2) and (3), and both branches are displayed.
\end{enumerate}  Thus unspecified $\mu$-positions and unspecified row
bits remain arbitrary, which is what makes the lemmas translation-invariant. If
$\mu'=\mu(\cdot+\ell)$, then $\mu'(i)+\mu'(i+2)=\delta_{i+\ell}$; so if a word of
Table~\ref{tab:words} occurs in $\delta$ starting at index $\ell$, it occurs in the change sequence of
$\mu'$ at index $0$; a parity complement (which leaves that change sequence unchanged) makes
$\mu'(0)=\mu'(1)=0$, and then, by $\mu'(i+2)=\mu'(i)+\delta_{i+\ell}$, the sequence $\mu'$ on the
displayed positions is the normalized $\mu$-word of that word; the path table gives a robust pair
for $\mu'$, which carries back to $\mu$ with the same target difference, by the symmetries above.

\emph{The forbidden words force changes to be far apart.} By the local forbidden-word lemma and
the symmetries of the first paragraph, an occurrence in $\delta$ of any of the twelve words would
give a robust pair for $\mu$; under our assumption, therefore, none of the twelve $\delta$-words of
Table~\ref{tab:words} occurs.

The spacing lemma (Lemma~\ref{lem:spacing}), applied to $\delta$, now gives three facts: no two adjacent
entries of $\delta$ equal one, by part~(1); the word $101$ does not occur, by part~(2); and so, by
part~(3), distinct ones in $\delta$ have distance at least eight.

\emph{An isolated change also forces a robust pair.} It remains to treat a change
with no other change nearby. If $\delta$ is nonzero,
translate a nonzero entry to index $0$. Lemma~\ref{lem:spacing}(3) gives
\begin{equation}\label{eq:isolated}
\delta_0=1,\qquad \delta_i=0\quad(1\le|i|\le7).
\end{equation}
Use independent parity complements to arrange $\mu(0)=\mu(1)=0$. Put
$M(j)=\mu(2j)$. Then
\begin{equation}\label{eq:isoM}
M(-3),\dots,M(4)=00001111,
\end{equation}
and
\begin{equation}\label{eq:isoodd}
\mu(-7)=\mu(-5)=\dots=\mu(9)=0 .
\end{equation}
Take base row $d_0=-1$ with $\sigma(-1)=0$. Equation~\eqref{eq:isoodd} identifies
\[
e_j\sim o_{j-1}\qquad(-3\le j\le5).
\]The nine base edges $e_jo_{j-1}$ ($-3\le j\le5$) are pairwise disjoint; collapse each of them to a
single vertex, and label these vertices $k=j+3$, so that $k$
ranges from $0$ to $8$. For an even row $d=2t$, the edge from $e_j$ to $o_{j+2t}$
becomes the edge from $k$ to $k+(2t+1)$, present exactly when $M(k-3+t)=\sigma(2t)$, the cell
being $(2j+2t,2t)$ with $\mu(2j+2t)=M(j+t)$. We use only $t=0,1,2$, giving
lengths $1,3,5$.

Three row bits, $\sigma(0)$, $\sigma(2)$ and $\sigma(4)$, decide which chords are present. The
following paths join the
collapsed vertices $3$ and $4$:
\[
\begin{array}{l|l}
\text{row choices} & \text{path}\\\hline
\sigma(0)=0,\ \sigma(2),\sigma(4)\ \text{arbitrary} & 3\text{--}4\\
\sigma(0)=1,\ \sigma(2)=0,\ \sigma(4)=0 & 3\text{--}0\text{--}5\text{--}4\\
\sigma(0)=1,\ \sigma(2)=0,\ \sigma(4)=1 & 3\text{--}8\text{--}7\text{--}2\text{--}5\text{--}4\\
\sigma(0)=1,\ \sigma(2)=1,\ \sigma(4)\ \text{arbitrary} & 3\text{--}6\text{--}5\text{--}4
\end{array}
\]
For example, a collapsed edge from $k$ to $k+(2t+1)$ is valid when the center bit
$M(k-3+t)$ equals $\sigma(2t)$. All vertices lie among the nine collapsed base edges, so every
collapsed path lifts to a path of cells in the original grid.

Lifting back gives the robust pair. Vertices $3$ and $4$ correspond to $e_0$ and
$e_1$. Their base cells are
$(-1,-1)$ and $(1,-1)$, with the same $\mu$-colour $0$. They are connected for
every choice of $\sigma$ with $\sigma(-1)=0$. This is the required robust pair,
contradicting the assumption that none exists.

Consequently $\delta$ vanishes. By \eqref{eq:delta}, $\mu$ is $2$-periodic, which
proves Theorem~\ref{thm:twelve-words} by contraposition. The only robust pairs this argument exhibits are those of
the local lemmas for the twelve words, with the target differences of Table~\ref{tab:words}, and the isolated-change pair
$e_0,e_1$, with target difference $1$; the symmetries of the first paragraph do not change $p-p'$.
This proves the assertion on target differences.
\end{proof}

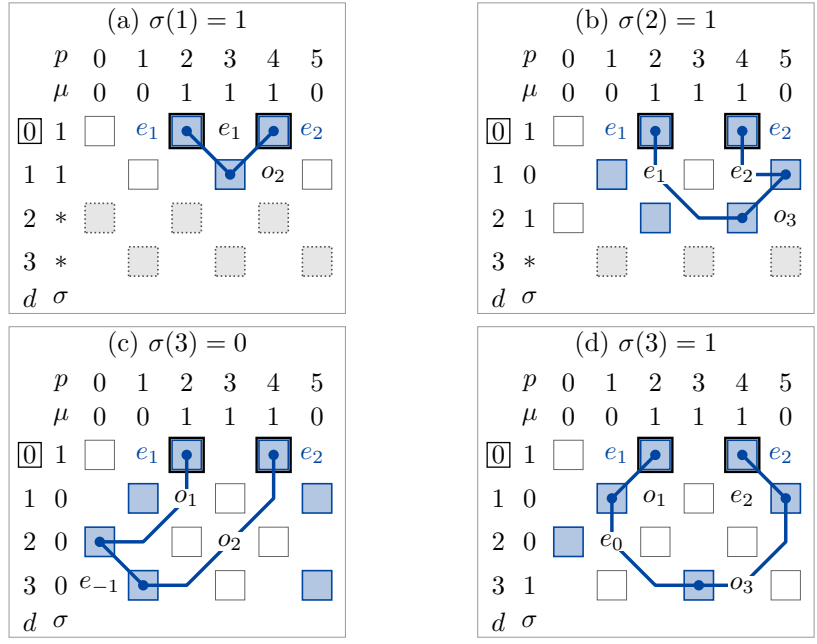
\begin{figure}[!htbp]
\centering
\providecolor{figaccent}{RGB}{0,70,160}\begin{tikzpicture}[x=1cm,y=1cm,font=\normalsize,
  every node/.style={inner sep=0pt,outer sep=0pt},
  present/.style={draw=figaccent,fill=figaccent!30,line width=0.6pt},
  absent/.style={draw=black!60,fill=white,line width=0.3pt},
  undet/.style={draw=black!70,fill=black!10,densely dotted,line width=0.6pt},
  basecell/.style={draw=black,line width=0.9pt},
  rook/.style={figaccent,line width=1.3pt,line cap=round,line join=round},
  frame/.style={draw=black!40,line width=0.3pt}]
\draw[frame] (0.000,-2.413) rectangle (4.450,1.696);
\node[font=\normalsize] at (2.225,1.456) {(a)\ $\sigma(1)=1$};
\node at (1.200,0.956) {$0$};
\node at (1.200,0.506) {$0$};
\node at (1.775,0.956) {$1$};
\node at (1.775,0.506) {$0$};
\node at (2.350,0.956) {$2$};
\node at (2.350,0.506) {$1$};
\node at (2.925,0.956) {$3$};
\node at (2.925,0.506) {$1$};
\node at (3.500,0.956) {$4$};
\node at (3.500,0.506) {$1$};
\node at (4.075,0.956) {$5$};
\node at (4.075,0.506) {$0$};
\node at (0.688,0.956) {$p$};
\node at (0.688,0.506) {$\mu$};
\node at (0.275,-2.203) {$d$};
\node at (0.688,-2.203) {$\sigma$};
\node at (0.275,0.000) {$0$};
\node at (0.688,0.000) {$1$};
\node at (0.275,-0.575) {$1$};
\node at (0.688,-0.575) {$1$};
\node at (0.275,-1.150) {$2$};
\node at (0.688,-1.150) {$\ast$};
\node at (0.275,-1.725) {$3$};
\node at (0.688,-1.725) {$\ast$};
\draw[line width=0.5pt] (0.125,-0.163) rectangle (0.425,0.163);
\draw[absent] (1.006,-0.194) rectangle (1.394,0.194);
\draw[present] (2.156,-0.194) rectangle (2.544,0.194);
\draw[present] (3.306,-0.194) rectangle (3.694,0.194);
\draw[absent] (1.581,-0.769) rectangle (1.969,-0.381);
\draw[present] (2.731,-0.769) rectangle (3.119,-0.381);
\draw[absent] (3.881,-0.769) rectangle (4.269,-0.381);
\draw[undet] (1.006,-1.344) rectangle (1.394,-0.956);
\draw[undet] (2.156,-1.344) rectangle (2.544,-0.956);
\draw[undet] (3.306,-1.344) rectangle (3.694,-0.956);
\draw[undet] (1.581,-1.919) rectangle (1.969,-1.531);
\draw[undet] (2.731,-1.919) rectangle (3.119,-1.531);
\draw[undet] (3.881,-1.919) rectangle (4.269,-1.531);
\draw[basecell] (2.125,-0.225) rectangle (2.575,0.225);
\draw[basecell] (3.275,-0.225) rectangle (3.725,0.225);
\draw[rook] (2.350,0.000) -- (2.925,-0.575);
\draw[rook] (2.925,-0.575) -- (3.500,0.000);
\fill[figaccent] (2.350,0.000) circle (0.0688);
\fill[figaccent] (2.925,-0.575) circle (0.0688);
\fill[figaccent] (3.500,0.000) circle (0.0688);
\node at (2.925,0.000) {$e_{1}$};
\node at (3.500,-0.575) {$o_{2}$};
\node at (1.837,0.000) {\color{figaccent}$e_{1}$};
\node at (4.013,0.000) {\color{figaccent}$e_{2}$};
\draw[frame] (6.200,-2.413) rectangle (10.650,1.696);
\node[font=\normalsize] at (8.425,1.456) {(b)\ $\sigma(2)=1$};
\node at (7.400,0.956) {$0$};
\node at (7.400,0.506) {$0$};
\node at (7.975,0.956) {$1$};
\node at (7.975,0.506) {$0$};
\node at (8.550,0.956) {$2$};
\node at (8.550,0.506) {$1$};
\node at (9.125,0.956) {$3$};
\node at (9.125,0.506) {$1$};
\node at (9.700,0.956) {$4$};
\node at (9.700,0.506) {$1$};
\node at (10.275,0.956) {$5$};
\node at (10.275,0.506) {$0$};
\node at (6.888,0.956) {$p$};
\node at (6.888,0.506) {$\mu$};
\node at (6.475,-2.203) {$d$};
\node at (6.888,-2.203) {$\sigma$};
\node at (6.475,0.000) {$0$};
\node at (6.888,0.000) {$1$};
\node at (6.475,-0.575) {$1$};
\node at (6.888,-0.575) {$0$};
\node at (6.475,-1.150) {$2$};
\node at (6.888,-1.150) {$1$};
\node at (6.475,-1.725) {$3$};
\node at (6.888,-1.725) {$\ast$};
\draw[line width=0.5pt] (6.325,-0.163) rectangle (6.625,0.163);
\draw[absent] (7.206,-0.194) rectangle (7.594,0.194);
\draw[present] (8.356,-0.194) rectangle (8.744,0.194);
\draw[present] (9.506,-0.194) rectangle (9.894,0.194);
\draw[present] (7.781,-0.769) rectangle (8.169,-0.381);
\draw[absent] (8.931,-0.769) rectangle (9.319,-0.381);
\draw[present] (10.081,-0.769) rectangle (10.469,-0.381);
\draw[absent] (7.206,-1.344) rectangle (7.594,-0.956);
\draw[present] (8.356,-1.344) rectangle (8.744,-0.956);
\draw[present] (9.506,-1.344) rectangle (9.894,-0.956);
\draw[undet] (7.781,-1.919) rectangle (8.169,-1.531);
\draw[undet] (8.931,-1.919) rectangle (9.319,-1.531);
\draw[undet] (10.081,-1.919) rectangle (10.469,-1.531);
\draw[basecell] (8.325,-0.225) rectangle (8.775,0.225);
\draw[basecell] (9.475,-0.225) rectangle (9.925,0.225);
\draw[rook] (8.550,0.000) -- (8.550,-0.575) -- (9.125,-1.150) -- (9.700,-1.150);
\draw[rook] (9.700,-1.150) -- (10.275,-0.575);
\draw[rook] (10.275,-0.575) -- (9.700,-0.575) -- (9.700,0.000);
\fill[figaccent] (8.550,0.000) circle (0.0688);
\fill[figaccent] (9.700,-1.150) circle (0.0688);
\fill[figaccent] (10.275,-0.575) circle (0.0688);
\fill[figaccent] (9.700,0.000) circle (0.0688);
\node[fill=white,inner sep=0.5pt] at (8.550,-0.575) {$e_{1}$};
\node at (10.275,-1.150) {$o_{3}$};
\node[fill=white,inner sep=0.5pt] at (9.700,-0.575) {$e_{2}$};
\node at (8.037,0.000) {\color{figaccent}$e_{1}$};
\node at (10.213,0.000) {\color{figaccent}$e_{2}$};
\draw[frame] (0.000,-6.696) rectangle (4.450,-2.588);
\node[font=\normalsize] at (2.225,-2.828) {(c)\ $\sigma(3)=0$};
\node at (1.200,-3.328) {$0$};
\node at (1.200,-3.778) {$0$};
\node at (1.775,-3.328) {$1$};
\node at (1.775,-3.778) {$0$};
\node at (2.350,-3.328) {$2$};
\node at (2.350,-3.778) {$1$};
\node at (2.925,-3.328) {$3$};
\node at (2.925,-3.778) {$1$};
\node at (3.500,-3.328) {$4$};
\node at (3.500,-3.778) {$1$};
\node at (4.075,-3.328) {$5$};
\node at (4.075,-3.778) {$0$};
\node at (0.688,-3.328) {$p$};
\node at (0.688,-3.778) {$\mu$};
\node at (0.275,-6.486) {$d$};
\node at (0.688,-6.486) {$\sigma$};
\node at (0.275,-4.284) {$0$};
\node at (0.688,-4.284) {$1$};
\node at (0.275,-4.859) {$1$};
\node at (0.688,-4.859) {$0$};
\node at (0.275,-5.434) {$2$};
\node at (0.688,-5.434) {$0$};
\node at (0.275,-6.009) {$3$};
\node at (0.688,-6.009) {$0$};
\draw[line width=0.5pt] (0.125,-4.446) rectangle (0.425,-4.121);
\draw[absent] (1.006,-4.478) rectangle (1.394,-4.090);
\draw[present] (2.156,-4.478) rectangle (2.544,-4.090);
\draw[present] (3.306,-4.478) rectangle (3.694,-4.090);
\draw[present] (1.581,-5.053) rectangle (1.969,-4.665);
\draw[absent] (2.731,-5.053) rectangle (3.119,-4.665);
\draw[present] (3.881,-5.053) rectangle (4.269,-4.665);
\draw[present] (1.006,-5.628) rectangle (1.394,-5.240);
\draw[absent] (2.156,-5.628) rectangle (2.544,-5.240);
\draw[absent] (3.306,-5.628) rectangle (3.694,-5.240);
\draw[present] (1.581,-6.203) rectangle (1.969,-5.815);
\draw[absent] (2.731,-6.203) rectangle (3.119,-5.815);
\draw[present] (3.881,-6.203) rectangle (4.269,-5.815);
\draw[basecell] (2.125,-4.509) rectangle (2.575,-4.059);
\draw[basecell] (3.275,-4.509) rectangle (3.725,-4.059);
\draw[rook] (2.350,-4.284) -- (2.350,-4.859) -- (1.775,-5.434) -- (1.200,-5.434);
\draw[rook] (1.200,-5.434) -- (1.775,-6.009);
\draw[rook] (1.775,-6.009) -- (2.350,-6.009) -- (2.925,-5.434) -- (3.500,-4.859) -- (3.500,-4.284);
\fill[figaccent] (2.350,-4.284) circle (0.0688);
\fill[figaccent] (1.200,-5.434) circle (0.0688);
\fill[figaccent] (1.775,-6.009) circle (0.0688);
\fill[figaccent] (3.500,-4.284) circle (0.0688);
\node[fill=white,inner sep=0.5pt] at (2.350,-4.859) {$o_{1}$};
\node at (1.200,-6.009) {$e_{-1}$};
\node[fill=white,inner sep=0.5pt] at (2.925,-5.434) {$o_{2}$};
\node at (1.837,-4.284) {\color{figaccent}$e_{1}$};
\node at (4.013,-4.284) {\color{figaccent}$e_{2}$};
\draw[frame] (6.200,-6.696) rectangle (10.650,-2.588);
\node[font=\normalsize] at (8.425,-2.828) {(d)\ $\sigma(3)=1$};
\node at (7.400,-3.328) {$0$};
\node at (7.400,-3.778) {$0$};
\node at (7.975,-3.328) {$1$};
\node at (7.975,-3.778) {$0$};
\node at (8.550,-3.328) {$2$};
\node at (8.550,-3.778) {$1$};
\node at (9.125,-3.328) {$3$};
\node at (9.125,-3.778) {$1$};
\node at (9.700,-3.328) {$4$};
\node at (9.700,-3.778) {$1$};
\node at (10.275,-3.328) {$5$};
\node at (10.275,-3.778) {$0$};
\node at (6.888,-3.328) {$p$};
\node at (6.888,-3.778) {$\mu$};
\node at (6.475,-6.486) {$d$};
\node at (6.888,-6.486) {$\sigma$};
\node at (6.475,-4.284) {$0$};
\node at (6.888,-4.284) {$1$};
\node at (6.475,-4.859) {$1$};
\node at (6.888,-4.859) {$0$};
\node at (6.475,-5.434) {$2$};
\node at (6.888,-5.434) {$0$};
\node at (6.475,-6.009) {$3$};
\node at (6.888,-6.009) {$1$};
\draw[line width=0.5pt] (6.325,-4.446) rectangle (6.625,-4.121);
\draw[absent] (7.206,-4.478) rectangle (7.594,-4.090);
\draw[present] (8.356,-4.478) rectangle (8.744,-4.090);
\draw[present] (9.506,-4.478) rectangle (9.894,-4.090);
\draw[present] (7.781,-5.053) rectangle (8.169,-4.665);
\draw[absent] (8.931,-5.053) rectangle (9.319,-4.665);
\draw[present] (10.081,-5.053) rectangle (10.469,-4.665);
\draw[present] (7.206,-5.628) rectangle (7.594,-5.240);
\draw[absent] (8.356,-5.628) rectangle (8.744,-5.240);
\draw[absent] (9.506,-5.628) rectangle (9.894,-5.240);
\draw[absent] (7.781,-6.203) rectangle (8.169,-5.815);
\draw[present] (8.931,-6.203) rectangle (9.319,-5.815);
\draw[absent] (10.081,-6.203) rectangle (10.469,-5.815);
\draw[basecell] (8.325,-4.509) rectangle (8.775,-4.059);
\draw[basecell] (9.475,-4.509) rectangle (9.925,-4.059);
\draw[rook] (8.550,-4.284) -- (7.975,-4.859);
\draw[rook] (7.975,-4.859) -- (7.975,-5.434) -- (8.550,-6.009) -- (9.125,-6.009);
\draw[rook] (9.125,-6.009) -- (9.700,-6.009) -- (10.275,-5.434) -- (10.275,-4.859);
\draw[rook] (10.275,-4.859) -- (9.700,-4.284);
\fill[figaccent] (8.550,-4.284) circle (0.0688);
\fill[figaccent] (7.975,-4.859) circle (0.0688);
\fill[figaccent] (9.125,-6.009) circle (0.0688);
\fill[figaccent] (10.275,-4.859) circle (0.0688);
\fill[figaccent] (9.700,-4.284) circle (0.0688);
\node at (8.550,-4.859) {$o_{1}$};
\node[fill=white,inner sep=0.5pt] at (7.975,-5.434) {$e_{0}$};
\node[fill=white,inner sep=0.5pt] at (9.700,-6.009) {$o_{3}$};
\node at (9.700,-4.859) {$e_{2}$};
\node at (8.037,-4.284) {\color{figaccent}$e_{1}$};
\node at (10.213,-4.284) {\color{figaccent}$e_{2}$};
\end{tikzpicture}
\caption{The grid $G(\mu,\sigma)$ of Definition~\ref{def:grid} for the word $W_A=1101$ of Table~\ref{tab:words}. A cell
is shaded if $\mu(p)=\sigma(d)$, open if not, dotted ($\ast$) if $\sigma(d)$ is not fixed; outside the window $\mu$,
$\sigma$ are arbitrary. For $\sigma(0)=1$ the framed cells of the boxed base row form a robust pair: whatever the other
row bits, a path joins $e_1,e_2$. Panels (a)--(d) are the rows of the path table of $W_A$ (Appendix~\ref{app:paths}),
each assuming its title's bit and the bits forced before; the heavy line joins that row's path cells ($\bullet$), each
joint labelled by the shared vertex.}
\label{fig:grid}
\end{figure}

\begin{corollary}[Period two on every component]\label{cor:period2}
Assume \textup{(S)} and $H=\langle c\rangle$, and let $x,y\in\Ii$ with chosen
halves. On every component of their two-matchings graph, finite or bi-infinite,
\begin{equation}\label{eq:period2}
M(i+2)=M(i),\qquad N(i+2)=N(i).
\end{equation}
On a bi-infinite component and on every cycle with $m\ge3$, $m\ne4$, both masks
are constant. On an eight-vertex cycle, both are constant or exactly one is
alternating.
\end{corollary}

\begin{proof}
Lift the masks to $\Z$ and interleave them to $\mu$ as in Lemma~\ref{lem:robust-pair}(1), and put
$\delta_k=\mu(k)+\mu(k+2)$ in $\F_2$. Then $\delta_{2i}=M(i)+M(i+1)$ and $\delta_{2i+1}=N(i)+N(i+1)$;
so \eqref{eq:period2} says $\delta_k=\delta_{k+2}$ for all $k$, and both masks are constant exactly when
$\delta=0$. For $m\in\{1,2\}$ the masks are indexed by $\Z/m\Z$ with $m$ dividing $2$, so
\eqref{eq:period2} holds.

\emph{Claim: if $\delta\ne0$, then $\Theta$ is a cycle and $m$ divides the target difference of every pair
given by Theorem~\ref{thm:twelve-words}.} Indeed $\mu$ is not $2$-periodic, and the cells
$(p,d_0),(p',d_0)$ and the bit $\varepsilon$ of that theorem satisfy the hypotheses of
Lemma~\ref{lem:robust-pair}(3), with $h$ the chosen half on diagonal $d_0$ if $\varepsilon=1$ and its
complement if $\varepsilon=0$. That lemma excludes $\pi e_j\ne\pi e_{j'}$. Since $j\ne j'$ and $\pi$ is the
identity or reduction modulo $m$, $\Theta$ is a cycle and $m$ divides $|j-j'|$.

\emph{Bi-infinite components and cycles with $m\ge3$, $m\ne4$.} If $\delta\ne0$,
Theorem~\ref{thm:twelve-words} gives a pair with target difference $1$, $2$ or $4$, contradicting the
claim. So $\delta=0$: both masks are constant, and \eqref{eq:period2} holds.

\emph{Eight-vertex cycles, $m=4$.} Here $\mu$ and $\delta$ have period $8$. If $\delta=0$, both masks are
constant. Otherwise, by the claim, $\delta$ contains no word of Table~\ref{tab:words} with target difference
$1$ or $2$, that is, none of $W_A$--$W_F$ and $W_H$--$W_L$; by Lemma~\ref{lem:spacing}(1) it contains no $11$.
If $101$ occurs in $\delta$ at positions $k,k+1,k+2$, then $\delta_{k+3}=0$ (no $11$) and $\delta_{k+4}=1$
(else $W_F=10100$ occurs at $k$), so $101$ occurs at $k+2$; by induction and periodicity
$\delta_{k+2r}=1$ and $\delta_{k+2r+1}=0$ for all $r\in\Z$. Thus $\delta_k=\delta_{k+2}$ for all $k$, which is
\eqref{eq:period2}; the mask whose differences are the ones of $\delta$ alternates and the other is constant.
If $101$ does not occur, Lemma~\ref{lem:spacing}(3) puts distinct ones of $\delta$ at distance at least $8$, so
some $k$ has $\delta_k=1$ and $\delta_{k+2}=\delta_{k+4}=\delta_{k+6}=0$; but these four terms sum to
$\mu(k)+\mu(k+8)=0$, a contradiction.
\end{proof}

\begin{remark}\label{rem:eightcycle}
The mixed eight-cycle left open by Corollary~\ref{cor:period2} is excluded
in Section~\ref{ss:align} (Proposition~\ref{prop:mixed8two}), using the translation
square root.
\end{remark}

\begin{theorem}[Fourth powers from period two]\label{thm:rs4}
Assume \textup{(S)} and $H=\langle c\rangle$. For every $x,y\in\Ii$,
$(xy)^4=1$. Every component of the two-matchings graph of $x$ and $y$ has $2$,
$4$ or $8$ vertices.
\end{theorem}

\begin{proof}
By Corollary~\ref{cor:period2}, $(xy)^2$ preserves $C_x$ on every component:
it shifts the red indices by $\pm2$, while the mask has period two; on
components of size two or four it is the identity. The element $(xy)^2$ of $P$ preserves $C_x$, so, by Lemma~\ref{lem:cross-nbhd}(3),
$[(xy)^2,x]=1$. Since $x(xy)x=(xy)^{-1}$, conjugation by the matching both fixes the square of the
product and inverts it, so that square is its own inverse; this proves $(xy)^4=1$. Every component is then finite with
$m\mid4$ (Lemma~\ref{lem:components}(1),(3)): only component sizes $2,4,8$ remain.
\end{proof}

\subsection{Diagonal alignment, the translation square root, and pair preservation}\label{ss:align}

\begin{lemma}[Group algebras of finite cyclic $2$-groups]\label{rem:truncated}
Let $e\ge1$ be an integer.
\begin{enumerate}
\item Let $e$ be a power of two, let $g$ generate a cyclic group of order $e$, and put $\theta=g-1$ in its
  group algebra $\F_2[\langle g\rangle]=\F_2[g]/(g^e-1)$. Then $\F_2[\langle g\rangle]=\F_2[\theta]/(\theta^e)$, under which
  the augmentation (the sum of the coefficients) is $\theta\mapsto0$; if $e\ge2$, then
  $\theta^{e/2}=1+\iota$, where $\iota=g^{e/2}$ is the involution of $\langle g\rangle$.
\item In $\F_2[\theta]/(\theta^e)$ the units are exactly the elements $1+\theta f$; in a group algebra as
  in~(1) they are the elements of augmentation one.
\item The ideals of $\F_2[\theta]/(\theta^e)$ form the chain $(1)\supset(\theta)\supset\dots\supset(\theta^e)=0$,
  and every quotient ring is $\F_2[\theta]/(\theta^j)$ for some $j\le e$.
\end{enumerate}
\end{lemma}

\begin{proof}
(1) In characteristic two $(g-1)^k=g^k-1$ whenever $k$ is a power of two. For $k=e$ this gives
$\theta^e=g^e-1$, so $\F_2[g]/(g^e-1)=\F_2[\theta]/(\theta^e)$; the augmentation sends $g$ to $1$, hence
$\theta$ to $0$. For $e\ge2$ and $k=e/2$ it gives $\theta^{e/2}=g^{e/2}-1=1+\iota$.

(2) In $\F_2[\theta]/(\theta^e)$ every $1+\theta f$ is a unit, with inverse $\sum_{i=0}^{e-1}(\theta f)^i$,
and every $\theta f$ is nilpotent. So the units are exactly the elements $1+\theta f$; in a group algebra as
in~(1) they are the elements of augmentation one, since the augmentation is $\theta\mapsto0$.

(3) Every nonzero element of $\F_2[\theta]/(\theta^e)$ is $\theta^k u$ with $0\le k<e$ and $u$ a unit,
by~(2). Hence every ideal is $(\theta^j)$, where $j$ is the least such $k$ in it ($j=e$ for the zero
ideal), and the quotient by $(\theta^j)$ is $\F_2[\theta]/(\theta^j)$.
\end{proof}

\begin{lemma}[Diagonal alignment]\label{lem:diag}
Let $\Lambda$ be a Pr\"ufer $2$-group, $\iota$ its unique involution, and
$\mathcal{T}:\Lambda\to\F_2$ with $\Pi(\mathcal{T})=\{1\}$. Suppose four distinct indices
$a_0,a_1,a_2,a_3$ have the property that, for every $z\in\Lambda$, the four
values $f_i=\mathcal{T}(za_i)$ satisfy
\begin{equation}\label{eq:star}
(f_0=f_1\ \text{and}\ f_2=f_3)\quad\text{or}\quad(f_1=f_2\ \text{and}\ f_3=f_0).
\end{equation}
Then
\begin{equation}\label{eq:diag}
a_2=\iota a_0,\qquad a_3=\iota a_1 .
\end{equation}
\end{lemma}

The four indices first split into two $\iota$-pairs, because $a_0+a_1+a_2+a_3$ annihilates $\mathcal{T}$ and, as $\iota$ is not a period, every element of the group algebra of a finite cyclic subgroup of $\Lambda$ that annihilates $\mathcal{T}$ is a multiple of $1+\iota$; a parity count of the values of $\mathcal{T}$ along the powers of $a_2a_0^{-1}$ then excludes the two non-diagonal pairings.

\begin{proof}
We first show that the four indices form two pairs differing by the involution.
Condition \eqref{eq:star} implies $f_0+f_1+f_2+f_3=0$, since either alternative splits the four
values into two equal pairs. Let $\F_2[\Lambda]$ act on functions $\phi:\Lambda\to\F_2$ by
$(\lambda\phi)(z)=\phi(z\lambda)$ for $\lambda\in\Lambda$, extended linearly; this is a module over the
commutative ring $\F_2[\Lambda]$. Then $((a_0+a_1+a_2+a_3)\mathcal{T})(z)=\sum_i\mathcal{T}(za_i)=f_0+f_1+f_2+f_3=0$ for every
$z$: the group-ring element $a_0+a_1+a_2+a_3$ annihilates $\mathcal{T}$, and it is nonzero because the
$a_i$ are distinct. Like $A$ (Section~\ref{s:setup}), $\Lambda$ is the union of a chain of finite cyclic
subgroups, so the four indices lie in one of them, $\Lambda_e=\langle g\rangle$, where $e$ is a
power of two; since they are distinct, $e\ge4$, so $\iota=g^{e/2}$ lies in $\Lambda_e$. By
Lemma~\ref{rem:truncated}(1) its group algebra is $\F_2[\theta]/(\theta^e)$, $\theta=g-1$, with
$\theta^{e/2}=1+\iota$, and by Lemma~\ref{rem:truncated}(3) the annihilator of $\mathcal{T}$ in this ring is an
ideal $(\theta^j)$. Since $\iota$ is not a period, $\theta^{e/2}\mathcal{T}=(1+\iota)\mathcal{T}$ is
nonzero. Consequently $j>e/2$, so every element of the annihilator is a multiple of the power of the
variable in the last formula, which is one plus the involution in characteristic two. It follows
that this element is divisible by $(1+\iota)$.
Equivalently, its four-element support is the union of two $\iota$-pairs: a multiple of one plus
the involution is unchanged by multiplication with the involution, so its support is a union of
such pairs.
It remains to exclude the other pairings. The non-diagonal
pairings are $\{a_0,a_1\},\{a_2,a_3\}$ and $\{a_1,a_2\},\{a_3,a_0\}$. The rotation $a'_i=a_{i+1}$
carries the second pairing to the first and exchanges the two alternatives of \eqref{eq:star}; so we
may assume $a_1=\iota a_0$ and $a_3=\iota a_2$. Put
\[
h=a_2a_0^{-1},\qquad \psi(z)=\mathcal{T}(za_0),\qquad \chi(z)=\psi(z)+\psi(z\iota).
\]
The relation $f_0+f_1+f_2+f_3=0$ says $\chi(zh)=\chi(z)$. Also $\chi$ is
nonzero, because the involution is not a period. At a point where $\chi(z)=1$, the first two values
differ, so the first alternative in \eqref{eq:star}
is impossible; the second, in which the middle two values agree, requires $\psi(zh)=\psi(z)+1$.
The same is true at
every $zh^k$, because the invariance of that function noted above keeps it equal to one along these
translates. The element $h$ is neither $1$ nor $\iota$: $h=1$ would give $a_2=a_0$, and $h=\iota$
would give $a_2=\iota a_0=a_1$, both against distinctness. So $h$ has order $2^n$ with $n\ge2$ and
$h^{2^{n-1}}=\iota$, by Lemma~\ref{lem:prufer}(1) for $\Lambda\cong A$. Iterating the bit flip
$\psi(zh^{k+1})=\psi(zh^k)+1$ an even
number $2^{n-1}$ of times gives $\psi(z\iota)=\psi(z)$, contradicting
$\chi(z)=1$.
Thus only the diagonal pairing \eqref{eq:diag} is possible.
\end{proof}

\begin{proposition}[No mixed eight-cycle: translation root]\label{prop:mixed8two}
Assume \textup{(S)} and $H=\langle c\rangle$. Let $x,y\in\Ii$. No eight-point
component of their two-matchings graph has one mask constant and the other
alternating. Both halves are constant on all eight-point components.
\end{proposition}

A mixed component lies in one side, Lemma~\ref{lem:diag} makes it the union of two cosets of $\langle d\rangle$ with
$d$ of order four, and conjugating $y$ by the square root $R(d)$ of $x$ gives an element of $\Ii$ whose product with
$y$ has order three on the component, against Theorem~\ref{thm:rs4}.

\begin{proof}
Suppose an eight-point $x,y$ component has the $x$-half constant and the
$y$-half alternating. We show that a translation square root of the first matching
then produces a product of two full matchings of order three on the component. By
Definition~\ref{def:halves}, $(t_x,v_x,x)=(t^g,v^g,\rho^g)$ for some $g\in P$; conjugate by $g$, that is,
apply $\omega\mapsto g\omega g^{-1}$ together with the vertex map $p\mapsto g(p)$, so that $x=\rho$; its chosen half is one
of $t,v$, with fixed set $A$ or $B$. The entire component $\Theta$ is in one
side: the $x$-mask is constant, so the chosen half of $x$ either fixes every vertex of $\Theta$ or swaps
every $x$-pair of $\Theta$, and in the second case its complementary half fixes every vertex of $\Theta$; so $\Theta$ lies in $A$ or in $B$. Label its vertices $0,\dots,7$:
\[
x=(0\,1)(2\,3)(4\,5)(6\,7),\qquad y=(0\,7)(1\,2)(3\,4)(5\,6).
\]
The nonconstant $y$-half and its complement restrict, in one order, to
\[
(1\,2)(5\,6),\qquad(3\,4)(7\,0).
\]
Figure~\ref{fig:cosets} shows the component. For an arbitrary vertex $p'$ in the opposite side, its column $p\mapsto E(p,p')$ on $\Theta$ is constant on
the four $x$-pairs for $x=\rho$. Write the four
pair-values as
$f_0,\dots,f_3$. Exactly one of the two $y$-halves fixes $p'$. That half lies in $P$ and preserves $\Theta$, so the column of $p'$ takes equal values at the two points of each of
its $2$-cycles on $\Theta$: if it restricts to $(1\,2)(5\,6)$, then $f_0=f_1$ and $f_2=f_3$; if it
restricts to $(3\,4)(7\,0)$, then $f_1=f_2$ and $f_3=f_0$. So the column of $p'$
satisfies \eqref{eq:star}.

The alignment lemma now identifies the pairing of the four pairs. The four
$x$-pairs are labelled by distinct elements of $A/\langle c\rangle$: the pair
$\{a,ac\}$, or $\{ab,acb\}$, by $a\langle c\rangle$. The \emph{exterior indicator} is
$\mathcal{T}:A/\langle c\rangle\to\F_2$, $\mathcal{T}(a\langle c\rangle)=\one_T(a)$, well defined because
$cT=T$; it has no period: if $\tilde h\langle c\rangle\in\Pi(\mathcal{T})$, then $\one_T(\tilde ha)=\one_T(a)$ for all
$a\in A$, so $\tilde hT=T$, $\tilde h\in\Pi(T)=\langle c\rangle$ and $\tilde h\langle c\rangle=1$. By
Lemma~\ref{lem:cross-nbhd}(1), $E(a,zb)=\one_T(az)$, so the value of the column of an opposite
vertex on the $x$-pair labelled $a_i$ is $f_i=\mathcal{T}(\lambda a_i)$, where $\lambda=z\langle c\rangle$ for
the vertex $zb$ (if $\Theta\subseteq A$) and $\lambda=a\langle c\rangle$ for the vertex $a$ (if
$\Theta\subseteq B$); as the vertex runs over the opposite side, $\lambda$ runs over all of
$A/\langle c\rangle$, so \eqref{eq:star} holds for every $\lambda$. The group
$\Lambda=A/\langle c\rangle$ is isomorphic to $A$ (Lemma~\ref{lem:prufer}(3)), a Pr\"ufer
$2$-group. By Lemma~\ref{lem:diag} (with
$\Lambda=A/\langle c\rangle$) their $\iota$-pairing is diagonal:
\[
\{x\text{-pair }0,\ x\text{-pair }2\},\qquad\{x\text{-pair }1,\ x\text{-pair }3\}.
\]
Let $d\in A$ have order four, so $d^2=c$. These paired pairs are exactly two
right cosets of $\langle d\rangle=A_2$ in the given side. In particular $\Theta$
is their union. The permutation $D=R(d)$ preserves $\Theta$, preserves each of
the two four-element cosets, and satisfies $D^2=x$.

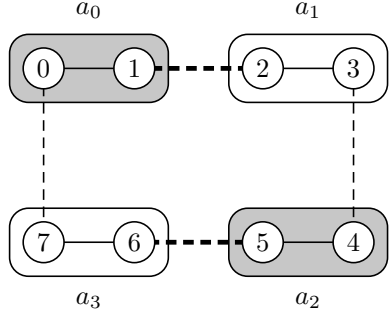
\begin{figure}[!htb]
\centering
\begin{tikzpicture}[x=1cm,y=1cm,font=\normalsize,
  every node/.style={inner sep=0pt,outer sep=0pt},
  vtx/.style={circle,draw=black,fill=white,line width=0.6pt,minimum size=5.4mm,inner sep=0pt},
  boxS/.style={rounded corners=2.5mm,draw=black,line width=0.6pt,fill=black!22},boxO/.style={rounded corners=2.5mm,draw=black,line width=0.6pt,fill=white},xedge/.style={line width=0.6pt,black},
  yheavy/.style={line width=1.8pt,black,dash pattern=on 4pt off 2.5pt},ylight/.style={line width=0.6pt,black,dash pattern=on 4pt off 2.5pt}]\draw[boxS] (-0.45,-0.45) rectangle (1.65,0.45);
\draw[boxO] (2.45,-0.45) rectangle (4.55,0.45);
\draw[boxS] (2.45,-2.75) rectangle (4.55,-1.85);
\draw[boxO] (-0.45,-2.75) rectangle (1.65,-1.85);
\draw[xedge] (0.00,0.00) -- (1.20,0.00);
\draw[xedge] (2.90,0.00) -- (4.10,0.00);
\draw[xedge] (4.10,-2.30) -- (2.90,-2.30);
\draw[xedge] (1.20,-2.30) -- (0.00,-2.30);
\draw[yheavy] (1.20,0.00) -- (2.90,0.00);
\draw[yheavy] (2.90,-2.30) -- (1.20,-2.30);
\draw[ylight] (0.00,0.00) -- (0.00,-2.30);
\draw[ylight] (4.10,0.00) -- (4.10,-2.30);
\node[vtx] at (0.00,0.00) {$0$};
\node[vtx] at (1.20,0.00) {$1$};
\node[vtx] at (2.90,0.00) {$2$};
\node[vtx] at (4.10,0.00) {$3$};
\node[vtx] at (4.10,-2.30) {$4$};
\node[vtx] at (2.90,-2.30) {$5$};
\node[vtx] at (1.20,-2.30) {$6$};
\node[vtx] at (0.00,-2.30) {$7$};
\node at (0.60,0.77) {$a_{0}$};
\node at (3.50,0.77) {$a_{1}$};
\node at (3.50,-3.07) {$a_{2}$};
\node at (0.60,-3.07) {$a_{3}$};
\end{tikzpicture}
\caption{The eight-point component of Proposition~\ref{prop:mixed8two}, with
  $x=(0\,1)(2\,3)(4\,5)(6\,7)$ (solid, inside the boxes) and $y=(0\,7)(1\,2)(3\,4)(5\,6)$ (dashed); the boxes
  are the four $x$-pairs, labelled by $a_0,\dots,a_3\in A/\langle c\rangle$. The two $y$-halves restrict to
  $(1\,2)(5\,6)$ (heavy) and $(3\,4)(7\,0)$ (light), and each joins adjacent boxes: the column of an
  opposite vertex fixed by the first has $f_0=f_1$, $f_2=f_3$, and one fixed by the second $f_1=f_2$,
  $f_3=f_0$, the two alternatives of \eqref{eq:star}. Lemma~\ref{lem:diag} pairs opposite boxes,
  $a_2=\iota a_0$, $a_3=\iota a_1$: the shaded boxes $\{0,1,4,5\}$ and the open boxes $\{2,3,6,7\}$ are the
  two cosets of $\langle d\rangle$, each preserved by $R(d)$. The proposition shows that no such component
  occurs.}
\label{fig:cosets}
\end{figure}

The root is determined up to orientation on each coset. Those two cosets are
$\{0,1,4,5\}$ and $\{2,3,6,7\}$. There are exactly four possible restrictions of
$D$:
\[
\begin{array}{ll}
D & yDyD^{-1}\\\hline
(0\,4\,1\,5)(2\,6\,3\,7) & (0\,4\,1)(2\,3\,7)\\
(0\,4\,1\,5)(2\,7\,3\,6) & (0\,1\,5)(2\,7\,6)\\
(0\,5\,1\,4)(2\,6\,3\,7) & (1\,4\,5)(2\,6\,3)\\
(0\,5\,1\,4)(2\,7\,3\,6) & (0\,5\,4)(3\,6\,7)
\end{array}
\]
All omitted points are fixed. Each product has order three on $\Theta$. The
list is complete: on each block the restriction is a four-cycle whose square is
the restriction of $x$, and each block has exactly two such four-cycles; so no
orientation of the actual translation square root is omitted.

But $y$ and $DyD^{-1}$ both belong to $\Ii$. Theorem~\ref{thm:rs4} says that
their product has fourth power equal to the identity globally. Its invariant
restriction to $\Theta$ therefore has order dividing four, so it cannot have order three.
Contradiction.

If instead the $y$-half is constant and the $x$-half alternating, interchange
$x,y$ and use the conjugate translation square root for $y$. Therefore every mixed
eight-point component is impossible. Since
Corollary~\ref{cor:period2} leaves,
on an eight-vertex cycle, only both masks constant or exactly one alternating, and exactly one
alternating is the mixed case just excluded,
both halves are constant on all eight-point components.
\end{proof}

\begin{lemma}[Pair-preserving criterion]\label{lem:transfer}
Assume \textup{(S)} and $H=\langle c\rangle$, let $y\in\Ii$ have complementary
halves $t_y,v_y$ with $C_y=\Fix(t_y)$, and let $\Sigma_y$ be the group of permutations of $\Qg$ mapping
each $y$-pair onto itself. If $g\in P$ preserves $C_y$ and acts
within individual $y$-pairs on $C_y$, then $g\in \Sigma_y$.
\end{lemma}

\begin{proof}
Every $y$-pair lies in $C_y$ or in $C_y^\complement$: if $h$ is the half of $y$ fixing a point $\gamma$ and $h'$ the
other, then $y=h'h$ and $h'$ is an involution, so
$h'(y\gamma)=h\gamma=\gamma\ne y\gamma$; so $y\gamma\notin\Fix(h')$,
that is, $y\gamma\in\Fix(h)$. For
$\beta\notin C_y$ and $\alpha\in C_y$,
\[
E(\alpha,g\beta)=E(g^{-1}\alpha,\beta)=E(\alpha,\beta).
\]
The last equality holds because $g$ maps the $y$-pair $\{\alpha,y\alpha\}\subseteq C_y$ into, hence onto,
itself (by hypothesis), so $g^{-1}\alpha\in\{\alpha,y\alpha\}$; and the two points of an $y$-pair in $C_y$
have the same neighbourhood in $C_y^\complement$ (Lemma~\ref{lem:cross-nbhd}(3)). Since $gC_y=C_y$ and $g$ is a
bijection, $g\beta\in C_y^\complement$. Thus $\{\beta,y\beta\}$ and $\{g\beta,yg\beta\}$ are $y$-pairs inside
$C_y^\complement$ with the same neighbourhood in $C_y$ (by the display), so they are the same pair, which puts
$g\beta$ in $\{\beta,y\beta\}$. Hence $g$ maps every $y$-pair into, and so, being injective, onto itself: on each
$y$-pair it is the identity or the swap, that is, $g\in \Sigma_y$.
\end{proof}

\begin{theorem}[Commuting full matchings]\label{thm:commute2}
Assume \textup{(S)} and $H=\langle c\rangle$. For all $x,y\in\Ii$, $[x,y]=1$.
Consequently $V=\langle\Ii\rangle$ is elementary abelian and normal in $P$.
\end{theorem}

The proof places $[x,t_y]$ and $[x,v_y]$ in the abelian group $\Sigma_y$ by Lemma~\ref{lem:transfer}, checking its hypothesis on each component of the two-matchings graph of $x$ and $y$; then $[x,y]\in\Sigma_y$ swaps no $y$-pair, so it is the identity.

\begin{proof}
Let $y$ have halves $t_y,v_y$ as above; then $y=t_yv_y$, and each half sends a point to itself or to its
$y$-partner, so $t_y,v_y$ lie in the group $\Sigma_y$ of Lemma~\ref{lem:transfer}, and each $y$-pair lies in $C_y$ or in
$C_y^\complement=\Fix(v_y)$; and $(v_y,t_y)$ is a complementary pair of $y$. Apply
Lemma~\ref{lem:transfer} to
$g=[x,t_y]$ and $g'=[x,v_y]$. By Theorem~\ref{thm:rs4}, every component of the
two-matchings graph of $x$ and $y$ has $2$, $4$ or $8$ vertices. On a component of
$2m$ vertices with $m\in\{1,2\}$, $q=xy$ has order $m$, so $q^2=1$
there; $x$ commutes with $y$ on components of
size two or four, and conjugation by $x$ takes a partial $y$-flip (a permutation of the component mapping each $y$-pair
onto itself) to another partial $y$-flip. Since $g=[x,t_y]=(t_y^{-1})^x\,t_y$ and $g'=(v_y^{-1})^x\,v_y$, on a component of
size two or four $(t_y^{-1})^x$ is a partial $y$-flip by the preceding sentence, and a product of
partial $y$-flips is one; thus $g,g'$ act within individual $y$-pairs there. On
an eight-point component the $y$-halves are constant
(Proposition~\ref{prop:mixed8two}). A constant mask makes $t_y$ equal to $1$ or to $y$ on the
component; as $y$ has no fixed point, $C_y=\Fix(t_y)$ then contains the component or misses it. If $C_y$
contains the component, $g$ fixes it pointwise, since the half fixing that side is the identity
there and the first matching preserves the component; if $C_y$ misses the component, then
$C_y^\complement=\Fix(v_y)$ contains it, $v_y$ is the identity there, and $g'$ fixes it pointwise. Take
$p\in C_y$. If $p$ lies on a component of size two or four, then $g(p)\in\{p,yp\}\subseteq C_y$. If $p$
lies on an eight-point component, that component meets $C_y$, so it lies in $C_y$ and $g(p)=p$. Hence
$g(C_y)\subseteq C_y$ with $g$ acting within $y$-pairs on $C_y$; since $C_y$ is a union of $y$-pairs,
$g(C_y)=C_y$. The same argument with $v_y$ and $C_y^\complement$ applies to $g'$. Lemma~\ref{lem:transfer} (applied
to $t_y$, and to $v_y$ with the roles of $C_y$ and $C_y^\complement$ exchanged) gives $g,g'\in \Sigma_y$.

Every element of $\Sigma_y$ fixes or swaps each $y$-pair independently, so $\Sigma_y\cong\prod_{\Delta}C_2$ over the $y$-pairs $\Delta$ (unrestricted product); in particular $\Sigma_y$ is
abelian, and it contains $v_y$. Hence
\[
[x,y]=[x,t_yv_y]=[x,v_y][x,t_y]^{v_y}\in \Sigma_y .
\]
A commutator $[x,y]$ cannot send $\alpha$ to $y\alpha$: otherwise $x^{-1}yx$
would fix $y\alpha$, which is impossible. Consequently $[x,y]=1$: as a flip of the pairs that swaps no pair, it is the
identity. The elements of $\Ii$ are commuting
involutions, so $V=\langle\Ii\rangle$ is elementary abelian; it is normal
because $\Ii$ is a conjugacy class.
\end{proof}

With $[x,y]=1$ for all $x,y\in\Ii$, what remains in this branch is to show that $P$ is imprimitive; the
affine lemma (Lemma~\ref{lem:affine}) does this in Theorem~\ref{thm:two}.

\subsection{The affine lemma}\label{ss:affine}

An \emph{affine action} of a group on an $\F_2$-vector space $V$ is an action in which each element
acts as a map $\xi\mapsto \Upsilon(\xi)+\upsilon$, with $\Upsilon$ a linear automorphism of $V$ and $\upsilon\in V$.

\begin{lemma}[Affine lemma]\label{lem:affine}
An affine action of $A=C_{2^\infty}$ on an infinite $\F_2$-vector space cannot
have finitely many orbits.
\end{lemma}

\begin{proof}
We linearize the action and use the units of the group algebra, which commute
with the action of the Pr\"ufer group. Homogenize on $\hat V=V\oplus\F_2$, with affine
hyperplane $\Phi=V\times\{1\}$: an element $a\in A$ acting on $V$ as $\xi\mapsto \Upsilon_a(\xi)+\upsilon_a$ acts on $\hat V$
by the linear map $(\xi,\vartheta)\mapsto(\Upsilon_a(\xi)+\vartheta\upsilon_a,\vartheta)$, which is the given action on
$\Phi$. Let the group algebra $\F_2[A]$ act on $\hat V$ by linear extension, and suppose that $A$ has
finitely many orbits on $\Phi$.

\emph{Claim~1. The unit group is divisible and preserves the affine hyperplane.} Every
augmentation-one element is a unit, since it lies in the group algebra of some finite cyclic
subgroup (Lemma~\ref{rem:truncated}(2)). A unit $u$ of $\F_2[A]$ is a sum of an odd number of elements of $A$ (its augmentation is
$1$), and each element of $A$ preserves the last coordinate of $\hat V$, so $u$ and $u^{-1}$ map $\Phi$ into
itself: the units preserve $\Phi$. They commute with $A$, as $\F_2[A]$ is commutative. The group
$\F_2[A]^*$ is abelian and is a $2$-group. Indeed every unit lies in $\F_2[A_n]\cong\F_2[\theta]/(\theta^{2^n})$ for some $n$ and has the form $1+\theta f$ by Lemma~\ref{rem:truncated}(1),(2), so $(1+\theta f)^{2^n}=1$. It also has surjective squaring: if $u=a_1+\dots+a_k$ with
distinct $a_i\in A$, choose $\tilde a_i\in A$ with $(\tilde a_i)^2=a_i$; in characteristic two
$(\tilde a_1+\dots+\tilde a_k)^2=a_1+\dots+a_k=u$, and $\tilde a_1+\dots+\tilde a_k$ has augmentation one. It is therefore
divisible and has no nontrivial finite image.

\emph{Claim~2. Finitely many $A$-orbits force its image at any point to be locally cyclic.} By
Claim~1 the units permute the finitely many $A$-orbits on $\Phi$. This gives a homomorphism from
$\F_2[A]^*$ to a finite symmetric group, which is trivial because $\F_2[A]^*$ has no nontrivial finite
image; so every unit maps every $A$-orbit to itself. Fix $\xi\in \Phi$ and put
$\mathfrak a=\Ann_{\F_2[A]}(\xi)$. For every unit $u$ there is $a\in A$ with $u\xi=a\xi$, that is,
$u-a\in\mathfrak a$. So the image of $\F_2[A]^*$ in $\F_2[A]/\mathfrak a$ equals the image $\bar A$ of
$A$. This image is locally cyclic, being a homomorphic image of $A$, so it has at most one involution.

\emph{Claim~3. The truncated polynomial subrings then force the original $A$-action to be trivial.}
For $d\in A$ put $\theta=\bar d-1$. The finite subring $\F_2[\bar d]$ of $\F_2[A]/\mathfrak a$ is a
quotient of $\F_2[\langle d\rangle]\cong\F_2[\theta]/(\theta^{|d|})$, so it is
$\F_2[\theta]/(\theta^e)$ for some $e\ge1$ (Lemma~\ref{rem:truncated}(1),~(3)). Its units lift to units of
$\F_2[A]$: each has the form $1+\theta f$ (Lemma~\ref{rem:truncated}(2)), and the same expression in
$\F_2[\langle d\rangle]$ has augmentation one, so it is a
unit of $\F_2[A]$ by Claim~1. By Claim~2 the units of $\F_2[\theta]/(\theta^e)$ therefore lie in
$\bar A$. If $e\ge4$, the two independent involutions
$1+\theta^{e-1}$, $1+\theta^{e-2}$ in its unit group (for instance
$(1+\theta^{e-2})^2=1+\theta^{2e-4}=1$, since $2e-4\ge e$)
contradict local cyclicity. Hence $e\le3$. Given $a\in A$, choose $d$ with
$d^4=a$. Then $\bar a-1=(\bar d-1)^4=\theta^4=0$. Thus every point $\xi$ of $\Phi$ is fixed
by $A$, contradicting finitely many orbits on the infinite set $\Phi$.
\end{proof}

\begin{remark}
For a locally cyclic group that is not a $2$-group the conclusion can fail: the multiplicative group of an algebraic closure $\overline{\F}_2$ of $\F_2$ is locally cyclic, the direct sum of the Pr\"ufer $p$-groups for the odd primes $p$, and it acts linearly on the infinite $\F_2$-vector space $\overline{\F}_2$ with the two orbits $\{0\}$ and $\overline{\F}_2^{*}$. Claim~1 of the proof needs $A$ to have $2$-power order: in $\F_2[\langle a\rangle]$ with $|a|=3$ the augmentation-one element $1+a+a^2$ is not a unit, since $(1+a+a^2)(1+a)=1+a^3=0$.
\end{remark}

\begin{theorem}[The case $|H|=2$]\label{thm:two}
Assume \textup{(S)} and $H=\langle c\rangle$. All members of $\Ii$ commute.
Consequently $P$ is imprimitive.
\end{theorem}

\begin{proof}
The first assertion is
Theorem~\ref{thm:commute2}. If $P$
were primitive, its
nontrivial normal subgroup $V=\langle\Ii\rangle$ would be transitive. An abelian
transitive faithful permutation group is regular (an element fixing one point fixes each translate of it, by commutativity, hence every point, by transitivity). Identify $\Qg$ with $V$ by $n\mapsto n(1)$. Since
$V$ is a regular normal subgroup of $P$, $P=V\rtimes P_1$, with $P_1$ the stabilizer of the vertex
$1$ acting on $V$ by conjugation, so $P$ acts on the $\F_2$-vector space $V$ by affine maps. Under
$n\mapsto n(1)$, the element $g\in P$ sends $n(1)$ to $(gng^{-1})(g(1))=((gng^{-1})m)(1)$, where $m\in V$
has $m(1)=g(1)$. So $g$ acts as $n\mapsto gng^{-1}+m$: an automorphism of $V$ followed by a
translation. The group
$V$ is infinite because $\Qg$ is. The subgroup $R(A)$ of $P$ is a copy of $A$:
$a\mapsto R(a)$ is a homomorphism, since $R(a)R(a')=R(a'a)=R(aa')$ as $A$ is abelian, and it is
injective, since $R(a)(1)=a$. So restricting the affine action of $P$ gives an affine action of
$A=C_{2^\infty}$ on the infinite $\F_2$-vector space $V$. For $d,a\in A$, $R(a)(d)=da$ and
$R(a)(db)=dba=da^{-1}b$. So the orbits of $R(A)$ on $\Qg$ are $A$ and $B=Ab$, and under the
identification $n\mapsto n(1)$ the action has exactly two orbits on $V$, contradicting
Lemma~\ref{lem:affine}. Hence $P$ is not primitive. It is transitive, since $R(\Qg)\le P$ and $R(p^{-1}q)$ sends $p$ to $q$. Being transitive and not
primitive, $P$ is imprimitive.
\end{proof}

\section{Proof of Theorem~\ref{thm:main}}\label{s:proof}

\begin{proof}[Proof of Theorem~\ref{thm:main}]
Suppose $G$ is a primitive permutation group containing a regular subgroup
isomorphic to $\Qg$, and that $G$ is not doubly transitive. A regular action of
$\Qg$ is permutation isomorphic to the right regular action (identify the point
set with $\Qg$ via the orbit map of one point; the left regular action is
conjugate to the right one by inversion), so we may assume
$R(\Qg)\le G\le\Sym(\Qg)$. By Proposition~\ref{prop:reduction}, the orbital
graph $\Gamma$ through $(1,c)$ is $\Cay(\Qg,X\cup Tb)$ with
$X=X^{-1}\subseteq A\setminus\{1\}$, $c\in X$, $cT=T$, $\varnothing\ne T\ne A$;
it is nonempty and noncomplete; $P=\Aut\Gamma$ is primitive; and $H=\Pi(T)$ is
a proper, finite, cyclic subgroup of $A$ containing $c$.

If $U=X\symd cX$ were finite, Lemma~\ref{lem:Ufinite} would make $P$
imprimitive. Thus $U$ is infinite, and (S) holds.

If $|H|=2$, Theorem~\ref{thm:two} contradicts primitivity. If $|H|\ge4$,
Corollary~\ref{cor:block} contradicts primitivity. Therefore no such $G$ exists: every primitive permutation group containing a
regular subgroup isomorphic to $\Qg$ is doubly transitive. \end{proof}

The proof thus reduces the theorem to a single orbital graph. By Proposition~\ref{prop:reduction}, a
primitive overgroup of $R(\Qg)$ that is not doubly transitive has the orbital graph $\Cay(\Qg,X\cup Tb)$
through $(1,c)$, and the rest of the argument concerns only the pair $(X,T)$ and the subgroup $H=\Pi(T)$
of $A$. Lemma~\ref{lem:Ufinite} excludes the case in which $X$ and $cX$ differ in finitely many elements.
In the remaining case, Theorem~\ref{thm:two} and Corollary~\ref{cor:block} contradict primitivity
according as $|H|=2$ or $|H|\ge4$; since $H$ is a finite subgroup of the Pr\"ufer $2$-group $A$
containing $c$, its order is a power of $2$ and at least $2$, so these are the only cases.
Section~\ref{s:boundary} turns to the boundary of the example. There the Pr\"ufer group $A$ cannot be
enlarged by an abelian factor: for every countably infinite abelian group $F$, the group
$\Dic(A\times F,c)$ is not a $B$-group (Corollary~\ref{cor:dic-product}), so Theorem~\ref{thm:main} does
not extend from $\Qg$ to these larger dicyclic groups.

\section{The boundary of the example}\label{s:boundary}

\begin{proposition}[Product-action obstruction]\label{prop:product-action}
If $K$ and $M$ are countably infinite groups, no semidirect product $K\rtimes M$ is a $B$-group.
\end{proposition}

The construction is that of Cameron and Johnson \cite[Proposition~4.1]{CameronJohnson1987},
stated there for direct products of two countably infinite groups. The finite analogue is Wielandt's: a direct product of at least two finite groups of the same
order at least $3$ is not a $B$-group \cite[Theorem~25.7]{WielandtFPG}, see also
\cite[Lemma~1.1]{PonomarenkoRyabov2024}. Cameron noted in 2001 that a direct product of two countable subgroups is not a $B$-group,
since it is contained in the primitive group $S_\omega\wr S_2$ in its product action
\cite[\S3]{Cameron2001RandomGraphRevisited}. We observe that the argument needs only that left multiplication act coordinatewise, which
holds for semidirect products. 

\begin{proof}
Write the elements of $K\rtimes M$ as pairs $(k,m)$, $k\in K$, $m\in M$, and let $m\mapsto\varphi_m$
be the action of $M$ on $K$ by automorphisms. The left multiplications
\[
(k,m)\cdot(p,q)=(k\varphi_m(p),\,mq)
\]
form the regular representation of $K\rtimes M$ on $K\times M$. Each acts on the two coordinates
separately, by the permutation $p\mapsto k\varphi_m(p)$ of $K$ and the permutation $q\mapsto mq$ of
$M$, so it lies in $\Sym(K)\times\Sym(M)$, acting coordinatewise on $K\times M$. Since $K$ and $M$
are both countably infinite, we may identify each of them with one countable set $\Omega$; then
$K\times M$ becomes $\Omega^2$, and every left multiplication lies in the base group of
\[
W=\Sym(\Omega)\wr C_2
\]
in its product action on $\Omega^2$ \cite[Section~2.7]{dixon_mortimer_1996}.

The group $W$ is a primitive permutation group of rank $3$
\cite[proof of Proposition~4.1]{CameronJohnson1987}; primitivity is also the case of
\cite[Lemma~2.7A]{dixon_mortimer_1996} in which $\Sym(\Omega)$ is primitive and not regular on
$\Omega$ and $C_2$ is transitive on the two coordinates. Rank $3$ means that a point stabilizer has
exactly two orbits on the remaining points, so $W$ is not doubly transitive, and $K\rtimes M$ is not a
$B$-group.
\end{proof}

For an abelian group $D$ with an involution $c$, we write
$\Dic(D,c)=\langle D,b\mid b^2=c,\ b^{-1}ab=a^{-1}\ (a\in D)\rangle$, so that $\Qg=\Dic(A,c)$. As for
$\Qg$ (Section~\ref{s:setup}, with $D$ in place of $A$), every element of $\Dic(D,c)$ is $d$ or $db$ for a unique $d\in D$, and
$ba=a^{-1}b$, $(db)^{-1}=cdb$ for $a,d\in D$; this is the standard description of a generalized
dicyclic group. We write $C_{3^\infty}$ for the Pr\"ufer $3$-group, the union of a chain of cyclic
groups of orders $3^n$; for an abelian group $\mathcal{K}$ we write $c$ also for the involution $(c,1)$ of
$A\times \mathcal{K}$.

\begin{corollary}\label{cor:dic-product}
For every countably infinite abelian group $\mathcal{K}$, $\Dic(A\times \mathcal{K},c)\cong \mathcal{K}\rtimes\Qg$ is not a
$B$-group. In particular $\Dic(A\times C_{3^\infty},c)$, which like $\Qg$ is countable, locally finite,
has an abelian subgroup of index two and a unique involution, is not a $B$-group.
\end{corollary}

In $\Dic(A\times \mathcal{K},c)$, $(ab)^2=a\cdot(bab^{-1})\cdot b^2=c$ for every $a\in A\times \mathcal{K}$, so
$(A\times \mathcal{K})b\subseteq\sqrt{c}=\{x:x^2=c\}$ and the group is $\sqrt{c}\cup\sqrt{c}\,b^{-1}$.

\begin{proof}
Write $G=\Dic(A\times \mathcal{K},c)$ and identify $A$ and $\mathcal{K}$ with $A\times1$ and $1\times \mathcal{K}$. The subgroup
$\mathcal{K}$ is normal: it is central in $A\times \mathcal{K}$ and inverted by $b$, and $A\times \mathcal{K}$ and $b$ generate $G$.
The relations of $\Qg$ hold in $G$, and by the normal form the elements $a$, $ab$ $(a\in A)$ of $G$ are
distinct, so $\langle A,b\rangle=A\cup Ab\cong\Qg$. This subgroup meets $\mathcal{K}$ trivially, as $\mathcal{K}\le A\times \mathcal{K}$
misses $Ab$ and $A\cap \mathcal{K}=1$, and $\mathcal{K}\langle A,b\rangle\supseteq \mathcal{K}A\cup \mathcal{K}(Ab)=G$. Hence
$G=\mathcal{K}\rtimes\langle A,b\rangle\cong \mathcal{K}\rtimes\Qg$, and Proposition~\ref{prop:product-action}, applied to
the countably infinite groups $\mathcal{K}$ and $\Qg$ (Section~\ref{s:setup}), shows that $G$ is not a $B$-group.

For $\mathcal{K}=C_{3^\infty}$ the four properties of the statement are standard for $\Dic(D,c)$ with $D$ a
countable abelian torsion group whose only involution is $c$, here $D=A\times C_{3^\infty}$ and, for
$\Qg$, $D=A$; for the involution, every element of $Db$ squares to $c\ne1$ (the computation before this
proof).
\end{proof}

We do not know any countable $B$-group other than $\Qg$.

\section*{Acknowledgements}
The author thanks P.~J.~Cameron for recording the problem in the Kourovka
Notebook \cite{Kourovka}.

\medskip\noindent\textbf{Use of AI tools.} The author used the AI systems ChatGPT (OpenAI) and
Claude (Anthropic) in this work, in developing and checking the arguments and in preparing the
manuscript. The author takes full responsibility for the content of this paper. No AI system is an
author.

\appendix

\section{Local path tables for Theorem~\ref{thm:twelve-words}}\label{app:paths}

The reader can check every row by the single rule
$e_j\sim o_i$ if and only if $\mu(i+j)=\sigma(i-j)$ (equation~\eqref{eq:grid}).

Each table is read by the rules (1)--(4) stated in the proof of Theorem~\ref{thm:twelve-words},
under ``Local forbidden-word lemma''.

This appendix displays the tables of the local lemmas for the twelve words $W_A$--$W_L$. In each subsection
below, no values of $\mu$ or $\sigma$ outside that subsection have been used. Programs that re-check these
tables are provided as ancillary files of the arXiv version.

\subsection*{$W_A$. Change word \texttt{1101}}
Normalized $\mu$-word: \texttt{001110}. Fix $\sigma(0)=1$. Base cells $(2,0)$ and $(4,0)$. Targets: $e_{1},e_{2}$.

{\small
\noindent\begin{tabular}{@{}p{2.2cm}p{6.7cm}p{3.0cm}@{}}
Assumption & Connecting path & Consequence\\ \hline
$\sigma(1)=1$ & $e_{1}\,\text{--}\,\allowbreak o_{2}\,\text{--}\,\allowbreak e_{2}$ & Therefore $\sigma(1)=0$.\\
$\sigma(2)=1$ & $e_{1}\,\text{--}\,\allowbreak o_{3}\,\text{--}\,\allowbreak e_{2}$ & Therefore $\sigma(2)=0$.\\
$\sigma(3)=0$ & $e_{1}\,\text{--}\,\allowbreak o_{1}\,\text{--}\,\allowbreak e_{-1}\,\text{--}\,\allowbreak o_{2}\,\text{--}\,\allowbreak e_{2}$ & Contradiction.\\
$\sigma(3)=1$ & $e_{1}\,\text{--}\,\allowbreak o_{1}\,\text{--}\,\allowbreak e_{0}\,\text{--}\,\allowbreak o_{3}\,\text{--}\,\allowbreak e_{2}$ & Contradiction.\\
\end{tabular}}

\subsection*{$W_B$. Change word \texttt{11001}}
Normalized $\mu$-word: \texttt{0011110}. Fix $\sigma(-1)=1$. Base cells $(3,-1)$ and $(5,-1)$. Targets: $e_{2},e_{3}$.

{\small
\noindent\begin{tabular}{@{}p{2.2cm}p{6.7cm}p{3.0cm}@{}}
Assumption & Connecting path & Consequence\\ \hline
$\sigma(0)=1$ & $e_{2}\,\text{--}\,\allowbreak o_{2}\,\text{--}\,\allowbreak e_{3}$ & Therefore $\sigma(0)=0$.\\
$\sigma(1)=1$ & $e_{2}\,\text{--}\,\allowbreak o_{3}\,\text{--}\,\allowbreak e_{3}$ & Therefore $\sigma(1)=0$.\\
$\sigma(2)=1$ & $e_{2}\,\text{--}\,\allowbreak o_{1}\,\text{--}\,\allowbreak e_{0}\,\text{--}\,\allowbreak o_{2}\,\text{--}\,\allowbreak e_{3}$ & Therefore $\sigma(2)=0$.\\
$\sigma(3)=0$ & $e_{2}\,\text{--}\,\allowbreak o_{1}\,\text{--}\,\allowbreak e_{-1}\,\text{--}\,\allowbreak o_{2}\,\text{--}\,\allowbreak e_{3}$ & Contradiction.\\
$\sigma(3)=1$ & $e_{2}\,\text{--}\,\allowbreak o_{1}\,\text{--}\,\allowbreak e_{0}\,\text{--}\,\allowbreak o_{3}\,\text{--}\,\allowbreak e_{3}$ & Contradiction.\\
\end{tabular}}

\subsection*{$W_C$. Change word \texttt{110000}}
Normalized $\mu$-word: \texttt{00111111}. Fix $\sigma(0)=1$. Base cells $(2,0)$ and $(4,0)$. Targets: $e_{1},e_{2}$.

{\small
\noindent\begin{tabular}{@{}p{2.2cm}p{6.7cm}p{3.0cm}@{}}
Assumption & Connecting path & Consequence\\ \hline
$\sigma(1)=1$ & $e_{1}\,\text{--}\,\allowbreak o_{2}\,\text{--}\,\allowbreak e_{2}$ & Therefore $\sigma(1)=0$.\\
$\sigma(2)=1$ & $e_{1}\,\text{--}\,\allowbreak o_{1}\,\text{--}\,\allowbreak e_{0}\,\text{--}\,\allowbreak o_{2}\,\text{--}\,\allowbreak e_{2}$ & Therefore $\sigma(2)=0$.\\
$\sigma(3)=0$ & $e_{1}\,\text{--}\,\allowbreak o_{1}\,\text{--}\,\allowbreak e_{-1}\,\text{--}\,\allowbreak o_{2}\,\text{--}\,\allowbreak e_{2}$ & Therefore $\sigma(3)=1$.\\
$\sigma(4)=1$ & $e_{1}\,\text{--}\,\allowbreak o_{5}\,\text{--}\,\allowbreak e_{2}$ & Therefore $\sigma(4)=0$.\\
$\sigma(5)=0$ & $e_{1}\,\text{--}\,\allowbreak o_{1}\,\text{--}\,\allowbreak e_{0}\,\text{--}\,\allowbreak o_{3}\,\text{--}\,\allowbreak e_{-2}\,\text{--}\,\allowbreak o_{2}\,\text{--}\,\allowbreak e_{2}$ & Contradiction.\\
$\sigma(5)=1$ & $e_{1}\,\text{--}\,\allowbreak o_{1}\,\text{--}\,\allowbreak e_{0}\,\text{--}\,\allowbreak o_{5}\,\text{--}\,\allowbreak e_{2}$ & Contradiction.\\
\end{tabular}}

\subsection*{$W_D$. Change word \texttt{110001}}
Normalized $\mu$-word: \texttt{00111110}. Fix $\sigma(0)=1$. Base cells $(2,0)$ and $(6,0)$. Targets: $e_{1},e_{3}$.

{\small
\noindent\begin{tabular}{@{}p{2.2cm}p{6.7cm}p{3.0cm}@{}}
Assumption & Connecting path & Consequence\\ \hline
$\sigma(1)=1$ & $e_{1}\,\text{--}\,\allowbreak o_{2}\,\text{--}\,\allowbreak e_{2}\,\text{--}\,\allowbreak o_{3}\,\text{--}\,\allowbreak e_{3}$ & Therefore $\sigma(1)=0$.\\
$\sigma(2)=1$ & $e_{1}\,\text{--}\,\allowbreak o_{3}\,\text{--}\,\allowbreak e_{3}$ & Therefore $\sigma(2)=0$.\\
$\sigma(3)=1$ & $e_{1}\,\text{--}\,\allowbreak o_{4}\,\text{--}\,\allowbreak e_{3}$ & Therefore $\sigma(3)=0$.\\
$\sigma(4)=1$ & $e_{1}\,\text{--}\,\allowbreak o_{1}\,\text{--}\,\allowbreak e_{-1}\,\text{--}\,\allowbreak o_{3}\,\text{--}\,\allowbreak e_{3}$ & Therefore $\sigma(4)=0$.\\
$\sigma(5)=0$ & $e_{1}\,\text{--}\,\allowbreak o_{1}\,\text{--}\,\allowbreak e_{-1}\,\text{--}\,\allowbreak o_{2}\,\text{--}\,\allowbreak e_{-2}\,\text{--}\,\allowbreak o_{3}\,\text{--}\,\allowbreak e_{3}$ & Contradiction.\\
$\sigma(5)=1$ & $e_{1}\,\text{--}\,\allowbreak o_{1}\,\text{--}\,\allowbreak e_{-1}\,\text{--}\,\allowbreak o_{4}\,\text{--}\,\allowbreak e_{3}$ & Contradiction.\\
\end{tabular}}

\subsection*{$W_E$. Change word \texttt{11111}}
Normalized $\mu$-word: \texttt{0011001}. Fix $\sigma(-1)=0$. Base cells $(1,-1)$ and $(5,-1)$. Targets: $e_{1},e_{3}$.

{\small
\noindent\begin{tabular}{@{}p{2.2cm}p{6.7cm}p{3.0cm}@{}}
Assumption & Connecting path & Consequence\\ \hline
$\sigma(1)=1$ & $e_{1}\,\text{--}\,\allowbreak o_{2}\,\text{--}\,\allowbreak e_{3}$ & Therefore $\sigma(1)=0$.\\
\end{tabular}}

\emph{Case $\sigma(0)=0$.} All preceding forced values remain in force.

{\small
\noindent\begin{tabular}{@{}p{2.2cm}p{6.7cm}p{3.0cm}@{}}
Assumption & Connecting path & Consequence\\ \hline
$\sigma(2)=0$ & $e_{1}\,\text{--}\,\allowbreak o_{3}\,\text{--}\,\allowbreak e_{2}\,\text{--}\,\allowbreak o_{2}\,\text{--}\,\allowbreak e_{3}$ & Contradiction.\\
$\sigma(2)=1$ & $e_{1}\,\text{--}\,\allowbreak o_{0}\,\text{--}\,\allowbreak e_{0}\,\text{--}\,\allowbreak o_{2}\,\text{--}\,\allowbreak e_{3}$ & Contradiction.\\
\end{tabular}}

\emph{Case $\sigma(0)=1$.} All preceding forced values remain in force.

{\small
\noindent\begin{tabular}{@{}p{2.2cm}p{6.7cm}p{3.0cm}@{}}
Assumption & Connecting path & Consequence\\ \hline
$\sigma(2)=0$ & $e_{1}\,\text{--}\,\allowbreak o_{3}\,\text{--}\,\allowbreak e_{3}$ & Contradiction.\\
$\sigma(2)=1$ & $e_{1}\,\text{--}\,\allowbreak o_{1}\,\text{--}\,\allowbreak e_{0}\,\text{--}\,\allowbreak o_{2}\,\text{--}\,\allowbreak e_{3}$ & Contradiction.\\
\end{tabular}}

\subsection*{$W_F$. Change word \texttt{10100}}
Normalized $\mu$-word: \texttt{0010000}. Fix $\sigma(-1)=0$. Base cells $(1,-1)$ and $(3,-1)$. Targets: $e_{1},e_{2}$.

{\small
\noindent\begin{tabular}{@{}p{2.2cm}p{6.7cm}p{3.0cm}@{}}
Assumption & Connecting path & Consequence\\ \hline
$\sigma(0)=1$ & $e_{1}\,\text{--}\,\allowbreak o_{1}\,\text{--}\,\allowbreak e_{2}$ & Therefore $\sigma(0)=0$.\\
$\sigma(2)=0$ & $e_{1}\,\text{--}\,\allowbreak o_{3}\,\text{--}\,\allowbreak e_{3}\,\text{--}\,\allowbreak o_{2}\,\text{--}\,\allowbreak e_{2}$ & Contradiction.\\
$\sigma(2)=1$ & $e_{1}\,\text{--}\,\allowbreak o_{0}\,\text{--}\,\allowbreak e_{0}\,\text{--}\,\allowbreak o_{2}\,\text{--}\,\allowbreak e_{2}$ & Contradiction.\\
\end{tabular}}

\subsection*{$W_G$. Change word \texttt{101010101}}
Normalized $\mu$-word: \texttt{00100010001}. Fix $\sigma(-1)=0$. Base cells $(1,-1)$ and $(9,-1)$. Targets: $e_{1},e_{5}$.

\emph{Case $\sigma(0)=0$.} All preceding forced values remain in force.

{\small
\noindent\begin{tabular}{@{}p{2.2cm}p{6.7cm}p{3.0cm}@{}}
Assumption & Connecting path & Consequence\\ \hline
$\sigma(2)=0$ & $e_{1}\,\text{--}\,\allowbreak o_{3}\,\text{--}\,\allowbreak e_{4}\,\text{--}\,\allowbreak o_{4}\,\text{--}\,\allowbreak e_{5}$ & Contradiction.\\
$\sigma(2)=1$ & $e_{1}\,\text{--}\,\allowbreak o_{0}\,\text{--}\,\allowbreak e_{0}\,\text{--}\,\allowbreak o_{2}\,\text{--}\,\allowbreak e_{2}\,\text{--}\,\allowbreak o_{4}\,\text{--}\,\allowbreak e_{5}$ & Contradiction.\\
\end{tabular}}

\emph{Case $\sigma(0)=1$.} All preceding forced values remain in force.

{\small
\noindent\begin{tabular}{@{}p{2.2cm}p{6.7cm}p{3.0cm}@{}}
Assumption & Connecting path & Consequence\\ \hline
$\sigma(2)=0$ & $e_{1}\,\text{--}\,\allowbreak o_{3}\,\text{--}\,\allowbreak e_{3}\,\text{--}\,\allowbreak o_{5}\,\text{--}\,\allowbreak e_{5}$ & Contradiction.\\
$\sigma(2)=1$ & $e_{1}\,\text{--}\,\allowbreak o_{1}\,\text{--}\,\allowbreak e_{2}\,\text{--}\,\allowbreak o_{4}\,\text{--}\,\allowbreak e_{5}$ & Contradiction.\\
\end{tabular}}

\subsection*{$W_H$. Change word \texttt{100100}}
Normalized $\mu$-word: \texttt{00101111}. Fix $\sigma(0)=1$. Base cells $(2,0)$ and $(4,0)$. Targets: $e_{1},e_{2}$.

{\small
\noindent\begin{tabular}{@{}p{2.2cm}p{6.7cm}p{3.0cm}@{}}
Assumption & Connecting path & Consequence\\ \hline
$\sigma(1)=0$ & $e_{1}\,\text{--}\,\allowbreak o_{2}\,\text{--}\,\allowbreak e_{2}$ & Therefore $\sigma(1)=1$.\\
$\sigma(2)=1$ & $e_{1}\,\text{--}\,\allowbreak o_{3}\,\text{--}\,\allowbreak e_{2}$ & Therefore $\sigma(2)=0$.\\
$\sigma(3)=0$ & $e_{1}\,\text{--}\,\allowbreak o_{1}\,\text{--}\,\allowbreak e_{-1}\,\text{--}\,\allowbreak o_{2}\,\text{--}\,\allowbreak e_{2}$ & Contradiction.\\
$\sigma(3)=1$ & $e_{1}\,\text{--}\,\allowbreak o_{4}\,\text{--}\,\allowbreak e_{3}\,\text{--}\,\allowbreak o_{3}\,\text{--}\,\allowbreak e_{2}$ & Contradiction.\\
\end{tabular}}

\subsection*{$W_I$. Change word \texttt{1000100}}
Normalized $\mu$-word: \texttt{001010000}. Fix $\sigma(1)=0$. Base cells $(1,1)$ and $(5,1)$. Targets: $e_{0},e_{2}$.

{\small
\noindent\begin{tabular}{@{}p{2.2cm}p{6.7cm}p{3.0cm}@{}}
Assumption & Connecting path & Consequence\\ \hline
$\sigma(0)=1$ & $e_{0}\,\text{--}\,\allowbreak o_{1}\,\text{--}\,\allowbreak e_{1}\,\text{--}\,\allowbreak o_{2}\,\text{--}\,\allowbreak e_{2}$ & Therefore $\sigma(0)=0$.\\
$\sigma(-4)=1$ & $e_{0}\,\text{--}\,\allowbreak o_{0}\,\text{--}\,\allowbreak e_{4}\,\text{--}\,\allowbreak o_{4}\,\text{--}\,\allowbreak e_{3}\,\text{--}\,\allowbreak o_{3}\,\text{--}\,\allowbreak e_{2}$ & Therefore $\sigma(-4)=0$.\\
$\sigma(-2)=0$ & $e_{0}\,\text{--}\,\allowbreak o_{1}\,\text{--}\,\allowbreak e_{5}\,\text{--}\,\allowbreak o_{3}\,\text{--}\,\allowbreak e_{2}$ & Contradiction.\\
$\sigma(-2)=1$ & $e_{0}\,\text{--}\,\allowbreak o_{0}\,\text{--}\,\allowbreak e_{2}$ & Contradiction.\\
\end{tabular}}

\subsection*{$W_J$. Change word \texttt{10000100}}
Normalized $\mu$-word: \texttt{0010101111}. Fix $\sigma(0)=1$. Base cells $(2,0)$ and $(4,0)$. Targets: $e_{1},e_{2}$.

{\small
\noindent\begin{tabular}{@{}p{2.2cm}p{6.7cm}p{3.0cm}@{}}
Assumption & Connecting path & Consequence\\ \hline
$\sigma(1)=0$ & $e_{1}\,\text{--}\,\allowbreak o_{2}\,\text{--}\,\allowbreak e_{2}$ & Therefore $\sigma(1)=1$.\\
$\sigma(2)=1$ & $e_{1}\,\text{--}\,\allowbreak o_{3}\,\text{--}\,\allowbreak e_{3}\,\text{--}\,\allowbreak o_{4}\,\text{--}\,\allowbreak e_{2}$ & Therefore $\sigma(2)=0$.\\
$\sigma(3)=0$ & $e_{1}\,\text{--}\,\allowbreak o_{1}\,\text{--}\,\allowbreak e_{-1}\,\text{--}\,\allowbreak o_{2}\,\text{--}\,\allowbreak e_{2}$ & Therefore $\sigma(3)=1$.\\
$\sigma(5)=0$ & $e_{1}\,\text{--}\,\allowbreak o_{1}\,\text{--}\,\allowbreak e_{-1}\,\text{--}\,\allowbreak o_{4}\,\text{--}\,\allowbreak e_{4}\,\text{--}\,\allowbreak o_{5}\,\text{--}\,\allowbreak e_{2}$ & Contradiction.\\
$\sigma(5)=1$ & $e_{1}\,\text{--}\,\allowbreak o_{6}\,\text{--}\,\allowbreak e_{3}\,\text{--}\,\allowbreak o_{4}\,\text{--}\,\allowbreak e_{4}\,\text{--}\,\allowbreak o_{5}\,\text{--}\,\allowbreak e_{2}$ & Contradiction.\\
\end{tabular}}

\subsection*{$W_K$. Change word \texttt{100000100}}
Normalized $\mu$-word: \texttt{00101010000}. Fix $\sigma(-1)=0$. Base cells $(5,-1)$ and $(7,-1)$. Targets: $e_{3},e_{4}$.

{\small
\noindent\begin{tabular}{@{}p{2.2cm}p{6.7cm}p{3.0cm}@{}}
Assumption & Connecting path & Consequence\\ \hline
$\sigma(0)=1$ & $e_{3}\,\text{--}\,\allowbreak o_{3}\,\text{--}\,\allowbreak e_{4}$ & Therefore $\sigma(0)=0$.\\
$\sigma(2)=0$ & $e_{3}\,\text{--}\,\allowbreak o_{5}\,\text{--}\,\allowbreak e_{5}\,\text{--}\,\allowbreak o_{4}\,\text{--}\,\allowbreak e_{4}$ & Contradiction.\\
$\sigma(2)=1$ & $e_{3}\,\text{--}\,\allowbreak o_{2}\,\text{--}\,\allowbreak e_{0}\,\text{--}\,\allowbreak o_{0}\,\text{--}\,\allowbreak e_{1}\,\text{--}\,\allowbreak o_{3}\,\text{--}\,\allowbreak e_{4}$ & Contradiction.\\
\end{tabular}}

\subsection*{$W_L$. Change word \texttt{1000000100}}
Normalized $\mu$-word: \texttt{001010101111}. Fix $\sigma(0)=1$. Base cells $(4,0)$ and $(8,0)$. Targets: $e_{2},e_{4}$.

{\small
\noindent\begin{tabular}{@{}p{2.2cm}p{6.7cm}p{3.0cm}@{}}
Assumption & Connecting path & Consequence\\ \hline
$\sigma(2)=1$ & $e_{2}\,\text{--}\,\allowbreak o_{4}\,\text{--}\,\allowbreak e_{4}$ & Therefore $\sigma(2)=0$.\\
$\sigma(3)=0$ & $e_{2}\,\text{--}\,\allowbreak o_{2}\,\text{--}\,\allowbreak e_{-1}\,\text{--}\,\allowbreak o_{1}\,\text{--}\,\allowbreak e_{1}\,\text{--}\,\allowbreak o_{4}\,\text{--}\,\allowbreak e_{4}$ & Therefore $\sigma(3)=1$.\\
$\sigma(4)=1$ & $e_{2}\,\text{--}\,\allowbreak o_{6}\,\text{--}\,\allowbreak e_{3}\,\text{--}\,\allowbreak o_{7}\,\text{--}\,\allowbreak e_{4}$ & Therefore $\sigma(4)=0$.\\
$\sigma(5)=0$ & $e_{2}\,\text{--}\,\allowbreak o_{2}\,\text{--}\,\allowbreak e_{-2}\,\text{--}\,\allowbreak o_{3}\,\text{--}\,\allowbreak e_{3}\,\text{--}\,\allowbreak o_{6}\,\text{--}\,\allowbreak e_{1}\,\text{--}\,\allowbreak o_{1}\,\text{--}\,\allowbreak e_{-1}\,\text{--}\,\allowbreak o_{4}\,\text{--}\,\allowbreak e_{4}$ & Contradiction.\\
$\sigma(5)=1$ & $e_{2}\,\text{--}\,\allowbreak o_{7}\,\text{--}\,\allowbreak e_{4}$ & Contradiction.\\
\end{tabular}}

\clearpage
\section{Notation and terminology}\label{app:notation}
\enlargethispage{\baselineskip}
\begingroup
\centering
\footnotesize
\setlength{\tabcolsep}{3pt}
\begin{tabular}{@{}p{0.17\textwidth}p{0.62\textwidth}p{0.17\textwidth}@{}}
\hline
Symbol or term & Meaning & Defined in\\\hline
$\Qg$, $b$ & the locally quaternion group $\Dic(A,c)$; $b^2=c$ & \S\ref{s:setup}\\
$A$, $c$ & $A=C_{2^\infty}$, the Pr\"ufer $2$-group; $c$ its unique involution & \S\ref{s:setup}\\
$A_n$ & the subgroup of $A$ of order $2^n$ ($n\ge0$) & \S\ref{s:setup}\\
$B$; sides & $B=Ab$; the cosets $A$ and $B$ are the two sides of $\Qg$ & \S\ref{s:setup}\\
$R(g)$, $R(\Qg)$ & the right translation $x\mapsto xg$ ($g\in\Qg$); $R(\Qg)=\{R(g):g\in\Qg\}$
  & \S\ref{ss:conv}\\
$\Pi(Y)$ & the period subgroup $\{h\in D:hY=Y\}$ of $Y\subseteq D$, $D$ abelian & \S\ref{ss:conv}\\
$G_{(C)}$, $G_{\{C\}}$ & the pointwise and the setwise stabilizer of $C$ in $G$ & \S\ref{ss:conv}\\
$\Fix(g)$; $g$-pair & the set of fixed points of $g$; for an involution $g$, a two-point orbit
  $\{p,g(p)\}$ of $g$ & \S\ref{ss:conv}\\
$\Alt(\Omega)$ & the group of even finitary permutations of a set $\Omega$
  & before Thm.~\ref{thm:wielandt}\\
$X$, $T$, $S$, $\Gamma$ & $X,T\subseteq A$, $S=X\cup Tb$, and $\Gamma=\Cay(\Qg,S)$
  & {\raggedright\S\ref{s:setup}, Lem.~\ref{lem:adj}\par}\\
$E(p,q)$ & $1$ if $p\sim q$ and $0$ otherwise & Lem.~\ref{lem:adj}\\
$H$, $\nu$ & $H=\Pi(T)$, a proper, finite, cyclic subgroup of $A$ containing $c$; the branch
  $|H|\ge4$ means $H=A_\nu$ with $\nu\ge2$
  & {\raggedright\S\ref{s:setup}, Def.~\ref{def:standing}\par}\\
$U$ & $U=X\symd cX$ & \S\ref{s:setup}\\
(S) & the standing hypothesis: $\Gamma=\Cay(\Qg,X\cup Tb)$ with $X=X^{-1}\subseteq A\setminus\{1\}$,
  $c\in X$, $cT=T$, $\varnothing\ne T\ne A$, and $U$ infinite & Def.~\ref{def:standing}\\
$P$, $\rho$, $\Ii$ & $P=\Aut\Gamma$; $\rho=R(c)$, the central translation by $c$; $\Ii=\rho^P$
  & {\raggedright\S\ref{s:setup}, Def.~\ref{def:standing}\par}\\
$t$, $v$ & the half-flips: $t(a)=a$, $t(ab)=cab$, $v(a)=ca$, $v(ab)=ab$ ($a\in A$); $tv=\rho$,
  $\Fix(t)=A$, $\Fix(v)=B$ & {\raggedright\eqref{eq:tv}, Lem.~\ref{lem:halfflips}\par}\\
(H); $\beta_h$ & the half-stabilizer rigidity: $\pst{A}$ equals $\{\beta_h:h\in H\}\cong H$, where
  $\beta_h(a)=a$ and $\beta_h(ab)=hab$ ($a\in A$) & Lem.~\ref{lem:H}\\
full matching & each $x\in\Ii$ & Def.~\ref{def:halves}\\
$(t_x,v_x)$; halves & a complementary pair for $x\in\Ii$: $(t_x,v_x,x)=(t^g,v^g,\rho^g)$ for some
  $g\in P$; its members are the halves of $x$, $t_x$ the chosen half and $v_x$ its complementary
  half & Def.~\ref{def:halves}\\
$t,v,s,w$ (\S\S\ref{s:big}--\ref{s:sync}) & $(t,v)=(t_x,v_x)$ and $(s,w)=(t_y,v_y)$ for $x,y\in\Ii$
  & \S\ref{s:big}\\
half-involution & a $P$-conjugate of $t$ & Def.~\ref{def:halves}\\
$C$, $C^\complement$ & the fixed set $C=\Fix(t')$ of a half-involution $t'$; $C^\complement$ its complement
  & Cor.~\ref{cor:H}\\
supported root & of a half-involution $t'$ with $C=\Fix(t')$: an element of $\pst{C}$ some power
  of which is $t'$; a supported square root is a $\zeta\in\pst{C}$ of order four with $\zeta^2=t'$
  & Cor.~\ref{cor:H}(3)\\
$\Opt(J)$ & the $2$-core of $J$, its largest normal $2$-subgroup & \S\ref{ss:finite}\\
{\raggedright two-matchings graph\par} & of $x,y\in\Ii$: the union of the two perfect matchings $\{p,x(p)\}$ and
  $\{p,y(p)\}$; its components are the two-matching components & Def.~\ref{def:twomatch}\\
$F$ in \S\ref{s:big} & $F=\langle t_x,v_x,t_y,v_y\rangle$, for $x,y\in\Ii$
  & Prop.~\ref{prop:two-core}(3)\\
$e_i$, $o_i$ & the vertices of an alternating two-matching component, labelled so that one
  full matching pairs $e_i$ with $o_i$ and the other pairs $o_i$ with $e_{i+1}$
  & Def.~\ref{def:twomatch}\\
masks $M$, $N$ & $M_i$ (or $M(i)$) is $1$ exactly when the chosen half of the first matching
  exchanges $e_i,o_i$; $N_i$ (or $N(i)$) is $1$ exactly when the chosen half of the second
  exchanges $o_i,e_{i+1}$ & Def.~\ref{def:twomatch}\\
{\raggedright constant, alternating, mixed\par} & a mask is constant if all its entries are equal and alternating if
  any two cyclically consecutive entries differ; a component is mixed if one of its masks is
  alternating and the other constant & Def.~\ref{def:twomatch}\\
$\langle t^P\rangle$ & the subgroup generated by the half-involutions & Cor.~\ref{cor:block}\\
$\hat V$ & in the proof of Lem.~\ref{lem:affine}, the space $V\oplus\F_2$ & Lem.~\ref{lem:affine}\\
$(t_x,v_x)$, $C_x$ & for $x\in\Ii$: a complementary pair of $x$, $C_x=\Fix(t_x)$,
  $C_x^\complement=\Fix(v_x)$; for $x=\rho$, $(t,v)$ and $C_\rho=A$; likewise $(t_y,v_y)$, $C_y$
  & \S\ref{s:two}\\
$G(\mu,\sigma)$; cell $(p,d)$ & for $\mu,\sigma:\Z\to\{0,1\}$, the bipartite graph with vertices $e_j$ and $o_i$
  in which $e_j\sim o_i$ if and only if $\mu(i+j)=\sigma(i-j)$; cell $(i{+}j,i{-}j)$
  & Def.~\ref{def:grid}\\
\end{tabular}
\par
\begin{tabular}{@{}p{0.17\textwidth}p{0.62\textwidth}p{0.17\textwidth}@{}}
robust pair & two cells $(p,d_0)$, $(p',d_0)$ of one row, $p\ne p'$, with $\mu(p)=\mu(p')=\varepsilon$,
  rook-connected in $G(\mu,\sigma)$ for every $\sigma$ with $\sigma(d_0)=\varepsilon$ & Def.~\ref{def:grid}\\
$\delta$ & the change sequence of $\mu:\Z\to\{0,1\}$: $\delta_i=\mu(i)+\mu(i+2)\pmod2$
  & before Tab.~\ref{tab:words}\\
$W_A$--$W_L$ & the twelve $\delta$-words of Table~\ref{tab:words}; each forces a robust pair
  & Tab.~\ref{tab:words}\\
target difference & of a robust pair $(p,d_0)$, $(p',d_0)$: $|j-j'|=|p-p'|/2$, where $j=(p-d_0)/2$ and
  $j'=(p'-d_0)/2$ & Thm.~\ref{thm:twelve-words}\\
\hline
\end{tabular}
\par
\endgroup

\bibliographystyle{amsplain-names}
\bibliography{refs}

\end{document}